\documentclass[12pt,reqno]{amsart}
\usepackage[margin=1in]{geometry}
\usepackage{graphicx}
\usepackage{amsmath,amsthm,amsfonts,mathrsfs,amssymb,float,color}
\usepackage{graphicx}
\usepackage{enumerate,enumitem}
\usepackage{textcomp}
\usepackage{verbatim}
\usepackage{ upgreek }
\usepackage[T1]{fontenc}
\usepackage[utf8]{inputenc}
\usepackage{hyperref}
\hypersetup{colorlinks=true,citecolor=red,linkcolor=blue}
\usepackage{epsfig,subfigure,fancybox,balance}
\usepackage{diagbox}
\usepackage{youngtab}
\usepackage[numbers]{natbib}
\usepackage{ dsfont }
\usepackage{enumitem}
\newcommand*{\rom}[1]{\expandafter\@slowromancap\romannumeral #1@}

\newcommand{\beq}[1]{\begin{equation} \label{#1}}
\newcommand{\eeq}{\end{equation}}
\newcommand{\bea}{\bed\begin{array}{rl}}
\newcommand{\eea}{\end{array}\eed}
\newcommand{\bed}{\begin{displaymath}}
\newcommand{\eed}{\end{displaymath}}
\newcommand{\barray}{\begin{array}{ll}}
\newcommand{\earray}{\end{array}}

\newcommand{\beqa}[1]{\begin{equation}\label{#1}\barray}
\newcommand{\eeqa}{\earray\end{equation}}

\numberwithin{equation}{section}

\newtheorem{thm}{Theorem}[section]

\definecolor{gray}{rgb}{0.75, 0.75, 0.75}

\newtheorem{theorem}[thm]{Theorem}
\newtheorem{corollary}[thm]{Corollary}
\newtheorem{lemma}[thm]{Lemma}
\newtheorem{proposition}[thm]{Proposition}
\newtheorem{remark}[thm]{Remark}
\newtheorem{example}[thm]{Example}

\allowdisplaybreaks

\newcommand{\R}{\mathbb R}
\newcommand{\E}{\mathbb E}
\newcommand{\Pp}{\mathbb P}
\newcommand{\Ff}{\mathcal F}
\newcommand{\Gg}{\mathcal G}
\newcommand{\Hh}{\mathcal H}
\newcommand{\Law}{\mathcal L}
\newcommand{\pr}{\operatorname{pr}}
\newcommand{\dd}{\,\mathrm d}
\newcommand{\ee}{\mathrm e}
\newcommand{\eps}{\varepsilon}
\newcommand{\norm}[1]{\lVert #1\rVert}
\newcommand{\abs}[1]{\lvert #1\rvert}
\newcommand{\ip}[2]{\langle #1,#2\rangle}
\newcommand{\Wass}{\mathcal W}
\newcommand{\Ptwo}{\mathcal P_2}
\newcommand{\slowspace}{\mathsf E}
\newcommand{\fastspace}{\mathsf H}

\begin{document}
\title{On Fast-Slow Mean-Field Forward-Backward Stochastic Systems}
\author{Yihao Sheng}
\address{Department of Mathematics, University of Connecticut,
Storrs, CT 06269, USA}
\email{yihao.sheng@uconn.edu}

\author{Fuke Wu}
\address{School of Mathematics and Statistics, Huazhong University of Science and Technology,
Wuhan 430074, China}
\email{wufuke@hust.edu.cn}

\author{George Yin}
\address{Department of Mathematics, University of Connecticut,
Storrs, CT 06269, USA}
\email{gyin@uconn.edu}
\date{}



\begin{abstract}
We establish an averaging principle for a class of multiscale mean-field forward-backward stochastic differential equations and identify several novel phenomena that are absent from classical fast-slow systems. In contrast with classical fast-slow systems, the effective dynamics cannot in general be obtained by simply freezing deterministic slow parameters and averaging against the invariant measure of the resulting fast equation. The appropriate averaging object is instead provided by a frozen fast dynamics in a random environment and its associated conditional invariant measures, which retain the coupling between the slow state and its distribution. The forward-backward structure creates a further obstruction: local averaging estimates need not remain stable when propagated over an arbitrary time horizon. We identify a uniform restart stability condition for the averaged system under which this obstruction can be overcome.
Using a joint lifted semigroup for the state-law dynamics, together with a two-scale discretization and a Gordin-type decomposition, we prove strong averaging for both the forward and backward components with optimal convergence rate \(O(\varepsilon^{1/2})\). As an application, we apply the general theory to a class of mean-field stochastic control problems and develop an efficient algorithm for solving such mean-field control problems.
\end{abstract}
\maketitle

\tableofcontents

\section{Introduction}\label{sec:introduction}

Stochastic models in engineering, operations research, economics,
biology, and the physical sciences often describe interactions
between processes evolving on different time scales; see, for example,
\cite{Khasminskii1968,Kushner1984,Kushner1990,sheng2026near,
ShengWuYinZong2026,YinZhang2013} and the references therein.
In a two-time-scale formulation, this separation is represented
by a small parameter \(\varepsilon>0\), with the fast component
evolving on the accelerated scale \(t/\varepsilon\) and the slow
component on the original scale \(t\).
A basic question is whether the slow dynamics admits an effective
description as \(\varepsilon\to0\).
The averaging principle addresses this question by incorporating
the influence of the fast motion through averaged coefficients.
Since the foundational work of Khasminskii \cite{Khasminskii1968},
a variety of methods have been developed for averaging forward
diffusion systems, including time discretization, asymptotic
expansions, and Poisson-equation techniques; see, for example,
\cite{KhasminskiiYin2004,LiuRocknerSunXie2020,
PardouxVeretennikov2001,RocknerXie2021}.
The present paper considers this problem for coupled mean-field
forward--backward stochastic systems, where the averaging procedure
must account for both the forward--backward interaction and the
dependence of the coefficients on the joint law of the slow and
fast variables.

The extension of the averaging principle from purely forward diffusions
to the coupled mean-field forward--backward systems studied here
involves two distinct difficulties. The first stems from the
forward--backward coupling and is already present in systems without
distributional dependence. A coupled forward--backward stochastic
differential equation is a stochastic two-point boundary-value problem:
a perturbation in the forward state propagates forward in time,
whereas a perturbation in the backward state propagates backward and,
through the coupling, feeds back into the forward dynamics.
Consequently, local solvability and stability do not by themselves
guarantee the corresponding properties on an arbitrary fixed time
horizon; see, for example,
\cite{Antonelli1993,Delarue2002,MaWuZhangZhang2015,
PengWu1999,Zhang2017}.

The implications of this difficulty for averaging were investigated
in our companion work \cite{SWY_GlobalFBSDE}, which provides the
starting point for the present analysis.
There we constructed a law-independent example in which the frozen
fast dynamics is exponentially stable and both the original and
formally averaged systems are uniquely solvable for the prescribed
boundary data, yet averaging fails at a critical time horizon.
To overcome this obstruction, we established a global-in-time
averaging principle under a uniformly Lipschitz decoupling-field
condition for the averaged system.
For the averaged mean-field system considered here, we formulate
the corresponding stability requirement directly as a uniform
$L^2$-Lipschitz estimate for the restart map.

Second, the fast coefficients in the present model depend on the joint law of the slow and fast variables. After the slow pair is frozen, it must therefore retain its randomness: replacing it by a deterministic point would destroy the joint law entering the coefficients. The resulting frozen motion is a McKean--Vlasov diffusion evolving in a random static environment. Its nonlinear law flow becomes Markovian only after the current law is incorporated into the state. This motivates the use of a lifted-semigroup argument, in the spirit of
\cite{HongHuLiuYang2026,RenRocknerWang2022}. Unlike the usual lift on \(\R^m\times\Ptwo(\R^m)\), the lift needed here contains the static environment, the fast state, and their current joint law.

The backward counterpart of McKean--Vlasov theory was initiated by
Buckdahn, Djehiche, Li, and Peng~\cite{BuckdahnDjehicheLiPeng2009}, who
introduced mean-field BSDEs through a limit of high-dimensional interacting
forward--backward systems.  The theory was subsequently developed for
coupled mean-field FBSDEs, with particular emphasis on well-posedness,
stability, and their connections with mean field control and mean field
games; see, for example,
\cite{BensoussanYamZhang2015,CarmonaDelarue2015,CarmonaDelarue2018}.
The present work lies at the intersection of these developments and the
classical theory of stochastic averaging: we establish a quantitative
averaging principle for coupled fast--slow mean-field FBSDEs on an
arbitrary fixed time horizon.

For quantitative averaging of forward fast--slow McKean--Vlasov equations,
R\"ockner, Sun, and Xie~\cite{RocknerSunXie2021} established a strong
convergence rate of order \(1/3\) by time discretization and improved it to
the optimal order \(1/2\) by means of a Poisson equation under additional
regularity assumptions. The present result achieves the same optimal
convergence order \(1/2\) for the full mean-field forward--backward
system, simultaneously controlling the slow forward state, the backward
state, and both martingale integrands. Moreover, the argument does not rely
on a Poisson equation and requires neither ellipticity nor spatial or Lions
differentiability of the coefficients, nor any time regularity of the
averaged martingale integrand. Instead, it is based on global Lipschitz and
dissipativity conditions together with the uniform restart stability of the
averaged forward--backward system.

Compared with the extensive theory for forward stochastic systems, averaging principles for backward and forward--backward equations are relatively scarce. Early results were obtained in connection with homogenization and singularly perturbed semilinear equations; see, for example,
\cite{BriandHu1999,PardouxVeretennikov1997}. More recent strong averaging results for multiscale FBSDEs include
\cite{JiLiu2026Potential,JiLiu2026Optimal,ShengWuYinZong2026,XuLian2023}.
In particular, Ji and Liu
\cite{JiLiu2026Potential,JiLiu2026Optimal} obtained the optimal order \(1/2\)
under substantially stronger structural or regularity assumptions. In
\cite{ShengWuYinZong2026}, we studied singularly perturbed FBSDEs arising from
stochastic control and established averaging results together with
applications to near-optimal control problems.

Against this background, the present framework provides a unified treatment of multiscale FBSDEs under substantially weaker structural requirements. In the law-independent case, it encompasses the principal settings considered in
\cite{JiLiu2026Potential,JiLiu2026Optimal,
ShengWuYinZong2026,XuLian2023},
while allowing genuine forward--backward coupling, dependence of the backward driver on the martingale integrand, and degenerate diffusions. Moreover, the optimal convergence order \(1/2\) is obtained without using a Poisson equation.

Our proof builds on the two-scale localization developed in
\cite{SWY_GlobalFBSDE}, with a refined analysis of the fast
fluctuations in the mean-field setting.
A macroscopic partition, chosen independently of \(\eps\), is used
to localize the forward--backward coupling.
Within each macroscopic interval, we introduce a microscopic mesh
of order \(\eps\) and establish a tracking estimate for the
conditional law of the fast component relative to the corresponding
invariant kernels.
This estimate controls the conditional bias, while a
time-inhomogeneous Gordin decomposition controls the remaining
conditionally centered fluctuations across successive microscopic
intervals.
The resulting local estimates are then propagated across the
macroscopic intervals using the restart stability of the averaged
system.

Several genuinely mean-field features require additional care in carrying out this argument. First, the contraction of the frozen equation takes place on a Wasserstein fiber with a fixed slow marginal, rather than on the full Wasserstein space. Second, the equilibrium is a joint law, whose disintegration gives the conditional invariant distributions entering the averaged coefficients. Third, the backward comparison on each macroscopic interval is itself a mean-field BSDE. Its generator is therefore treated as a Lipschitz operator on \(L^2\), rather than as a pointwise Lipschitz driver.

These abstract structures also have a concrete realization in mean-field control. We consider an affine controlled McKean--Vlasov equation with costs depending on the joint law of the state and the control, where the fast equation relaxes the lifted Pontryagin residual. Under suitable convexity and dissipativity conditions, the frozen fast dynamics converges to the optimal feedback control, and the averaged system is precisely the mean-field Pontryagin system.

For this class of problems, the structural assumptions required by the averaging theorem can moreover be verified directly. Strong monotonicity yields well-posedness of the original fast--slow system, while strong convexity of the restarted control problems gives the uniform restart estimate in {\rm(H3)}. The main theorem therefore applies, yielding the optimal convergence rate \(O(\eps^{1/2})\) and showing, in particular, that the fast variable provides a near-optimal control.

The paper is organized as follows. Section~\ref{sec:model} introduces the model, assumptions, and main result. Section~\ref{sec:frozen} develops the frozen McKean--Vlasov theory and the associated joint lifted semigroup. Section~\ref{sec:averaged-preliminaries} establishes the preliminary properties and estimates for the averaged system needed in the subsequent analysis. Section~\ref{sec:proof} contains the global-in-time averaging argument. Section~\ref{sec:consequences} presents several consequences, including a counterexample illustrating the necessity of restart stability and verifiable sufficient conditions for this property. Finally, Section~\ref{sec:mf-control} applies the theory to mean-field stochastic control, establishes well-posedness and averaging for the resulting fast--slow system, shows that the fast variable provides an approximation of the optimal control, and presents a nonquadratic numerical experiment.

\section{Assumptions and Main Result}\label{sec:model}

Let \(\slowspace:=\R^n\times\R^p\) and
\(\fastspace:=\R^m\), equipped with their Euclidean norms.  The norm of a
matrix is the Frobenius norm.  For a Polish space \(E\), let
\(\mathcal B(E)\) denote its Borel sigma-field and \(\mathcal P(E)\) the set
of Borel probability measures on \(E\).  For a Euclidean space \(E\),
\(\Ptwo(E)\) denotes the set of probability measures on \(E\) with finite
second moment, equipped with the \(2\)-Wasserstein distance
\[
 W_2^2(\theta,\theta')
 :=
 \inf_{\Gamma\in\mathcal C(\theta,\theta')}
 \int_{E\times E}\abs{x-x'}^2\,\Gamma(\dd x,\dd x'),
\]
where \(\mathcal C(\theta,\theta')\) is the set of couplings of \(\theta\)
and \(\theta'\).  For a Borel measurable map \(T:E\to F\) between Polish
spaces and \(\mu\in\mathcal P(E)\), we write \(T_\#\mu\) for the push-forward
of \(\mu\) under \(T\), that is,
\[
 (T_\#\mu)(A):=\mu(T^{-1}(A)),
 \qquad A\in\mathcal B(F).
\]
We write \(\Law(Z)\) for the law of a random variable \(Z\).  Whenever
square-integrable random variables \(Z,Z'\) are realized on the same
probability space, we use without further comment
\[
 W_2^2(\Law(Z),\Law(Z'))\le\E\abs{Z-Z'}^2.
\]
For \(\theta\in\Ptwo(E)\), we write
\[
 m_2(\theta):=\bigl(\int_E\abs z^2\,\theta(\dd z)\bigr)^{1/2}.
\]

Let \((\Omega,\Ff,\Pp)\) support a sigma-field \(\Gg_0\) and two independent
Brownian motions \(W^1\) and \(W^2\) of dimensions \(d_1\) and \(d_2\),
respectively.  We assume that $\mathcal G_0$, $W^1$,
and $W^2$ are mutually independent. Moreover, $\mathcal G_0$ is assumed
to be sufficiently rich in the following sense: for every $k\ge 1$ and
every $\mu\in\mathcal P_2(\mathbb R^k)$, there exists a
$\mathcal G_0$-measurable random variable
$\xi\in L^2(\Omega;\mathbb R^k)$ such that
$
 \mathcal L(\xi)=\mu.
$

Let \(\mathbb F=(\Ff_t)_{0\le t\le T}\) be the
usual augmentation of
\[
 \Ff_t^0
 :=
 \Gg_0\vee
 \sigma(W_r^1,W_r^2:0\le r\le t).
\]
For \(s\in[0,T]\), define the completed filtration
\[
 \Ff_t^{s,1}
 :=
 \bigl(
 \Ff_s\vee
 \sigma(W_r^1-W_s^1:s\le r\le t)
 \bigr)^{\Pp},
 \qquad s\le t\le T.
\]
For a Euclidean space \(E\), \(L^2(\Omega;E)\) denotes the space of
square-integrable \(\Ff\)-measurable \(E\)-valued random variables, with norm
\(\|Z\|_{L^2}:=(\E\abs Z^2)^{1/2}\).  For a sub-sigma-field
\(\mathcal G\subset\Ff\), \(L^2(\mathcal G;E)\) denotes its subspace of
\(\mathcal G\)-measurable random variables.  For a filtration
\(\mathbb G=(\mathcal G_t)\), an interval \([a,b]\subset[0,T]\), and a
Euclidean space \(E\), let \(\mathcal S_{\mathbb G}^2(a,b;E)\) denote the
space of continuous \(\mathbb G\)-adapted \(E\)-valued processes \(X\) such
that
$
 \E\big[\sup_{a\le t\le b}\abs{X_t}^2\big]<\infty,
$
and let \(\mathcal H_{\mathbb G}^2(a,b;E)\) denote the space of
\(\mathbb G\)-progressively measurable \(E\)-valued processes \(Z\) such that
\[
 \E\int_a^b\abs{Z_t}^2\,\dd t<\infty.
\]
When the filtration is clear from the context, we simply write
\(\mathcal S^2(a,b;E)\) and \(\mathcal H^2(a,b;E)\).

Let \(E\) and \(F\) be Polish spaces. A Markov kernel from \(E\) to \(F\)
is a map
\[
 K:E\times\mathcal B(F)\to[0,1]
\]
such that \(K(x,\cdot)\) is a probability measure on \(F\) for every
\(x\in E\), and \(x\mapsto K(x,A)\) is Borel measurable for every
\(A\in\mathcal B(F)\). For a Borel measurable function \(\varphi:F\to G\),
where \(G\) is a finite-dimensional Euclidean space, whenever the integral
below is well defined, we write
\[
 K\varphi(x):=\int_F \varphi(y)\,K(x,\dd y),
 \qquad x\in E.
\]
The dual action of \(K\) on probability measures is denoted by \(K^*\):
for \(\mu\in\mathcal P(E)\),
\[
 (K^*\mu)(A):=\int_E K(x,A)\,\mu(\dd x),
 \qquad A\in\mathcal B(F).
\]
Thus \(K^*\mu\in\mathcal P(F)\). Throughout the paper, \(C\) denotes a generic positive constant whose
value may change from line to line and which depends only on fixed model
parameters, but not on \(\eps\). Dependence on additional parameters is
indicated by subscripts; in particular, \(C_T\) may also depend on the
time horizon \(T\).

For \(\eps>0\), consider
\begin{equation}\label{eq:original-mf-system}
\left\{
\begin{aligned}
\dd X_t^\eps
 &=
 b\bigl(
 X_t^\eps,Y_t^\eps,U_t^\eps,
 \Law(X_t^\eps,U_t^\eps,Y_t^\eps)
 \bigr)\dd t+
 \sigma\bigl(
 X_t^\eps,U_t^\eps,\Law(X_t^\eps,U_t^\eps)
 \bigr)\dd W_t^1,
\\
\dd Y_t^\eps
 &=
 \frac1\eps
 h\bigl(
 X_t^\eps,Y_t^\eps,U_t^\eps,
 \Law(X_t^\eps,U_t^\eps,Y_t^\eps)
 \bigr)\dd t+
 \frac1{\sqrt\eps}
 g\bigl(
 X_t^\eps,Y_t^\eps,U_t^\eps,
 \Law(X_t^\eps,U_t^\eps,Y_t^\eps)
 \bigr)\dd W_t^2,
\\
\dd U_t^\eps
 &=
 -F\bigl(
 X_t^\eps,Y_t^\eps,U_t^\eps,V_t^\eps,
 \Law(X_t^\eps,U_t^\eps,Y_t^\eps)
 \bigr)\dd t
+
 V_t^{1,\eps}\dd W_t^1+V_t^{2,\eps}\dd W_t^2,
\\
X_0^\eps&=x,\qquad Y_0^\eps=y,\qquad
U_T^\eps
=
 \beta\bigl(X_T^\eps,\Law(X_T^\eps)\bigr).
\end{aligned}
\right.
\end{equation}

Here \(X^\eps\in\R^n\), \(Y^\eps\in\R^m\),
\(U^\eps\in\R^p\), and
$
 V^\eps=(V^{1,\eps},V^{2,\eps})
 \in\R^{p\times(d_1+d_2)}.
$
The coefficients are of the form
\[
\begin{aligned}
 (b,h,g)&:
 \R^n\times\R^m\times\R^p
 \times\Ptwo(\slowspace\times\fastspace)
 \longrightarrow
 \R^n\times\R^m\times\R^{m\times d_2},
\\
 F&:
 \R^n\times\R^m\times\R^p
 \times\R^{p\times(d_1+d_2)}
 \times\Ptwo(\slowspace\times\fastspace)
 \longrightarrow\R^p,
\\
 \sigma&:
 \R^n\times\R^p\times\Ptwo(\slowspace)
 \longrightarrow\R^{n\times d_1},
 \qquad
 \beta:
 \R^n\times\Ptwo(\R^n)\longrightarrow\R^p.
\end{aligned}
\]

To introduce the averaged system, fix \(\eta\in\Ptwo(\slowspace)\) and set
\[
 {\Ptwo}_\eta
 :=
 \left\{
 \rho\in\Ptwo(\slowspace\times\fastspace):
 (\pr_{\slowspace})_\#\rho=\eta
 \right\},
\]
where \(\pr_{\slowspace}:\slowspace\times\fastspace\to\slowspace\) is the
canonical projection.  Let
\(\Theta=(\Theta^x,\Theta^u)\) be a square-integrable
\(\slowspace\)-valued random variable with \(\Law(\Theta)=\eta\).  Given
\(\rho\in{\Ptwo}_\eta\), choose \((\Theta,Y_0)\) with
\(\Law(\Theta,Y_0)=\rho\) and consider the unscaled frozen fast dynamics
\[
 \dd Y_t
 =
 h\bigl(
 \Theta^x,Y_t,\Theta^u,\Law(\Theta,Y_t)
 \bigr)\dd t
 +
 g\bigl(
 \Theta^x,Y_t,\Theta^u,\Law(\Theta,Y_t)
 \bigr)\dd B_t,
 \qquad t\ge0,
\]
where \(B\) is a \(d_2\)-dimensional Brownian motion independent of
\((\Theta,Y_0)\), and the random environment \(\Theta\) remains frozen in
time.  Thus the joint law \(\Law(\Theta,Y_t)\) evolves inside the fixed
fiber \({\Ptwo}_\eta\).  Section~\ref{sec:frozen} develops this frozen
McKean--Vlasov dynamics in detail.  Under the assumptions below, it admits
a unique invariant joint law
\(\Pi^\eta\in{\Ptwo}_\eta\). We write its disintegration as
\[
 \Pi^\eta(\dd(x,u),\dd z)
 =
 \eta(\dd(x,u))\,\nu^{x,u;\eta}(\dd z),
\]
where \((x,u)\mapsto\nu^{x,u;\eta}\) is a Markov kernel from
\(\slowspace\) to \(\fastspace\), representing the conditional invariant
law of the fast variable given the frozen slow state \((x,u)\).

Define
\begin{align}
 \bar b(x,u,\eta)
 &:=
 \int_{\R^m}
 b(x,z,u,\Pi^\eta)\,\nu^{x,u;\eta}(\dd z),
 \label{eq:averaged-b-definition}\\
 \bar F(x,u,v,\eta)
 &:=
 \int_{\R^m}
 F\bigl(x,z,u,(v,0),\Pi^\eta\bigr)
 \,\nu^{x,u;\eta}(\dd z),
 \label{eq:averaged-F-definition}
\end{align}
where \(v\in\R^{p\times d_1}\) and \((v,0)\) denotes its concatenation with
the zero matrix in \(\R^{p\times d_2}\).

The averaged mean-field FBSDE is
\begin{equation}\label{eq:averaged-mf-system}
\left\{
\begin{aligned}
\dd\bar X_t
 &=
 \bar b\bigl(
 \bar X_t,\bar U_t,\Law(\bar X_t,\bar U_t)
 \bigr)\dd t+
 \sigma\bigl(
 \bar X_t,\bar U_t,\Law(\bar X_t,\bar U_t)
 \bigr)\dd W_t^1,
\\
\dd\bar U_t
 &=
 -\bar F\bigl(
 \bar X_t,\bar U_t,\bar V_t,\Law(\bar X_t,\bar U_t)
 \bigr)\dd t
 +\bar V_t\dd W_t^1,
\\
\bar X_0&=x,
\\
\bar U_T
 &=
 \beta\bigl(\bar X_T,\Law(\bar X_T)\bigr).
\end{aligned}
\right.
\end{equation}

We impose the following assumptions.  

\medskip
\noindent\textbf{(H1) Global Lipschitz continuity.}
There exists $L>0$ such that, for all $x,x'\in\mathbb R^n$, $y,y'\in\mathbb R^m$, $u,u'\in\mathbb R^p$, $v,v'\in\mathbb R^{p\times(d_1+d_2)}$, $\rho,\rho'\in\mathcal P_2(\slowspace\times \fastspace)$, $\eta,\eta'\in\mathcal P_2(\slowspace)$, and $\theta,\theta'\in\mathcal P_2(\mathbb R^n)$,
\begin{align}
 &\abs{b(x,y,u,\rho)-b(x',y',u',\rho')}
 +\abs{h(x,y,u,\rho)-h(x',y',u',\rho')}
 \notag\\
 &\quad+
 \abs{g(x,y,u,\rho)-g(x',y',u',\rho')}
 \le
 L\bigl(
 \abs{x-x'}+\abs{y-y'}+\abs{u-u'}+W_2(\rho,\rho')
 \bigr),
 \label{eq:H1-bhg}\\
 &\abs{F(x,y,u,v,\rho)-F(x',y',u',v',\rho')}
 \notag\\
 &\qquad\qquad\qquad\qquad\qquad \ \le
 L\bigl(
 \abs{x-x'}+\abs{y-y'}+\abs{u-u'}+\abs{v-v'}
 +W_2(\rho,\rho')
 \bigr),
 \label{eq:H1-F}\\
 &\abs{\sigma(x,u,\eta)-\sigma(x',u',\eta')}
 \le
 L\bigl(
 \abs{x-x'}+\abs{u-u'}+W_2(\eta,\eta')
 \bigr),
 \label{eq:H1-sigma}\\
 &\abs{\beta(x,\theta)-\beta(x',\theta')}
 \le
 L\bigl(\abs{x-x'}+W_2(\theta,\theta')\bigr).
 \label{eq:H1-beta}
\end{align}

Consequently, for some \(C>0\),
\begin{align}
 \abs{b(x,y,u,\rho)}^2
 +\abs{h(x,y,u,\rho)}^2
 +\abs{g(x,y,u,\rho)}^2
 &\le
 C\bigl(
 1+\abs x^2+\abs y^2+\abs u^2+m_2(\rho)^2
 \bigr),
 \label{eq:linear-growth-bhg}\\
 \abs{F(x,y,u,v,\rho)}^2
 &\le
 C\bigl(
 1+\abs x^2+\abs y^2+\abs u^2+\abs v^2+m_2(\rho)^2
 \bigr).
 \label{eq:linear-growth-F}
\end{align}
Moreover,
\[
 \abs{\sigma(x,u,\eta)}^2
 \le C\bigl(1+\abs x^2+\abs u^2+m_2(\eta)^2\bigr),
 \qquad
 \abs{\beta(x,\theta)}^2
 \le C\bigl(1+\abs x^2+m_2(\theta)^2\bigr).
\]
\medskip
\noindent\textbf{(H2) Joint dissipativity of the fast equation.}
There exist constants
\[
 \kappa>K_2\ge0,\qquad K_1\ge0,
\]
such that
\begin{align}
 &2\ip{y-y'}{
 h(x,y,u,\rho)-h(x',y',u',\rho')}
 +\abs{
 g(x,y,u,\rho)-g(x',y',u',\rho')
 }^2
 \notag\\
 &\quad\le
 -\kappa\abs{y-y'}^2
 +K_1\bigl(\abs{x-x'}^2+\abs{u-u'}^2\bigr)
 +K_2 W_2^2(\rho,\rho').
 \label{eq:joint-dissipativity}
\end{align}
Throughout the paper,
\begin{equation*}
 \alpha:=\kappa-K_2>0.
\end{equation*}

\medskip
\noindent\textbf{(H3) Uniform restart stability.}
For every sufficiently small \(\eps>0\),
\eqref{eq:original-mf-system} has a unique solution with
\[
 (X^\eps,Y^\eps,U^\eps)
 \in\mathcal S^2(0,T;\R^n\times\R^m\times\R^p),
 \qquad
 V^\eps\in\mathcal H^2(0,T;\R^{p\times(d_1+d_2)}).
\]
Moreover, for every \(s\in[0,T]\) and every
\(\xi\in L^2(\Ff_s;\R^n)\), the averaged equation in
\eqref{eq:averaged-mf-system}, started from
\(\bar X_s^{s,\xi}=\xi\), has a unique solution, adapted to
\((\Ff_t^{s,1})_{s\le t\le T}\), in
\[
 (\bar X^{s,\xi},\bar U^{s,\xi})
 \in\mathcal S^2(s,T;\R^n\times\R^p),
 \qquad
 \bar V^{s,\xi}\in\mathcal H^2(s,T;\R^{p\times d_1}).
\]
There exists
\(L_{\rm rst}>0\), independent of \(s\) and of the laws of the initial
random variables, such that for all
\(\xi,\xi'\in L^2(\Ff_s;\R^n)\),
\begin{equation}\label{eq:restart-lipschitz}
 \norm{\bar U_s^{s,\xi}-\bar U_s^{s,\xi'}}_{L^2}
 \le
 L_{\rm rst}\norm{\xi-\xi'}_{L^2}.
\end{equation}

\begin{remark}\label{rem:H3-operator}
{\rm
In a classical Markovian FBSDE, one typically has a deterministic
decoupling field of the form
$
 U_t=u(t,X_t).
$
For a mean-field FBSDE, the corresponding object generally depends
also on the law of the state, so that one expects
$
 U_t=u\bigl(t,X_t,\Law(X_t)\bigr).
$
Thus, after restarting at time \(s\) from a random variable \(\xi\),
the value at the restart time is in general of the form
$
 \bar U_s^{s,\xi}
 =
 u\bigl(s,\xi,\Law(\xi)\bigr),
$
rather than a function of \((s,\xi)\) alone.  Obtaining
\eqref{eq:restart-lipschitz} from such a representation would require
uniform Lipschitz regularity of \(u\) in both the state and measure
variables.  Instead of imposing the existence and regularity of such a
mean-field decoupling field, {\rm(H3)} directly assumes the required
\(L^2\)-Lipschitz property of the restart map
$
 \xi\longmapsto \bar U_s^{s,\xi}.
$ Later, we give several verifiable sufficient conditions for {\rm(H3)}.
}
\end{remark}

\begin{theorem}\label{thm:main}
Assume {\rm(H1)--(H3)}.  There exist \(C_T<\infty\) and
\(\eps_T>0\), independent of \(\eps\), such that, for
\(0<\eps\le\eps_T\),
\begin{align}
 &\E\sup_{0\le t\le T}\abs{X_t^\eps-\bar X_t}^2
 +\E\sup_{0\le t\le T}\abs{U_t^\eps-\bar U_t}^2
+
 \E\int_0^T
 \bigl(
 \abs{V_t^{1,\eps}-\bar V_t}^2
 +\abs{V_t^{2,\eps}}^2
 \bigr)\dd t
 \notag\\
 &\qquad\qquad\qquad\qquad\qquad\qquad\qquad\qquad\qquad\qquad\qquad\qquad\le
 C_T(1+\abs x^2+\abs y^2)\eps.
 \label{eq:main-estimate}
\end{align}
\end{theorem}

The following example is a simple adaptation of the classical
Ornstein--Uhlenbeck sharpness example in \cite{Liu2010}, and shows that
the rate in Theorem~\ref{thm:main} cannot be improved under
{\rm(H1)--(H3)}.

\begin{example}
\label{ex:optimality}
{\rm The exponent in Theorem~\ref{thm:main} is optimal within the class covered
by {\rm(H1)--(H3)}.  Indeed, consider the law-independent case with
\(n=p=d_1=d_2=1\) and \(m=2\), and let
\[
 b(x,y,u,\rho)=y_1,\qquad
 h(x,y,u,\rho)=-y,\qquad
 g(x,y,u,\rho)=
 \begin{pmatrix}
  \sqrt{2}\\
  0
 \end{pmatrix},
 \qquad
 \sigma=F=\beta=0.
\]
Starting from \(X_0^\eps=0\) and \(Y_0^\eps=(0,0)\), the corresponding
system takes the explicit form
\[
 \begin{cases}
  \dd X_t^\eps=Y_t^{\eps,1}\dd t,\\[1mm]
  \dd Y_t^{\eps,1}
  =-\dfrac{1}{\eps}Y_t^{\eps,1}\dd t
   +\sqrt{\dfrac{2}{\eps}}\,\dd W_t^2,\\[2mm]
  \dd Y_t^{\eps,2}
  =-\dfrac{1}{\eps}Y_t^{\eps,2}\dd t,\\[2mm]
  \dd U_t^\eps=V_t^\eps\dd W_t^1,\qquad U_T^\eps=0.
 \end{cases}
\]
Hence
\[
 Y_t^{\eps,2}=0,\qquad
 U_t^\eps=V_t^\eps=0, \qquad Y_t^{\eps,1}
 =
 \sqrt{\frac{2}{\eps}}
 \int_0^t \ee^{-(t-s)/\eps}\dd W_s^2.
\]
Moreover, {\rm(H2)} holds with
\(\kappa=2\) and \(K_1=K_2=0\), while
{\rm(H1)} and {\rm(H3)} are immediate.

The frozen fast equation has invariant distribution
$
 N(0,1)\otimes\delta_0.
$
Therefore the averaged drift satisfies
\[
 \bar b
 =
 \int_{\R^2} y_1
 \bigl(N(0,1)\otimes\delta_0\bigr)(\dd y)
 =0,
\]
and consequently
\[
 \bar X_t=0,\qquad
 \bar U_t=\bar V_t=0.
\]
On the other hand, integrating the first equation and using Fubini's
theorem gives
\begin{align*}
 X_T^\eps
 &=
 \int_0^T Y_t^{\eps,1}\dd t=
 \sqrt{\frac{2}{\eps}}
 \int_0^T\int_0^t
 \ee^{-(t-s)/\eps}\dd W_s^2\,\dd t=
 \sqrt{2\eps}\int_0^T
 \bigl(1-\ee^{-(T-s)/\eps}\bigr)\dd W_s^2.
\end{align*}
Thus, by It\^o's isometry,
\begin{align*}
 \E\abs{X_T^\eps-\bar X_T}^2
 &=
 2\eps\int_0^T
 \bigl(1-\ee^{-(T-s)/\eps}\bigr)^2\dd s=
 2T\eps-3\eps^2
 +4\eps^2\ee^{-T/\eps}
 -\eps^2\ee^{-2T/\eps}.
\end{align*}
}
\end{example}

\section{Frozen Mean-Field Dynamics and the Joint Lifted Semigroup}
\label{sec:frozen}

\subsection{The Wasserstein Fiber and Frozen Dynamics}

As will become clear below, when constructing the frozen dynamics, the
distribution of the slow pair is kept fixed, while the dissipativity
condition (H2) effectively compares fast variables evolving in the same frozen
environment. It is therefore natural to use a transport structure that
preserves the slow component and allows transportation only along the
corresponding fast fibers. To describe this structure, we use the fibered
Wasserstein distance studied by Peszek and Poyato~\cite{PeszekPoyato2023}.
More precisely, fix $\eta\in\Ptwo(\slowspace)$. Recall from
Section~\ref{sec:model} that ${\Ptwo}_\eta$ is the set of joint laws on
$\slowspace\times\fastspace$ with slow marginal $\eta$, and let
$\Theta=(\Theta^x,\Theta^u)$ be a square-integrable
$\slowspace$-valued random variable with
$
 \Law(\Theta)=\eta.
$
Here $\Theta$ represents a random realization of the
prescribed slow marginal.

By the disintegration theorem
\cite[Theorem~5.3.1]{AmbrosioGigliSavare},
every $\rho\in\mathcal P_{2,\eta}$ admits an $\eta$-a.e.\ uniquely
determined Borel probability kernel
$
(x,u)\longmapsto \rho^{x,u}\in\mathcal P(\fastspace)
$
such that
\[
\rho(d(x,u),dy)
=
\eta(d(x,u))\,\rho^{x,u}(dy).
\]
Moreover, since $\rho\in\mathcal P_2$, one has
$\rho^{x,u}\in\mathcal P_2(\fastspace)$ for
$\eta$-a.e.\ $(x,u)$. Equivalently, if
$
 \Law(\Theta,Z)=\rho,
$
then
\[
 \Law(Z\mid\Theta)
 =
 \rho^{\Theta^x,\Theta^u}
 \qquad\text{a.s.}
\]

For $\rho,\rho'\in{\Ptwo}_\eta$, set
\begin{equation}\label{eq:fiber-distance}
 \Wass_{2,\eta}^2(\rho,\rho')
 :=
 \int_{\slowspace}
 W_2^2(\rho^{x,u},(\rho')^{x,u})\,
 \eta(\dd(x,u)).
\end{equation}

The following properties can be obtained from the corresponding results
for fibered optimal transport in
\cite[Theorem~2.13, Remark~3.2, and Propositions~3.4, 3.7, and 3.10]{PeszekPoyato2023},
after identifying the base variable $\omega$ therein with $(x,u)$
and the transported variable with $y$.
For the reader's convenience, we include a brief proof.

\begin{lemma}\label{lem:fiber-space}
The value in \eqref{eq:fiber-distance} is independent of the chosen
versions of the disintegrations. Moreover,
\begin{align}
 \Wass_{2,\eta}^2(\rho,\rho')
 &=
 \inf
 \left\{
 \E\abs{Z-Z'}^2:
 \Law(\Theta,Z)=\rho, \ \ 
 \Law(\Theta,Z')=\rho'
 \right\},
 \label{eq:fiber-coupling}
\end{align}
where the infimum is taken over all square-integrable joint
realizations of $(\Theta,Z,Z')$ with
$\Law(\Theta)=\eta$.
In particular,
\begin{equation}\label{eq:ordinary-below-fiber}
 W_2(\rho,\rho')\le\Wass_{2,\eta}(\rho,\rho'),
\end{equation}
and $({\Ptwo}_\eta,\Wass_{2,\eta})$ is a complete metric space.
\end{lemma}

\begin{proof}
Consider any coupling in \eqref{eq:fiber-coupling}.  Conditional on
\(\Theta=(x,u)\), the conditional law of \((Z,Z')\) is a coupling of
\(\rho^{x,u}\) and \((\rho')^{x,u}\).  Hence
\[
 \E\bigl[\abs{Z-Z'}^2\mid\Theta\bigr]
 \ge
 W_2^2(\rho^{\Theta^x,\Theta^u},(\rho')^{\Theta^x,\Theta^u})
 \quad\text{a.s.}
\]
Taking expectations and then the infimum over all such couplings gives
\[
 \inf\bigl\{\E\abs{Z-Z'}^2:
 \Law(\Theta,Z)=\rho,\ \Law(\Theta,Z')=\rho'\bigr\}
 \ge \Wass_{2,\eta}^2(\rho,\rho').
\]

Conversely, since the conditional kernels are Borel and the state spaces
are Euclidean, a measurable selection result for optimal transport plans
\cite[Corollary~5.22]{Villani2009} yields a Borel family of optimal
couplings \(\Gamma^{x,u}\) of \(\rho^{x,u}\) and \((\rho')^{x,u}\).
On the rich probability space, realize a random triple
\((\Theta,Z,Z')\) with law
\[
 \eta(\dd(x,u))\,\Gamma^{x,u}(\dd z,\dd z').
\]
This is an admissible coupling in \eqref{eq:fiber-coupling}, with
\[
 \E\abs{Z-Z'}^2
 =
 \Wass_{2,\eta}^2(\rho,\rho').
\]
Together with the preceding inequality, this proves
\eqref{eq:fiber-coupling}.  Moreover, since
\((\Theta,Z)\) and \((\Theta,Z')\) form an ordinary coupling of
\(\rho\) and \(\rho'\), we obtain
\[
 W_2(\rho,\rho')\le \Wass_{2,\eta}(\rho,\rho'),
\]
which proves \eqref{eq:ordinary-below-fiber}.

It remains to prove completeness.  Let \((\rho_j)_{j\ge1}\) be \(\Wass_{2,\eta}\)-Cauchy.  By passing to a
subsequence if necessary, which we still denote by \((\rho_j)\), we may
assume that
\[
 \Wass_{2,\eta}(\rho_{j+1},\rho_j)\le 2^{-j}.
\]

Write \(\theta=(x,u)\).  For each \(j\), Corollary~5.22 of
\cite{Villani2009} yields a Borel family
$
 \Gamma_j^\theta\in
 \Pi\bigl(\rho_j^\theta,\rho_{j+1}^\theta\bigr)
$
of optimal couplings.  Define the probability measure
\[
 Q_j(\dd\theta,\dd z_j,\dd z_{j+1})
 :=
 \eta(\dd\theta)\,
 \Gamma_j^\theta(\dd z_j,\dd z_{j+1}).
\]
Then the \((\theta,z_j)\) and \((\theta,z_{j+1})\) marginals of \(Q_j\)
are \(\rho_j\) and \(\rho_{j+1}\), respectively.  Thus \(Q_j\) may be
viewed as the joint law of a triple
\((\Theta,Z_j,Z_{j+1})\), where \(\Law(\Theta)=\eta\).

Since \(Q_j\) and \(Q_{j+1}\) have the same
\((\theta,z_{j+1})\) marginal \(\rho_{j+1}\), the gluing lemma
\cite[pp.~11--12]{Villani2009} yields a joint law of
\((\Theta,Z_j,Z_{j+1},Z_{j+2})\) whose
\((\Theta,Z_j,Z_{j+1})\) marginal is \(Q_j\) and whose
\((\Theta,Z_{j+1},Z_{j+2})\) marginal is \(Q_{j+1}\).  Iterating this construction and using a standard
extension argument yields a common realization
\[
 (\Theta,Z_1,Z_2,\ldots)
\]
such that, for every \(j\),
\[
 \Law(\Theta,Z_j,Z_{j+1})=Q_j.
\]
Therefore,
\[
 \begin{aligned}
 \E\abs{Z_{j+1}-Z_j}^2
 &=
 \int
 W_2^2\bigl(\rho_j^{x,u},\rho_{j+1}^{x,u}\bigr)
 \,\eta(\dd(x,u))=
 \Wass_{2,\eta}^2(\rho_j,\rho_{j+1})
 \le 2^{-2j}.
 \end{aligned}
\]
For \(k>j\), Minkowski's inequality yields
\[
 \norm{Z_k-Z_j}_{L^2}
 \le\sum_{r=j}^{k-1}2^{-r}.
\]
Thus \((Z_j)_{j\geq 1}\) converges in \(L^2\) to a random variable \(Z\).  Put
\(\rho:=\Law(\Theta,Z)\).  Then
\(\rho\in{\Ptwo}_\eta\), and \eqref{eq:fiber-coupling} gives
\[
 \Wass_{2,\eta}(\rho_j,\rho)
 \le\norm{Z_j-Z}_{L^2}\longrightarrow0
\]
along the selected subsequence.  The Cauchy property then gives convergence
of the original sequence to \(\rho\).  This proves completeness.
\end{proof}

Let \((\Theta,Y_0)\) be square integrable with
\(\Law(\Theta)=\eta\) and
\(\Law(\Theta,Y_0)=\rho\in{\Ptwo}_\eta\).  As introduced in
Section~\ref{sec:model}, the unscaled frozen McKean--Vlasov equation is
\begin{equation}\label{eq:frozen-mf-equation}
\dd Y_t
=
h\bigl(
\Theta^x,Y_t,\Theta^u,\Law(\Theta,Y_t)
\bigr)\dd t
+
g\bigl(
\Theta^x,Y_t,\Theta^u,\Law(\Theta,Y_t)
\bigr)\dd B_t,
\qquad t\ge 0.
\end{equation}
where \(B\) is a \(d_2\)-dimensional Brownian motion independent of
\((\Theta,Y_0)\).  The environment \(\Theta\) is random but does not
evolve in time.

\begin{proposition}\label{prop:frozen-wellposed}
Under {\rm(H1)}, equation \eqref{eq:frozen-mf-equation} has a unique strong
solution satisfying
\begin{equation}\label{second-moment-est}
     \E\sup_{0\le r\le t}\abs{Y_r}^2<\infty
 \quad\text{for every }t<\infty.
\end{equation}
Moreover, for every \(t\ge0\), the joint law \(\Law(\Theta,Y_t)\) is
uniquely determined by the initial joint law
$
 \rho=\Law(\Theta,Y_0).
$
Consequently,
\begin{equation*}
 \mathscr S_t^\eta\rho
 :=
 \Law(\Theta,Y_t)
\end{equation*}
defines a nonlinear semigroup on \({\Ptwo}_\eta\), namely,
\begin{equation}\label{eq:nonlinear-semigroup}
 \mathscr S_{t+s}^\eta
 =
 \mathscr S_t^\eta\mathscr S_s^\eta,
 \qquad s,t\ge0.
\end{equation}
\end{proposition}

\begin{proof}
Since the measure argument in \eqref{eq:frozen-mf-equation} is the joint law
\(\Law(\Theta,Y_t)\), rather than merely \(\Law(Y_t)\), we give the
fixed point estimate explicitly.  On a fixed interval \([0,\delta]\), set
\(Y^0_t=Y_0\) and, recursively, let \(Y^{j+1}\) solve
\begin{align*}
 Y_t^{j+1}
 &=
 Y_0+
 \int_0^t
 h\bigl(
 \Theta^x,Y_r^j,\Theta^u,\Law(\Theta,Y_r^j)
 \bigr)\dd r+
 \int_0^t
 g\bigl(
 \Theta^x,Y_r^j,\Theta^u,\Law(\Theta,Y_r^j)
 \bigr)\dd B_r.
\end{align*}
The Burkholder--Davis--Gundy inequality, \eqref{eq:H1-bhg}, and the inequality
\[
 W_2^2\bigl(\Law(\Theta,Y_r^j),
 \Law(\Theta,Y_r^{j-1})\bigr)
 \le \E\abs{Y_r^j-Y_r^{j-1}}^2
\]
give
\[
 \E\sup_{0\le t\le\delta}\abs{Y_t^{j+1}-Y_t^j}^2
 \le
 C(\delta+\delta^2)
 \E\sup_{0\le t\le\delta}\abs{Y_t^j-Y_t^{j-1}}^2.
\]
Choose \(\delta>0\), depending only on the Lipschitz constant \(L\) in
{\rm(H1)}, such that
\(C(\delta+\delta^2)<1/2\).  The Picard sequence is Cauchy in
\(\mathcal S^2(0,\delta;\R^m)\), and its limit solves
\eqref{eq:frozen-mf-equation}.  The same estimate, applied to two solutions,
proves uniqueness on \([0,\delta]\).  Iterating this construction gives a
unique solution on every finite interval.  The linear-growth bounds
\eqref{eq:linear-growth-bhg}, It\^o's formula, and Gronwall's inequality
give \eqref{second-moment-est}.

It remains to show that the joint law of the solution is determined by the
initial joint law.  Let
\((\Theta^i,Y^i,B^i)\), \(i=1,2\), be two solutions, possibly defined on
different probability spaces, such that
$
 \Law(\Theta^1,Y_0^1)
 =
 \Law(\Theta^2,Y_0^2)
 =
 \rho,
$
and set
$
 \mu_t^i:=\Law(\Theta^i,Y_t^i).
$
On a common probability space, take
\((\bar\Theta,\bar Y_0)\) with law \(\rho\) and a Brownian motion
\(\bar B\) independent of \((\bar\Theta,\bar Y_0)\).  For \(i=1,2\), let
\(\bar Y^i\) solve
\[
 \dd \bar Y_t^i
 =
 h\bigl(
 \bar\Theta^x,\bar Y_t^i,\bar\Theta^u,\mu_t^i
 \bigr)\dd t
 +
 g\bigl(
 \bar\Theta^x,\bar Y_t^i,\bar\Theta^u,\mu_t^i
 \bigr)\dd \bar B_t,
 \qquad
 \bar Y_0^i=\bar Y_0.
\]
For each fixed deterministic measure flow \((\mu_t^i)_{t\ge0}\), consider
the augmented system
\[
 \dd\Theta_t=0,
 \qquad
 \dd Y_t
 =
 h\bigl(\Theta_t^x,Y_t,\Theta_t^u,\mu_t^i\bigr)\dd t
 +
 g\bigl(\Theta_t^x,Y_t,\Theta_t^u,\mu_t^i\bigr)\dd B_t.
\]
This is a standard SDE with the static component \(\Theta_t\equiv\Theta_0\).
Both \((\Theta^i,Y^i)\) and \((\bar\Theta,\bar Y^i)\) solve this system, and
their initial joint laws coincide:
$
 \Law(\Theta^i,Y_0^i)
 =
 \Law(\bar\Theta,\bar Y_0)
 =
 \rho.
$
Uniqueness in law for the augmented system therefore yields
\[
 \Law(\bar\Theta,\bar Y_t^i)
 =
 \Law(\Theta^i,Y_t^i)
 =
 \mu_t^i,
 \qquad t\ge0.
\]

Set \(\Delta_t=\bar Y_t^1-\bar Y_t^2\).  By the
Burkholder--Davis--Gundy inequality and \eqref{eq:H1-bhg}, for every
\(T<\infty\) and \(t\le T\),
\[
 \E\sup_{0\le s\le t}\abs{\Delta_s}^2
 \le
 C_T\int_0^t
 \left(
 \E\abs{\Delta_r}^2
 +
 W_2^2(\mu_r^1,\mu_r^2)
 \right)\dd r.
\]
Since
$
 \mu_r^i=\Law(\bar\Theta,\bar Y_r^i),
$
we have
\[
 W_2^2(\mu_r^1,\mu_r^2)
 \le
 \E\abs{\bar Y_r^1-\bar Y_r^2}^2.
\]
Hence
\[
 \E\sup_{0\le s\le t}\abs{\Delta_s}^2
 \le
 C_T\int_0^t
 \E\sup_{0\le u\le r}\abs{\Delta_u}^2\dd r.
\]
Gronwall's inequality implies \(\bar Y^1=\bar Y^2\) on \([0,T]\).
Therefore
\[
 \mu_t^1=\mu_t^2,
 \qquad 0\le t\le T.
\]
Since \(T\) is arbitrary, \(\Law(\Theta,Y_t)\) is uniquely determined by
\(\rho\) for every \(t\ge0\).
\end{proof}

\begin{lemma}\label{lem:frozen-lyapunov}
Under {\rm(H1)--(H2)}, there is a constant \(C>0\)
such that every solution of \eqref{eq:frozen-mf-equation} satisfies
\begin{equation}\label{eq:frozen-uniform-moment}
 \sup_{t\ge0}\E\abs{Y_t}^2
 \le
 C\bigl(
 1+\E\abs{\Theta^x}^2+\E\abs{\Theta^u}^2+\E\abs{Y_0}^2
 \bigr).
\end{equation}
\end{lemma}

\begin{proof}
Write
$
 h_0=h(0,0,0,\delta_{(0,0)}),
 g_0=g(0,0,0,\delta_{(0,0)}),
$
where the Dirac mass is on
\(\slowspace\times\fastspace\).  Applying
\eqref{eq:joint-dissipativity} yields
\begin{align}
 &2\ip y{h(x,y,u,\rho)-h_0}
 +\abs{g(x,y,u,\rho)-g_0}^2
 \le
 -\kappa\abs y^2
 +K_1(\abs x^2+\abs u^2)
 +K_2 W_2^2(\rho,\delta_{(0,0)}).
 \label{eq:lyapunov-first}
\end{align}
For any \(\vartheta>0\), by Young's inequality and \eqref{eq:H1-bhg}, we have
\(2\ip y{h_0}\le \vartheta\abs y^2+C_\vartheta\) and
\[
 2\ip{g(x,y,u,\rho)-g_0}{g_0}
 \le \vartheta\abs{g(x,y,u,\rho)-g_0}^2+C_\vartheta
 \le C\vartheta\bigl(
 \abs x^2+\abs y^2+\abs u^2
 +W_2^2(\rho,\delta_{(0,0)})
 \bigr)+C_\vartheta.
\]
Since
\[
 \abs{g(x,y,u,\rho)}^2
 =
 \abs{g(x,y,u,\rho)-g_0}^2
 +2\ip{g(x,y,u,\rho)-g_0}{g_0}
 +\abs{g_0}^2,
\]
combining these estimates with \eqref{eq:lyapunov-first} gives
\begin{equation}
 2\ip y{h(x,y,u,\rho)}
 +\abs{g(x,y,u,\rho)}^2
 \le
 -(\kappa-C\vartheta)\abs y^2
 +C_\vartheta(1+\abs x^2+\abs u^2)
 +(K_2+C\vartheta)W_2^2(\rho,\delta_{(0,0)}).
 \label{eq:lyapunov-second}
\end{equation}
For
\(\rho=\Law(\Theta,Y_t)\),
\[
 W_2^2(\rho,\delta_{(0,0)})
 =
 \E\abs{\Theta^x}^2+\E\abs{\Theta^u}^2+\E\abs{Y_t}^2.
\]
Choose \(\vartheta>0\) sufficiently small so that
\(\kappa-K_2-C\vartheta\ge\alpha/2\), and set \(c_0:=\alpha/2\).
For
\[
 A_s:=
 2\ip{Y_s}{
 h(\Theta^x,Y_s,\Theta^u,\Law(\Theta,Y_s))}
 +
 \abs{g(\Theta^x,Y_s,\Theta^u,\Law(\Theta,Y_s))}^2,
\]
Taking expectations in \eqref{eq:lyapunov-second} gives
\begin{equation}\label{eqpre}
     \E A_s
 \le
 -c_0\E\abs{Y_s}^2
 +C\bigl(
 1+\E\abs{\Theta^x}^2+\E\abs{\Theta^u}^2
 \bigr).
\end{equation}

Now define
$
 \tau_N:=\inf\{r\ge0:\abs{Y_r}\ge N\}.
$
By It\^o's formula,
\[
 \E\abs{Y_{t\wedge\tau_N}}^2
 =
 \E\abs{Y_0}^2
 +
 \E\int_0^{t\wedge\tau_N} A_s\,\dd s,
\]
since the stopped stochastic integral has zero expectation.
By Proposition~\ref{prop:frozen-wellposed} and the linear-growth bound
\eqref{eq:linear-growth-bhg}, dominated convergence allows
\(N\to\infty\), yielding
\[
 \E\abs{Y_t}^2
 =
 \E\abs{Y_0}^2
 +
 \int_0^t\E A_s\,\dd s.
\]
Using \eqref{eqpre}, we obtain
\[
 \E\abs{Y_t}^2
 \le
 \E\abs{Y_0}^2
 -c_0\int_0^t\E\abs{Y_s}^2\,\dd s
 +Ct\bigl(
 1+\E\abs{\Theta^x}^2+\E\abs{\Theta^u}^2
 \bigr).
\]
Consequently,
\[
 \frac{\dd}{\dd t}\E\abs{Y_t}^2
 \le
 -c_0\E\abs{Y_t}^2
 +C\bigl(
 1+\E\abs{\Theta^x}^2+\E\abs{\Theta^u}^2
 \bigr).
\]
Gronwall's inequality then gives
\begin{equation}\label{eq:frozen-lyapunov-solution}
 \E\abs{Y_t}^2
 \le
 \ee^{-c_0t}\E\abs{Y_0}^2
 +C\bigl(1+\E\abs{\Theta^x}^2+\E\abs{\Theta^u}^2\bigr).
\end{equation}
Taking the supremum in \(t\) gives \eqref{eq:frozen-uniform-moment}.
\end{proof}

\begin{lemma}\label{lem:fiber-contraction}
Under {\rm(H1)--(H2)}, for every
\(\rho,\rho'\in{\Ptwo}_\eta\),
\begin{equation}\label{eq:fiber-contraction}
 \Wass_{2,\eta}^2(
 \mathscr S_t^\eta\rho,\mathscr S_t^\eta\rho')
 \le
 \ee^{-\alpha t}\Wass_{2,\eta}^2(\rho,\rho'),
 \qquad t\ge0.
\end{equation}
\end{lemma}

\begin{proof}
By Lemma~\ref{lem:fiber-space}, for any \(r>0\) we may realize
\((\Theta,Y_0,Y_0')\) so that the slow pair is the same,
\[
 \Law(\Theta,Y_0)=\rho,\qquad
 \Law(\Theta,Y_0')=\rho',
\]
and
\[
 \E\abs{Y_0-Y_0'}^2
 \le\Wass_{2,\eta}^2(\rho,\rho')+r.
\]
Drive the two frozen equations with the same Brownian motion.  Set
\(D_t=Y_t-Y_t'\).  It\^o's formula and
\eqref{eq:joint-dissipativity} give
\begin{align*}
 \frac{\dd}{\dd t}\E\abs{D_t}^2
 &\le
 -\kappa\E\abs{D_t}^2
 +K_2
 W_2^2(
 \mathscr S_t^\eta\rho,\mathscr S_t^\eta\rho').
\end{align*}
Since the two frozen equations use the same environment \(\Theta\),
the joint law of
$
 \bigl((\Theta,Y_t),(\Theta,Y_t')\bigr)
$
is a coupling of
\(\mathscr S_t^\eta\rho\) and
\(\mathscr S_t^\eta\rho'\).  Therefore,
\[
 W_2^2(
 \mathscr S_t^\eta\rho,\mathscr S_t^\eta\rho')
 \le
 \E\abs{(\Theta,Y_t)-(\Theta,Y_t')}^2
 =
 \E\abs{D_t}^2.
\]
Thus
\[
 \frac{\dd}{\dd t}\E\abs{D_t}^2
 \le-\alpha\E\abs{D_t}^2,
\]
and Gronwall's inequality gives
\[
 \E\abs{D_t}^2
 \le\ee^{-\alpha t}
 \bigl(\Wass_{2,\eta}^2(\rho,\rho')+r\bigr).
\]
Moreover, since
$
 \Law(\Theta,Y_t)=\mathscr S_t^\eta\rho,
 \Law(\Theta,Y_t')=\mathscr S_t^\eta\rho',
$
by
\eqref{eq:fiber-coupling}, we have
$
 \Wass_{2,\eta}^2(
 \mathscr S_t^\eta\rho,\mathscr S_t^\eta\rho')
 \le \E\abs{D_t}^2.
$
Letting \(r\downarrow0\) proves \eqref{eq:fiber-contraction}.
\end{proof}

\begin{remark}
{\rm The fiber metric is essential here because the dissipativity in
{\rm(H2)} acts on the fast variables under the same frozen environment.
Indeed, the coupling representation in
\eqref{eq:fiber-coupling} allows the two dynamics to be realized with a
common environment \(\Theta\), so that the slow-variable contribution in
{\rm(H2)} vanishes.  In contrast, an ordinary \(W_2\)-coupling of two
joint laws in \({\Ptwo}_\eta\) may transport mass between different
environment fibers, and therefore need not preserve this contraction
mechanism.  Thus \(\Wass_{2,\eta}\) is the natural metric for the frozen
dynamics.}
\end{remark}

\subsection{Invariant Laws and Stability}

In this subsection, we study the long-time behavior of the frozen dynamics. We first establish the existence and uniqueness of its invariant law, together with exponential convergence in the Wasserstein distance and suitable moment bounds, and then derive the stability properties of the invariant law needed in the subsequent averaging analysis.

\begin{proposition}\label{prop:joint-invariant}
Under {\rm(H1)--(H2)}, for every
\(\eta\in\Ptwo(\slowspace)\), there exists a unique
\(\Pi^\eta\in{\Ptwo}_\eta\) such that
\begin{equation}\label{eq:joint-invariance}
 \mathscr S_t^\eta\Pi^\eta=\Pi^\eta,
 \qquad t\ge0.
\end{equation}
Moreover, for every \(\rho\in{\Ptwo}_\eta\),
\begin{align}
 \Wass_{2,\eta}^2(
 \mathscr S_t^\eta\rho,\Pi^\eta)
 &\le
 \ee^{-\alpha t}\Wass_{2,\eta}^2(\rho,\Pi^\eta),
 \label{eq:joint-invariant-convergence}\\
 \int_{\slowspace\times\fastspace}\abs y^2
 \,\Pi^\eta(\dd(x,u),\dd y)
 &\le C\bigl(1+m_2(\eta)^2\bigr).
 \label{eq:joint-invariant-moment}
\end{align}
\end{proposition}

\begin{proof}
Fix \(t_0>0\).  Lemma~\ref{lem:fiber-contraction} shows that
\(\mathscr S_{t_0}^\eta\) is a strict contraction on the complete metric
space \(({\Ptwo}_\eta,\Wass_{2,\eta})\), with contraction constant
\(\ee^{-\alpha t_0/2}<1\).  Banach's fixed-point theorem gives a unique
\(\Pi^\eta\in{\Ptwo}_\eta\) satisfying
\(\mathscr S_{t_0}^\eta\Pi^\eta=\Pi^\eta\). For \(s\ge0\), the nonlinear semigroup property \eqref{eq:nonlinear-semigroup} gives
\[
 \mathscr S_{t_0}^\eta
 \bigl(\mathscr S_s^\eta\Pi^\eta\bigr)
 =
 \mathscr S_s^\eta
 \bigl(\mathscr S_{t_0}^\eta\Pi^\eta\bigr)
 =
 \mathscr S_s^\eta\Pi^\eta.
\]
Thus \(\mathscr S_s^\eta\Pi^\eta\) is another fixed point of
\(\mathscr S_{t_0}^\eta\), and uniqueness implies
\(\mathscr S_s^\eta\Pi^\eta=\Pi^\eta\).  This proves
\eqref{eq:joint-invariance}.  Taking \(\rho'=\Pi^\eta\) in
\eqref{eq:fiber-contraction} proves
\eqref{eq:joint-invariant-convergence}.

Finally, start \eqref{eq:frozen-mf-equation} from \(\Pi^\eta\).
Its law is stationary.  In particular,
$
 \Law(\Theta,Y_t)=\Pi^\eta.
$
Hence the law of \(Y_t\) is the \(y\)-marginal of \(\Pi^\eta\), and
therefore
$
 \E\abs{Y_t}^2
 =
 \int\abs y^2\,\Pi^\eta(\dd(x,u),\dd y).
$
Applying
\eqref{eq:frozen-lyapunov-solution} at time \(1\) gives
\[
 \int\abs y^2\,\Pi^\eta(\dd(x,u),\dd y)
 \le
 \ee^{-c_0}
 \int\abs y^2\,\Pi^\eta(\dd(x,u),\dd y)
 +C(1+m_2(\eta)^2).
\]
Move the first term on the right to the left to obtain
\eqref{eq:joint-invariant-moment}.
\end{proof}

For fixed \(\eta\in\Ptwo(\slowspace)\) and
\((x,u)\in\slowspace\), consider the following non-McKean--Vlasov SDE:
\begin{equation}\label{eq:conditional-frozen-sde}
 \dd Z_t
 =
 h(x,Z_t,u,\Pi^\eta)\dd t
 +g(x,Z_t,u,\Pi^\eta)\dd B_t.
\end{equation}

\begin{lemma}\label{lem:conditional-invariant}
Under {\rm(H1)--(H2)}, for every
\(\eta\in\Ptwo(\slowspace)\) and \((x,u)\in\slowspace\),
equation \eqref{eq:conditional-frozen-sde} has a unique invariant
probability measure, denoted by \(\nu^{x,u;\eta}\).  The map
\[
 (x,u)\longmapsto\nu^{x,u;\eta}
\]
is Borel measurable as a map from \(\slowspace\) to
\(\Ptwo(\fastspace)\), and
\begin{equation}\label{eq:conditional-invariant-moment}
 \int_{\R^m}\abs z^2\,\nu^{x,u;\eta}(\dd z)
 \le
 C\bigl(
 1+\abs x^2+\abs u^2+m_2(\eta)^2
 \bigr).
\end{equation}
Furthermore,
\begin{equation}\label{eq:joint-disintegration}
 \Pi^\eta(\dd(x,u),\dd y)
 =
 \eta(\dd(x,u))\,\nu^{x,u;\eta}(\dd y).
\end{equation}
\end{lemma}

\begin{proof}
For any square-integrable initial random variables \(Z_0\) and \(Z_0'\),
let \(Z\) and \(Z'\) be the corresponding solutions of
\eqref{eq:conditional-frozen-sde}, started from \(Z_0\) and \(Z_0'\),
respectively, and driven by the same Brownian motion \(B\).
Since the two equations have the same frozen pair \((x,u)\) and the same
fixed measure parameter \(\Pi^\eta\), \eqref{eq:joint-dissipativity} gives
\begin{equation}\label{eq:ordinary-contraction}
 \E\abs{Z_t-Z_t'}^2
 \le \ee^{-\kappa t}\E\abs{Z_0-Z_0'}^2.
\end{equation}
Apply
\eqref{eq:joint-dissipativity} to
\((x,z,u,\Pi^\eta)\) and
\((x,0,u,\Pi^\eta)\), Young's inequality and
\eqref{eq:H1-bhg} give constants \(c_1,C>0\) such that
\begin{align}
 2\ip z{h(x,z,u,\Pi^\eta)}
 +\abs{g(x,z,u,\Pi^\eta)}^2
 \le
 -c_1\abs z^2
 +C\bigl(
 1+\abs x^2+\abs u^2+m_2(\Pi^\eta)^2
 \bigr).
 \label{eq:ordinary-lyapunov}
\end{align}
Since the \((x,u)\)-marginal of \(\Pi^\eta\) is \(\eta\), its second moment satisfies
\[
 m_2(\Pi^\eta)^2
 =
 m_2(\eta)^2+\int\abs y^2\,\Pi^\eta(\dd(x,u),\dd y)
 \le C(1+m_2(\eta)^2),
\]
where the last inequality follows from \eqref{eq:joint-invariant-moment}. Let \(Q_t^{x,u;\eta}\) be the transition semigroup of
\eqref{eq:conditional-frozen-sde}; thus, if
\(Z^{x,u;\eta,y}\) denotes the solution started from \(y\), then, for every
bounded Borel measurable function \(\varphi:\fastspace\to\mathbb R\),
$
 Q_t^{x,u;\eta}\varphi(y)
 =
 \E\bigl[\varphi(Z_t^{x,u;\eta,y})\bigr].
$
Its dual action on probability measures is denoted by
\(Q_t^{x,u;\eta,*}\); equivalently, if \(Z_0\sim\zeta\), then
$
 Q_t^{x,u;\eta,*}\zeta=\Law(Z_t).
$ 
The linear
growth of the coefficients implies that
\(Q_t^{x,u;\eta,*}\) maps \(\Ptwo(\fastspace)\) into itself. For \(\zeta,\zeta'\in\Ptwo(\fastspace)\), choose an optimal coupling
\((Z_0,Z_0')\) of \((\zeta,\zeta')\), whose existence follows from
\cite[Theorem~4.1]{Villani2009}, and let the corresponding solutions
be driven by the same Brownian motion.  Then
\((Z_t,Z_t')\) is a coupling of
\(Q_t^{x,u;\eta,*}\zeta\) and \(Q_t^{x,u;\eta,*}\zeta'\).
Hence, by \eqref{eq:ordinary-contraction},
\begin{equation*}
 W_2^2\bigl(
 Q_t^{x,u;\eta,*}\zeta,
 Q_t^{x,u;\eta,*}\zeta'
 \bigr)
 \le
 \ee^{-\kappa t}W_2^2(\zeta,\zeta'),
 \qquad \zeta,\zeta'\in\Ptwo(\fastspace).
\end{equation*}
Since \((\Ptwo(\fastspace),W_2)\) is complete,
\(Q_1^{x,u;\eta,*}\) has, by Banach's fixed-point theorem, a unique fixed
point; call it \(\nu^{x,u;\eta}\).  For \(t\ge0\), the semigroup property
shows that
\[
 Q_1^{x,u;\eta,*}
 Q_t^{x,u;\eta,*}\nu^{x,u;\eta}
 =
 Q_t^{x,u;\eta,*}
 Q_1^{x,u;\eta,*}\nu^{x,u;\eta}
 =Q_t^{x,u;\eta,*}\nu^{x,u;\eta}.
\]
Uniqueness of the fixed point of \(Q_1^{x,u;\eta,*}\) therefore yields
\(Q_t^{x,u;\eta,*}\nu^{x,u;\eta}=\nu^{x,u;\eta}\), so this measure is
invariant for every time \(t\).

By the same localization argument as in the proof of
Lemma~\ref{lem:frozen-lyapunov}, It\^o's formula,
\eqref{eq:ordinary-lyapunov}, and Gronwall's inequality give
\begin{align*}
 \E\abs{Z_t}^2
 &\le
 \ee^{-c_1t}\E\abs{Z_0}^2
 +\frac C{c_1}(1-\ee^{-c_1t})
 \bigl(1+\abs x^2+\abs u^2+m_2(\Pi^\eta)^2\bigr).
\end{align*}
Starting from the invariant measure and taking \(t=1\) in this inequality
gives
\[
 (1-\ee^{-c_1})
 \int\abs z^2\,\nu^{x,u;\eta}(\dd z)
 \le
 \frac C{c_1}(1-\ee^{-c_1})
 \bigl(1+\abs x^2+\abs u^2+m_2(\Pi^\eta)^2\bigr),
\]
which, together with the preceding bound on \(m_2(\Pi^\eta)\), proves
\eqref{eq:conditional-invariant-moment}.

It remains to verify that these invariant measures form a Borel kernel.
Let \(Z^{x,u}\) and \(Z^{x',u'}\) start from zero and be driven by the same
Brownian motion.  Applying \eqref{eq:joint-dissipativity} with the same
joint measure \(\Pi^\eta\) gives
\[
 \frac{\dd}{\dd t}\E\abs{Z_t^{x,u}-Z_t^{x',u'}}^2
 \le
 -\kappa\E\abs{Z_t^{x,u}-Z_t^{x',u'}}^2
 +K_1\bigl(\abs{x-x'}^2+\abs{u-u'}^2\bigr).
\]
Gronwall's inequality yields
\begin{align*}
 W_2^2\bigl(\Law(Z_t^{x,u}),\Law(Z_t^{x',u'})\bigr)
 &\le
 \frac{K_1}{\kappa}(1-\ee^{-\kappa t})
 \bigl(\abs{x-x'}^2+\abs{u-u'}^2\bigr).
\end{align*}
The asserted Borel measurability will follow from
Lemma~\ref{lem:conditional-invariant-stability} below, which in fact
shows that, for fixed \(\eta\), the map
$
 (x,u)\longmapsto \nu^{x,u;\eta}
$
is Lipschitz continuous from \(\slowspace\) into
\((\Ptwo(\fastspace),W_2)\).

We now identify the disintegration of \(\Pi^\eta\).  Let
\(k^{x,u}\) be a Borel disintegration kernel such that
\[
 \Pi^\eta(\dd(x,u),\dd y)
 =
 \eta(\dd(x,u))\,k^{x,u}(\dd y).
\]
Start \eqref{eq:frozen-mf-equation} from \(\Pi^\eta\).  By
\eqref{eq:joint-invariance}, its joint law remains \(\Pi^\eta\); hence the
law appearing in its coefficients is the fixed measure \(\Pi^\eta\).
For \(\eta\)-almost every \((x,u)\), let
\(\Pp^{x,u}\) be a regular conditional probability given
\(\Theta=(x,u)\).  We now consider \eqref{eq:frozen-mf-equation} under $\Pp^{x,u}$. Under \(\Pp^{x,u}\), \(B\) remains a
\(d_2\)-dimensional Brownian motion and \(Y_0\) has law \(k^{x,u}\).
Since \eqref{eq:joint-invariance} gives
\(\Law(\Theta,Y_t)=\Pi^\eta\) for every \(t\ge0\), the measure argument
in \eqref{eq:frozen-mf-equation} is fixed at \(\Pi^\eta\).  Hence,
under \(\Pp^{x,u}\),
\[
 \dd Y_t
 =
 h(x,Y_t,u,\Pi^\eta)\dd t
 +
 g(x,Y_t,u,\Pi^\eta)\dd B_t,
\]
and therefore
$
 \Law_{\Pp^{x,u}}(Y_t)
 =
 Q_t^{x,u;\eta,*}k^{x,u}.
$

Let \(\mathcal A\) be the countable \(\pi\)-system of half-open rectangles
in \(\R^m\) with rational endpoints.  As
\(\sigma(\mathcal A)=\mathcal B(\R^m)\), the class \(\mathcal A\) is
measure determining; see, e.g., \cite[Lemma~1.17]{Kallenberg2021}.
The stationarity
\[
 \Law(\Theta,Y_q)=\Law(\Theta,Y_0)=\Pi^\eta
\]
implies that, for every bounded Borel function \(\psi\) on \(\slowspace\),
\(A\in\mathcal A\), and \(q\in\mathbb Q_+\),
\[
 \E\bigl[\psi(\Theta)\mathbf 1_A(Y_q)\bigr]
 =
 \E\bigl[\psi(\Theta)\mathbf 1_A(Y_0)\bigr].
\]
Using
\[
 \Law(Y_q\mid\Theta=(x,u))
 =
 Q_q^{x,u;\eta,*}k^{x,u},
 \qquad
 \Law(Y_0\mid\Theta=(x,u))
 =
 k^{x,u},
\]
we obtain
\[
 0
 =
 \int_{\slowspace}
 \psi(x,u)
 \Bigl(
 Q_q^{x,u;\eta,*}k^{x,u}(A)
 -
 k^{x,u}(A)
 \Bigr)
 \eta(\dd(x,u)).
\]
By the arbitrariness of \(\psi\), for every
\((q,A)\in\mathbb Q_+\times\mathcal A\), the identity
$
 Q_q^{x,u;\eta,*}k^{x,u}(A)
 =
 k^{x,u}(A)
$
holds for \(\eta\)-almost every \((x,u)\). By countability of
\(\mathbb Q_+\times\mathcal A\), these identities hold
simultaneously outside a single \(\eta\)-null set.  The
measure-determining property of \(\mathcal A\) then gives
\[
 Q_q^{x,u;\eta,*}k^{x,u}
 =
 k^{x,u},
 \qquad q\in\mathbb Q_+.
\]
For arbitrary \(t\ge0\), choose \(q_j\in\mathbb Q_+\) such that
\(q_j\to t\).  The path continuity of solutions to
\eqref{eq:conditional-frozen-sde} implies that, for every
\(\varphi\in C_b(\R^m)\),
\[
 Q_{q_j}^{x,u;\eta}\varphi(y)
 \longrightarrow
 Q_t^{x,u;\eta}\varphi(y),
 \qquad y\in\R^m.
\]
By bounded convergence and the invariance at rational times,
\[
 \int_{\R^m}Q_t^{x,u;\eta}\varphi(y)\,k^{x,u}(\dd y)
 =
 \lim_{j\to\infty}
 \int_{\R^m}Q_{q_j}^{x,u;\eta}\varphi(y)\,k^{x,u}(\dd y)
 =
 \int_{\R^m}\varphi(y)\,k^{x,u}(\dd y).
\]
Therefore
\[
 Q_t^{x,u;\eta,*}k^{x,u}=k^{x,u},
 \qquad t\ge0.
\]
Hence \(k^{x,u}\) is invariant for
\eqref{eq:conditional-frozen-sde}; uniqueness gives
\[
 k^{x,u}=\nu^{x,u;\eta}
 \qquad
 \text{for \(\eta\)-almost every \((x,u)\)}.
\]
Together with the Borel measurability supplied by
Lemma~\ref{lem:conditional-invariant-stability}, this yields
\eqref{eq:joint-disintegration}.
\end{proof}

\begin{lemma}\label{lem:joint-invariant-stability}
Under {\rm(H1)--(H2)}, there exists \(C>0\) such that, for all
\(\eta,\eta'\in\Ptwo(\slowspace)\),
\begin{equation}\label{eq:joint-invariant-stability}
 W_2(\Pi^\eta,\Pi^{\eta'})
 \le C W_2(\eta,\eta').
\end{equation}
\end{lemma}

\begin{proof}
Choose an optimal coupling
\((\Theta,\Theta')\) of \(\eta\) and \(\eta'\), so that
$
 \E\abs{\Theta-\Theta'}^2
 =
 W_2^2(\eta,\eta').
$
Let \(\Gamma\) denote the law of \((\Theta,\Theta')\), and write the
disintegrations
\[
 \Pi^\eta(\dd\theta,\dd y)
 =
 \eta(\dd\theta)k^\theta(\dd y),
 \qquad
 \Pi^{\eta'}(\dd\theta',\dd y')
 =
 \eta'(\dd\theta')k'^{\theta'}(\dd y').
\]
Define a probability measure on
\(\slowspace^2\times\fastspace^2\) by
\[
 \widehat\Gamma(
 \dd\theta,\dd\theta',\dd y,\dd y')
 :=
 \Gamma(\dd\theta,\dd\theta')\,
 k^\theta(\dd y)\,
 k'^{\theta'}(\dd y').
\]
Realize
\((\Theta,\Theta',Y_0,Y_0')\) with law \(\widehat\Gamma\).  Since the
marginals of \(\Gamma\) are \(\eta\) and \(\eta'\), respectively,
\[
 \Law(\Theta,Y_0)=\Pi^\eta,\qquad
 \Law(\Theta',Y_0')=\Pi^{\eta'}.
\]
Drive the two frozen equations synchronously.  Both marginal joint laws
remain stationary. Condition \eqref{eq:joint-dissipativity} gives
\begin{align}
 \frac{\dd}{\dd t}\E\abs{Y_t-Y_t'}^2
 &\le
 -\kappa\E\abs{Y_t-Y_t'}^2
 +K_1W_2^2(\eta,\eta')
 +K_2 W_2^2(\Pi^\eta,\Pi^{\eta'}).
 \label{eq:pi-stability-first}
\end{align}
Since
$
 \Law(\Theta,Y_t)=\Pi^\eta,
 \Law(\Theta',Y_t')=\Pi^{\eta'},
$
the pair
\(
((\Theta,Y_t),(\Theta',Y_t'))
\)
defines a coupling of \(\Pi^\eta\) and \(\Pi^{\eta'}\).  Hence,
using the optimality of the initial coupling
\((\Theta,\Theta')\),
\begin{equation}\label{eq:pi-stability-coupling}
 W_2^2(\Pi^\eta,\Pi^{\eta'})
 \le
 \E\abs{\Theta-\Theta'}^2+\E\abs{Y_t-Y_t'}^2
 =
 W_2^2(\eta,\eta')+\E\abs{Y_t-Y_t'}^2.
\end{equation}
Substitution into \eqref{eq:pi-stability-first} gives
\[
 \frac{\dd}{\dd t}\E\abs{Y_t-Y_t'}^2
 \le
 -\alpha\E\abs{Y_t-Y_t'}^2
 +(K_1+K_2)W_2^2(\eta,\eta').
\]
Therefore
\[
 \E\abs{Y_t-Y_t'}^2
 \le
 \ee^{-\alpha t}\E\abs{Y_0-Y_0'}^2
 +\frac{K_1+K_2}{\alpha}
 (1-\ee^{-\alpha t})W_2^2(\eta,\eta').
\]
Let \(t\to\infty\) in \eqref{eq:pi-stability-coupling}.  We obtain
\[
 W_2^2(\Pi^\eta,\Pi^{\eta'})
 \le
 \bigl(1+\frac{K_1+K_2}{\alpha}\bigr)
 W_2^2(\eta,\eta'),
\]
which proves the result.
\end{proof}

\begin{lemma}\label{lem:conditional-invariant-stability}
Under {\rm(H1)--(H2)}, there exists \(C>0\) such that, for all
\((x,u),(x',u')\in\slowspace\) and
\(\eta,\eta'\in\Ptwo(\slowspace)\),
\begin{align}
 W_2\bigl(
 \nu^{x,u;\eta},\nu^{x',u';\eta'}
 \bigr)
 \le
 C\bigl(
 \abs{x-x'}+\abs{u-u'}+W_2(\eta,\eta')
 \bigr).
 \label{eq:conditional-invariant-stability}
\end{align}
\end{lemma}

\begin{proof}
Let \((Z_0,Z_0')\) be an arbitrary coupling of
\(\nu^{x,u;\eta}\) and \(\nu^{x',u';\eta'}\), and drive the corresponding
solutions of \eqref{eq:conditional-frozen-sde} with the same Brownian
motion.  Each marginal remains stationary.  By
\eqref{eq:joint-dissipativity},
\[
 \frac{\dd}{\dd t}\E\abs{Z_t-Z_t'}^2
 \le
 -\kappa\E\abs{Z_t-Z_t'}^2
 +K_1(\abs{x-x'}^2+\abs{u-u'}^2)
 +K_2 W_2^2(\Pi^\eta,\Pi^{\eta'}).
\]
Lemma~\ref{lem:joint-invariant-stability} gives
\[
 \frac{\dd}{\dd t}\E\abs{Z_t-Z_t'}^2
 \le
 -\kappa\E\abs{Z_t-Z_t'}^2
 +C\bigl(
 \abs{x-x'}^2+\abs{u-u'}^2+W_2^2(\eta,\eta')
 \bigr).
\]
Gronwall's inequality gives
\[
 \E\abs{Z_t-Z_t'}^2
 \le
 \ee^{-\kappa t}\E\abs{Z_0-Z_0'}^2
 +C(1-\ee^{-\kappa t})
 \bigl(
 \abs{x-x'}^2+\abs{u-u'}^2+W_2^2(\eta,\eta')
 \bigr).
\]
Since both marginals remain stationary,
\((Z_t,Z_t')\) is a coupling of
\(\nu^{x,u;\eta}\) and \(\nu^{x',u';\eta'}\) for every \(t\ge0\).
Therefore,
$
 W_2^2(
 \nu^{x,u;\eta},\nu^{x',u';\eta'})
 \le
 \E\abs{Z_t-Z_t'}^2.
$
Letting \(t\to\infty\) yields
\begin{align*}
 W_2^2(
 \nu^{x,u;\eta},\nu^{x',u';\eta'})
 \le
 C\bigl(
 \abs{x-x'}^2+\abs{u-u'}^2+W_2^2(\eta,\eta')
 \bigr).
\end{align*}
This proves
\eqref{eq:conditional-invariant-stability}.
\end{proof}

\subsection{The Joint Lifted Semigroup and Averaged Coefficients}

Fix \(\eta\in\Ptwo(\slowspace)\) and
\(\rho\in{\Ptwo}_\eta\).  For
\((x,u,y)\in\slowspace\times\fastspace\), let
\(Y_t^{x,u,y,\rho}\) solve
\begin{equation}\label{eq:pointwise-fast-equation}
\left\{
\begin{aligned}
\dd Y_t^{x,u,y,\rho}
 &=
 h\bigl(
 x,Y_t^{x,u,y,\rho},u,\mathscr S_t^\eta\rho
 \bigr)\dd t+
 g\bigl(
 x,Y_t^{x,u,y,\rho},u,\mathscr S_t^\eta\rho
 \bigr)\dd B_t,
\\
Y_0^{x,u,y,\rho}&=y.
\end{aligned}
\right.
\end{equation}
For a measurable function
\(\Phi=\Phi(x,u,y,\rho)\), define
\begin{equation}\label{eq:lifted-semigroup-definition}
 \widetilde P_t\Phi(x,u,y,\rho)
 :=
 \E\bigl[
 \Phi\bigl(
 x,u,Y_t^{x,u,y,\rho},\mathscr S_t^\eta\rho
 \bigr)
 \bigr].
\end{equation}
Note that the slow marginal \(\eta\) is recovered from \(\rho\); hence it is not an
additional component of the lifted state.

\begin{remark}
{\rm The auxiliary equation above separates the pointwise fast dynamics from the
nonlinear evolution of the population law.  Indeed, once \(\rho\) is fixed,
\(\mathscr S_t^\eta\rho\) is a prescribed law flow, and
\(Y_t^{x,u,y,\rho}\) evolves as an ordinary SDE driven by this flow.
Enlarging the state to include \(\rho\) then allows one to recover a
Markov semigroup structure.  This construction is in the spirit of the
splitting device for McKean--Vlasov equations in
\cite{BuckdahnLiPengRainer2017} and the lifted semigroup approach in
\cite{HongHuLiuYang2026}.}
\end{remark}

\begin{lemma}\label{lem:lifted-semigroup-property}
Under {\rm(H1)}, the family \((\widetilde P_t)_{t\ge0}\) satisfies
\[
 \widetilde P_{t+s}
 =
 \widetilde P_s\widetilde P_t,
 \qquad s,t\ge0.
\]
\end{lemma}

\begin{proof}
Fix \(s,t\ge0\), and set
\[
 \rho_s:=\mathscr S_s^\eta\rho,
 \qquad
 \widehat Y_r:=Y_{s+r}^{x,u,y,\rho},
 \qquad
 \widehat B_r:=B_{s+r}-B_s,
 \quad 0\le r\le t.
\]
By \eqref{eq:nonlinear-semigroup},
\[
 \mathscr S_{s+r}^\eta\rho
 =
 \mathscr S_r^\eta\rho_s,
 \qquad r\ge0.
\]
Hence, by \eqref{eq:pointwise-fast-equation},
\begin{align*}
 \widehat Y_r
 &=
 Y_s^{x,u,y,\rho}
 +\int_0^r
 h\bigl(
 x,\widehat Y_q,u,\mathscr S_q^\eta\rho_s
 \bigr)\dd q
+
 \int_0^r
 g\bigl(
 x,\widehat Y_q,u,\mathscr S_q^\eta\rho_s
 \bigr)\dd\widehat B_q .
\end{align*}
Let
$
 \mathcal B_s:=\sigma(B_r:0\le r\le s)^{\Pp},
$
then \(Y_s^{x,u,y,\rho}\) is \(\mathcal B_s\)-measurable, and
\((\widehat B_r)_{r\ge0}\) is independent of \(\mathcal B_s\).
Thus strong uniqueness for \eqref{eq:pointwise-fast-equation} yields
\[
 \E\bigl[
 \Phi\bigl(
 x,u,Y_{s+t}^{x,u,y,\rho},
 \mathscr S_{s+t}^\eta\rho
 \bigr)
 \,\bigm|\,\mathcal B_s
 \bigr]
 =
 \widetilde P_t\Phi(x,u,z,\rho_s)
 \bigr|_{z=Y_s^{x,u,y,\rho}}
 \quad\text{a.s.}
\]
Therefore, by the tower property and
\eqref{eq:lifted-semigroup-definition},
\[
 \widetilde P_{t+s}\Phi(x,u,y,\rho)
 =
 \E\bigl[
 \widetilde P_t\Phi\bigl(
 x,u,Y_s^{x,u,y,\rho},\rho_s
 \bigr)
 \bigr]
 =\widetilde P_s(\widetilde P_t\Phi)(x,u,y,\rho).
\]
Hence
\[
 \widetilde P_{t+s}
 =
 \widetilde P_s\widetilde P_t .
\]
\end{proof}

The limiting projection associated with the lifted semigroup is defined by
\[
 \mathcal Q^\eta\Phi(x,u)
 :=
 \int_{\R^m}
 \Phi(x,u,z,\Pi^\eta)\,\nu^{x,u;\eta}(\dd z).
\]

\begin{lemma}
\label{lem:lifted-mixing}
Under {\rm(H1)--(H2)}, fix
\(\eta\in\Ptwo(\slowspace)\). Suppose that, uniformly in \((x,u)\),
\begin{equation}\label{eq:Phi-lipschitz}
 \abs{\Phi(x,u,y,\rho)-\Phi(x,u,y',\rho')}
 \le
 L_\Phi\bigl(\abs{y-y'}+W_2(\rho,\rho')\bigr)
\end{equation}
for all \(y,y'\in\fastspace\) and
\(\rho,\rho'\in{\Ptwo}_\eta\). Then, for every
\((x,u,y)\in\slowspace\times\fastspace\),
\(\rho\in{\Ptwo}_\eta\), and \(t\ge0\),
\begin{equation}
 \abs{
 \widetilde P_t\Phi(x,u,y,\rho)
 -\mathcal Q^\eta\Phi(x,u)}
 \le
 C L_\Phi\ee^{-\alpha t/2}
 \bigl(
 1+\abs x+\abs u+\abs y+m_2(\rho)
 \bigr).
 \label{eq:lifted-mixing}
\end{equation}
\end{lemma}

\begin{proof}
On the same filtered probability space, let \(\bar Z\) be driven by the
same Brownian motion \(B\) and solve
\[
 \dd\bar Z_t
 =
 h(x,\bar Z_t,u,\Pi^\eta)\dd t
 +g(x,\bar Z_t,u,\Pi^\eta)\dd B_t,
 \qquad
 \Law(\bar Z_0)=\nu^{x,u;\eta}.
\]
Then \(\bar Z\) is stationary.  Put
\(D_t=Y_t^{x,u,y,\rho}-\bar Z_t\).  By
\eqref{eq:joint-dissipativity},
\begin{equation}\label{eq:lifted-mixing-differential}
 \frac{\dd}{\dd t}\E\abs{D_t}^2
 \le
 -\kappa\E\abs{D_t}^2
 +K_2 W_2^2(\mathscr S_t^\eta\rho,\Pi^\eta).
\end{equation}
Equations \eqref{eq:ordinary-below-fiber} and
\eqref{eq:joint-invariant-convergence} imply
\begin{equation}\label{eq:law-mixing-bound}
 W_2^2(\mathscr S_t^\eta\rho,\Pi^\eta)
 \le
 \ee^{-\alpha t}
 \Wass_{2,\eta}^2(\rho,\Pi^\eta).
\end{equation}
Moreover, using the product coupling
\(\rho^{x,u}\otimes\nu^{x,u;\eta}\) on each fiber and
\eqref{eq:joint-invariant-moment}, we obtain
\begin{equation*}
 \Wass_{2,\eta}^2(\rho,\Pi^\eta)
 \le C\bigl(1+m_2(\rho)^2\bigr).
\end{equation*}

If \(K_2=0\), \eqref{eq:lifted-mixing-differential} directly gives
\(\E\abs{D_t}^2\le\ee^{-\kappa t}\E\abs{D_0}^2\), and
\(\alpha=\kappa\).  If \(K_2>0\), then
\(\kappa-\alpha=K_2\), and Gronwall's inequality gives
\begin{align*}
 \E\abs{D_t}^2
 &\le
 \ee^{-\kappa t}\E\abs{D_0}^2
 +K_2\Wass_{2,\eta}^2(\rho,\Pi^\eta)
 \int_0^t\ee^{-\kappa(t-r)}\ee^{-\alpha r}\dd r
\\
 &=
 \ee^{-\kappa t}\E\abs{D_0}^2
 +\bigl(\ee^{-\alpha t}-\ee^{-\kappa t}\bigr)
 \Wass_{2,\eta}^2(\rho,\Pi^\eta).
\end{align*}
Both cases therefore imply
\begin{align*}
 \E\abs{D_t}^2
 \le
 C\ee^{-\alpha t}
 \bigl(
 1+\abs x^2+\abs u^2+\abs y^2+m_2(\rho)^2
 \bigr),
\end{align*}
where \eqref{eq:conditional-invariant-moment} yields the required bound for
\(\E\abs{\bar Z_0}^2\).
Using \eqref{eq:Phi-lipschitz}, the stationary law of \(\bar Z_t\), and
\eqref{eq:law-mixing-bound},
\[
 \abs{
 \widetilde P_t\Phi(x,u,y,\rho)
 -\mathcal Q^\eta\Phi(x,u)}
\le
 L_\Phi\E\abs{D_t}
 +L_\Phi W_2(\mathscr S_t^\eta\rho,\Pi^\eta).
\]
which proves \eqref{eq:lifted-mixing}.
\end{proof}

\begin{lemma}\label{lem:averaged-coefficients}
Under {\rm(H1)--(H2)}, the functions in
\eqref{eq:averaged-b-definition} and
\eqref{eq:averaged-F-definition} are Borel measurable and satisfy
\begin{align}
 \abs{\bar b(x,u,\eta)-\bar b(x',u',\eta')}
 &\le
 C\bigl(
 \abs{x-x'}+\abs{u-u'}+W_2(\eta,\eta')
 \bigr),
 \label{eq:averaged-b-lipschitz}\\
 \abs{\bar F(x,u,v,\eta)-\bar F(x',u',v',\eta')}
 &\le
 C\bigl(
 \abs{x-x'}+\abs{u-u'}+\abs{v-v'}
 +W_2(\eta,\eta')
 \bigr),
 \label{eq:averaged-F-lipschitz}\\
 \abs{\bar b(x,u,\eta)}^2
 &\le
 C\bigl(1+\abs x^2+\abs u^2+m_2(\eta)^2\bigr),
 \label{eq:averaged-b-growth}\\
 \abs{\bar F(x,u,v,\eta)}^2
 &\le
 C\bigl(1+\abs x^2+\abs u^2+\abs v^2+m_2(\eta)^2\bigr).
 \label{eq:averaged-F-growth}
\end{align}
\end{lemma}

\begin{proof}
By \eqref{eq:linear-growth-bhg}, \eqref{eq:linear-growth-F},
\eqref{eq:joint-invariant-moment}, and
\eqref{eq:conditional-invariant-moment}, the integrals in
\eqref{eq:averaged-b-definition} and
\eqref{eq:averaged-F-definition} are finite.

Let \((Z,Z')\) be an optimal coupling of
\(\nu^{x,u;\eta}\) and \(\nu^{x',u';\eta'}\).  Then
\[
 \E\abs{Z-Z'}
 \le
 \bigl(\E\abs{Z-Z'}^2\bigr)^{1/2}
 =
 W_2\bigl(
 \nu^{x,u;\eta},\nu^{x',u';\eta'}
 \bigr).
\]
Hence, by \eqref{eq:averaged-b-definition},
\begin{align*}
 &\abs{\bar b(x,u,\eta)-\bar b(x',u',\eta')}
\\
 &\quad\le
 \E\abs{
 b(x,Z,u,\Pi^\eta)
 -
 b(x',Z',u',\Pi^{\eta'})}
\\
 &\quad\le
 L\bigl(
 \abs{x-x'}+\abs{u-u'}
 +W_2\bigl(
 \nu^{x,u;\eta},\nu^{x',u';\eta'}
 \bigr)
 +W_2(\Pi^\eta,\Pi^{\eta'})
 \bigr)
\\
 &\quad\le
 C\bigl(
 \abs{x-x'}+\abs{u-u'}+W_2(\eta,\eta')
 \bigr),
\end{align*}
where the last inequality follows from
\eqref{eq:joint-invariant-stability} and
\eqref{eq:conditional-invariant-stability}.  This proves
\eqref{eq:averaged-b-lipschitz}.

Similarly, by \eqref{eq:averaged-F-definition},
\begin{align*}
 &\abs{\bar F(x,u,v,\eta)-\bar F(x',u',v',\eta')}
\\
 &\quad\le
 L\bigl(
 \abs{x-x'}+\abs{u-u'}+\abs{v-v'}
 +W_2\bigl(
 \nu^{x,u;\eta},\nu^{x',u';\eta'}
 \bigr)
 +W_2(\Pi^\eta,\Pi^{\eta'})
 \bigr)
\\
 &\quad\le
 C\bigl(
 \abs{x-x'}+\abs{u-u'}+\abs{v-v'}
 +W_2(\eta,\eta')
 \bigr),
\end{align*}
which proves \eqref{eq:averaged-F-lipschitz}.  In particular,
\(\bar b\) and \(\bar F\) are continuous, hence Borel measurable.

Finally, Jensen's inequality and \eqref{eq:linear-growth-bhg} give
\begin{align*}
 \abs{\bar b(x,u,\eta)}^2
 &\le
 \int_{\R^m}
 \abs{b(x,z,u,\Pi^\eta)}^2
 \,\nu^{x,u;\eta}(\dd z)
\\
 &\le
 C\bigl(
 1+\abs x^2+\abs u^2
 +m_2(\Pi^\eta)^2
 +\int_{\R^m}\abs z^2
 \,\nu^{x,u;\eta}(\dd z)
 \bigr).
\end{align*}
Since \(\Pi^\eta\in{\Ptwo}_\eta\),
\[
 m_2(\Pi^\eta)^2
 =
 m_2(\eta)^2
 +
 \int_{\slowspace\times\fastspace}
 \abs y^2\,\Pi^\eta(\dd(x,u),\dd y),
\]
and therefore
\eqref{eq:joint-invariant-moment} and
\eqref{eq:conditional-invariant-moment} yield
\eqref{eq:averaged-b-growth}.

Likewise, Jensen's inequality and \eqref{eq:linear-growth-F} give
\begin{align*}
 \abs{\bar F(x,u,v,\eta)}^2
 &\le
 \int_{\R^m}
 \abs{
 F\bigl(x,z,u,(v,0),\Pi^\eta\bigr)
 }^2
 \,\nu^{x,u;\eta}(\dd z)
\\
 &\le
 C\bigl(
 1+\abs x^2+\abs u^2+\abs v^2
 +m_2(\Pi^\eta)^2
 +\int_{\R^m}\abs z^2
 \,\nu^{x,u;\eta}(\dd z)
 \bigr),
\end{align*}
and the same two moment estimates prove
\eqref{eq:averaged-F-growth}.
\end{proof}

\begin{remark}
\label{rem:random-environment-freezing}
{\rm At the end of this section, we explain why the frozen equation is formulated as in \eqref{eq:frozen-mf-equation}, rather than in the classical form. To make the idea clear and for comparison with existing works, in this remark we mainly consider the forward mean-field equation. In a fully coupled
mean-field problem, freezing the slow variables at deterministic parameters
and then averaging with respect to the invariant law of the resulting fast
McKean--Vlasov equation can lead to an incorrect effective dynamics.  This
phenomenon was pointed out by Hou, Li, and Xie
\cite{HouLiXie2024}.  The underlying issue is that freezing removes the
temporal evolution of the slow variables on the fast time scale, but it does
not remove their randomness across the population.

We first illustrate the obstruction by an explicit scalar example.
Fix $q\in(0,1)$ and let $\xi$ satisfy
\[
 \Pp(\xi=1)=\Pp(\xi=-1)=\frac12.
\]
Consider
\begin{equation}\label{eq:naive-freezing-example}
\left\{
\begin{aligned}
 \dd X_t^\eps
 &=m_t^\eps\dd t,
 \qquad X_0^\eps=\xi,
 \\
 \dd Y_t^\eps
 &=\frac1\eps
 \bigl(-Y_t^\eps+X_t^\eps+q m_t^\eps\bigr)\dd t
 +\sqrt{\frac2\eps}\,\dd W_t,
 \qquad Y_0^\eps=0,
\end{aligned}
\right.
\end{equation}
where $m_t^\eps:=\E Y_t^\eps$.  If
$n_t^\eps:=\E X_t^\eps$, then
\[
 \frac{\dd}{\dd t}n_t^\eps=m_t^\eps,
 \qquad
 \eps\frac{\dd}{\dd t}m_t^\eps
 =n_t^\eps-(1-q)m_t^\eps.
\]
Since $n_0^\eps=m_0^\eps=0$, one has
$n_t^\eps=m_t^\eps=0$ for every $t\ge0$.  Consequently,
$X_t^\eps=\xi$ for every $\eps>0$. Conditional on $\xi=x$, we have
\[
 \Law(Y_t^\eps\mid\xi=x)
 =\mathcal N\bigl(x(1-\ee^{-t/\eps}),
                  1-\ee^{-2t/\eps}\bigr),
\]
where $\mathcal N(\mu,\sigma^2)$ denotes the normal distribution with mean $\mu$ and variance $\sigma^2$. Thus, for every $t>0$,
\[
 \Law(Y_t^\eps)
 \longrightarrow
 \frac12\mathcal N(-1,1)+\frac12\mathcal N(1,1).
\]
Now apply the naive deterministic freezing procedure.  For fixed $x$, the
corresponding frozen McKean--Vlasov equation would be
\[
 \dd Z_s^x
 =\bigl(-Z_s^x+x+q\E Z_s^x\bigr)\dd s
 +\sqrt2\,\dd B_s.
\]
Its unique invariant law is
\[
 \zeta^x=\mathcal N\big(\frac{x}{1-q},1\big).
\]
Since the slow drift in \eqref{eq:naive-freezing-example} is the mean of the
fast law, this procedure would produce the averaged drift
\[
 \bar b_{\rm naive}(x)
 =\int_{\R}z\,\zeta^x(\dd z)
 =\frac{x}{1-q},
\]
and hence the candidate limit
\[
 \dd\bar X_t=\frac{\bar X_t}{1-q}\dd t,
 \qquad \bar X_0=\xi.
\]
This gives $\bar X_t=\ee^{t/(1-q)}\xi$, whereas the exact system satisfies
$X_t^\eps=\xi$ for every $\eps>0$.  The naive averaged equation is therefore
wrong even in this linear Gaussian example.

The correct freezing keeps the slow state random but constant on the fast
time scale.  In the preceding example, let
$\Theta\sim\frac12\delta_{-1}+\frac12\delta_1$ and consider
\[
 \dd Z_s
 =\bigl(-Z_s+\Theta+q\E Z_s\bigr)\dd s
 +\sqrt2\,\dd B_s,
 \qquad \Theta\ \hbox{constant in }s.
\]
Since
\(\E\Theta=0\), the equilibrium mean satisfies
\[
 \E Z_s=\E\Theta+q\E Z_s=q\E Z_s,
\]
and hence \(\E Z_s=0\). Consequently, the mean-field term vanishes at
equilibrium, and the frozen equation reduces to
\[
 \dd Z_s=(-Z_s+\Theta)\dd s+\sqrt2\,\dd B_s.
\]
Conditioning on \(\Theta=x\), we obtain the Ornstein--Uhlenbeck equation
\[
 \dd Z_s=-(Z_s-x)\dd s+\sqrt2\,\dd B_s,
\]
whose invariant law is \(\mathcal N(x,1)\). Therefore, the invariant joint
law of \((\Theta,Z)\) is
\[
 \Pi
 =
 \frac12\delta_{-1}\otimes\mathcal N(-1,1)
 +
 \frac12\delta_1\otimes\mathcal N(1,1).
\]
Its fast marginal is
$
 \Pi_Z
 =
 \frac12\mathcal N(-1,1)
 +
 \frac12\mathcal N(1,1),
$
which has mean zero. Consequently, the averaged slow drift is zero, in
agreement with the exact dynamics. The difference between the two
procedures is that the deterministic frozen equation replaces the whole
population by a single slow fiber $x$, whereas the random frozen environment
keeps all slow fibers, with their correct population weights, while making
each individual slow environment constant in the fast time variable.

This observation also clarifies why deterministic freezing remains valid in the McKean--Vlasov averaging framework when the fast coefficients are independent of the fast population law, see \cite{HongLiSun2025,RocknerSunXie2021}. More precisely, for
\[
 \dd Y_t^\eps
 =\frac1\eps
 H\bigl(X_t^\eps,\Law(X_t^\eps),Y_t^\eps\bigr)\dd t
 +\frac1{\sqrt\eps}
 G\bigl(X_t^\eps,\Law(X_t^\eps),Y_t^\eps\bigr)\dd W_t,
\]
fixing a slow law $\mu$ and a slow state $x$ gives
\[
 \dd Z_s^{x,\mu}
 =H(x,\mu,Z_s^{x,\mu})\dd s
 +G(x,\mu,Z_s^{x,\mu})\dd B_s.
\]
No information about the fast dynamics in the other slow fibers is required
in order to determine $Z^{x,\mu}$.  Equivalently, if one first freezes a
random environment $\Theta\sim\mu$ and then conditions on $\Theta=x$, the
resulting fiber equation is already closed.  Thus the random-environment
formulation decomposes into independent deterministic fibers, and one may
work directly with the parameters $(x,\mu)$ without losing any information.
This provides the justification for the validity of the standard deterministic freezing procedure in this setting.

The situation is fundamentally different when the fast coefficients depend on the fast population law. In the setting of \cite{HouLiXie2024}, where the
coefficients depend separately on the slow and fast marginals, the correct
frozen object is a coupled family indexed by the slow state.  More precisely, for each \((x,\mu)\), the frozen dynamics take the form
\begin{equation}\label{eq:HLX-family-schematic}
\left\{
\begin{aligned}
 \dd Z_s^{x,\mu}
 &=
 H\big(
 x,\mu,Z_s^{x,\mu},
 \int_{\R^n}
 \Law(Z_s^{\widetilde x,\mu})\,\mu(\dd\widetilde x)
 \big)\dd s
 +
 G\big(
 x,\mu,Z_s^{x,\mu},
 \int_{\R^n}
 \Law(Z_s^{\widetilde x,\mu})\,\mu(\dd\widetilde x)
 \big)\dd B_s,
 \\
 Z_0^{x,\mu}&=y_0.
\end{aligned}
\right.
\end{equation}
The law entering the coefficient is therefore not the law of a single fiber
$Z^{x,\mu}$, but the mixture of all conditional fast laws over the slow
population.  It should be emphasized that the proof in \cite{HouLiXie2024} does not proceed directly with the original nonlinear system. Instead, the authors introduce a Picard-type iteration in which the measure arguments at the $n$th step are determined by the $(n-1)$th step. For each fixed $n$ the
resulting system is a classical, linear but non-autonomous SDE.  They first
establish averaging and distributional limits for these non-autonomous
systems, obtain estimates uniform in $n$, and then pass to the limit
$n\to\infty$; see \cite{HouLiXie2024}.

Our formulation incorporates the same population consistency into a single joint-law dynamics and, at the same time, allows the coefficients to depend on the full joint law rather than only on separate marginals. Namely, for a prescribed slow marginal $\eta$, we freeze a random slow
environment $\Theta=(\Theta^x,\Theta^u)\sim\eta$ and consider
$
 \rho_t:=\Law(\Theta,Y_t)\in{\Ptwo}_\eta.
$
By disintegrating $\rho_t$ with respect to its slow marginal $\eta$, we may
write
$
 \rho_t(\dd(x,u),\dd y)
 =\eta(\dd(x,u))\,\rho_t^{x,u}(\dd y),
$
where $\rho_t^{x,u}$ denotes the conditional law of the fast variable on the
slow fiber $(x,u)$. Consequently,
\begin{equation}\label{eq:random-environment-fiber-representation}
\rho_t
=\int_{\slowspace}
\delta_{(x,u)}\otimes\rho_t^{x,u}
\eta(\dd(x,u)).
\end{equation}
Thus, a single joint law records simultaneously the slow population, the
conditional fast law on each slow fiber, and the dependence between the slow
and fast variables. Its fast marginal is given by
$$
 (\pr_{\fastspace})_\#\rho_t
 =\int_{\slowspace}\rho_t^{x,u}\,\eta(\dd(x,u)).
$$

Hence, if the coefficients depend on $\rho_t$ only through its slow and fast
marginals, the joint-law formulation \eqref{eq:random-environment-fiber-representation} reduces exactly to
the mixture structure in \eqref{eq:HLX-family-schematic}.  If there is no
dependence on the fast law at all, the fibers decouple further and one
recovers the classical deterministic freezing described above.  In the
present paper, however, $b,h,g$, and $F$ may depend on the full joint law of
the slow and fast variables.  Separate slow and fast marginals do not retain
this information, whereas $\rho_t=\Law(\Theta,Y_t)$ does.

This viewpoint leads directly to the invariant joint law $\Pi^\eta$ and its
disintegration
\[
 \Pi^\eta(\dd(x,u),\dd z)
 =\eta(\dd(x,u))\,\nu^{x,u;\eta}(\dd z),
\]
which in turn yields \eqref{eq:averaged-b-definition} and
\eqref{eq:averaged-F-definition}.  In contrast with the non-autonomous
iteration of \cite{HouLiXie2024}, we analyze the nonlinear frozen dynamics
directly on the Wasserstein fiber ${\Ptwo}_\eta$ and do not introduce an outer
sequence of classical SDEs to linearize the mean-field interaction. Moreover, the argument does not rely on a Poisson equation and does not
require ellipticity or Lions differentiability with respect to the measure
variable, thereby allowing degenerate diffusion coefficients. Such a
framework is particularly important for the mean-field control application
developed later in the paper.}
\end{remark}

\section{Mean-Field BSDE Estimates}
\label{sec:averaged-preliminaries}

\subsection{{\it A Priori} Estimates for an
\texorpdfstring{\(L^2\)}{L2}-Operator Driver}

We first establish an {\it a priori} estimate for BSDEs whose drivers act on square-integrable random variables rather than pointwise on their values. This form arises naturally from the mean-field dependence of the averaged driver and will be used both to derive moment bounds for the averaged flow and to control the local backward comparison. Since the relevant Lipschitz continuity holds in the \(L^2\)-norm rather than pointwise in \(\omega\), standard Lipschitz BSDE estimates cannot be applied directly.

\begin{lemma}\label{lem:operator-cancellation}
Let \(\widetilde W\) be a \(d\)-dimensional \(\mathbb F\)-Brownian motion,
let \(\xi\in L^2(\Ff_b;\R^p)\), and let
\(r=(r_t)_{a\le t\le b}\) be an \(\R^p\)-valued progressively measurable
process satisfying
\[
 \E\Bigl(\int_a^b\abs{r_s}\dd s\Bigr)^2<\infty.
\]
For \(s\in[a,b]\), let
\[
 \Psi_s:
 L^2(\Ff_s;\R^p)\times L^2(\Ff_s;\R^{p\times d})
 \longrightarrow L^2(\Ff_s;\R^p)
\]
satisfy \(\Psi_s(0,0)=0\). Assume that
\(s\mapsto\Psi_s(P_s,Q_s)\) is progressively measurable whenever
$
 (P,Q)\in
 \mathcal S^2(a,b;\R^p)\times
 \mathcal H^2(a,b;\R^{p\times d}),
$
and that, for all
\(p,p'\in L^2(\Ff_s;\R^p)\) and
\(q,q'\in L^2(\Ff_s;\R^{p\times d})\),
\begin{equation}\label{eq:operator-lipschitz}
 \norm{\Psi_s(p,q)-\Psi_s(p',q')}_{L^2}
 \le
 L\bigl(
 \norm{p-p'}_{L^2}+\norm{q-q'}_{L^2}
 \bigr).
\end{equation}
Suppose that
$
 (P,Q)\in
 \mathcal S^2(a,b;\R^p)\times
 \mathcal H^2(a,b;\R^{p\times d})
$
satisfies
\[
 P_t
 =
 \xi+\int_t^b\bigl[\Psi_s(P_s,Q_s)+r_s\bigr]\dd s
 -\int_t^b Q_s\dd\widetilde W_s.
\]
Then
\begin{equation}\label{eq:operator-cancellation}
 \E\sup_{a\le t\le b}\abs{P_t}^2
 +\E\int_a^b\abs{Q_t}^2\dd t
 \le
 C_{L,b-a}\Bigl[
 \E\abs{\xi}^2
 +\E\sup_{a\le t\le b}
 \Bigl|\int_t^b r_s\dd s\Bigr|^2
 \Bigr].
\end{equation}
Moreover, for every \(h_0>0\),
$
 \sup_{0\le b-a\le h_0}C_{L,b-a}<\infty.
$
\end{lemma}

\begin{proof}
Set \(K_t:=\int_a^t r_s\dd s\) and \(\widehat P_t:=P_t+K_t\). Then
\begin{equation}\label{eq:shifted-operator-bsde}
 \widehat P_t
 =
 \xi+K_b
 +\int_t^b\widehat\Psi_s(\widehat P_s,Q_s)\dd s
 -\int_t^b Q_s\dd\widetilde W_s,
\end{equation}
where \(\widehat\Psi_s(p,q):=\Psi_s(p-K_s,q)\). By
\eqref{eq:operator-lipschitz},
\begin{align}
 \norm{\widehat\Psi_s(p,q)-\widehat\Psi_s(p',q')}_{L^2}
 \le
 L\bigl(
 \norm{p-p'}_{L^2}
 +\norm{q-q'}_{L^2}
 \bigr),
 \label{eq:shifted-operator-lipschitz}\\
 \norm{\widehat\Psi_s(0,0)}_{L^2}
 =
 \norm{\Psi_s(-K_s,0)-\Psi_s(0,0)}_{L^2}
 \le L\norm{K_s}_{L^2}.
 \label{eq:shifted-origin}
\end{align}

We next derive the corresponding \(L^2\)-operator BSDE estimate. For any $\gamma>0$,
applying It\^o's formula to
\(\ee^{\gamma t}\abs{\widehat P_t}^2\), taking expectations, and using
\eqref{eq:shifted-operator-lipschitz}, we obtain
\begin{align*}
 &\ee^{\gamma t}\E\abs{\widehat P_t}^2
 +\E\int_t^b\ee^{\gamma s}
 \bigl(
 \gamma\abs{\widehat P_s}^2+\abs{Q_s}^2
 \bigr)\dd s
\\
 &\quad\le
 \ee^{\gamma b}\E\abs{\xi+K_b}^2
 +C_L\int_t^b\ee^{\gamma s}
 \E\abs{\widehat P_s}^2\dd s
 +\frac12\E\int_t^b\ee^{\gamma s}\abs{Q_s}^2\dd s
 +\int_t^b\ee^{\gamma s}
 \norm{\widehat\Psi_s(0,0)}_{L^2}^2\dd s .
\end{align*}
Choosing \(\gamma>C_L\) yields
\begin{equation}\label{eq:operator-energy}
 \sup_{a\le t\le b}\E\abs{\widehat P_t}^2
 +\E\int_a^b\abs{Q_s}^2\dd s
 \le
 C_{L,b-a}
 \Bigl[
 \E\abs{\xi+K_b}^2
 +\int_a^b\norm{\widehat\Psi_s(0,0)}_{L^2}^2\dd s
 \Bigr].
\end{equation}
Moreover, by \eqref{eq:shifted-operator-bsde}, the
Burkholder--Davis--Gundy inequality, and
\eqref{eq:shifted-operator-lipschitz},
\begin{align*}
 \E\sup_{a\le t\le b}\abs{\widehat P_t}^2
 &\le
 C\E\abs{\xi+K_b}^2
 +C(b-a)\int_a^b
 \norm{\widehat\Psi_s(\widehat P_s,Q_s)}_{L^2}^2\dd s
 +C\E\int_a^b\abs{Q_s}^2\dd s
\\
 &\le
 C_{L,b-a}
 \Bigl[
 \E\abs{\xi+K_b}^2
 +\int_a^b\norm{\widehat\Psi_s(0,0)}_{L^2}^2\dd s
 \Bigr],
\end{align*}
where the last inequality follows from \eqref{eq:operator-energy}. Hence
\begin{equation}\label{eq:shifted-final-estimate}
 \E\sup_{a\le t\le b}\abs{\widehat P_t}^2
 +\E\int_a^b\abs{Q_s}^2\dd s
 \le
 C_{L,b-a}
 \Bigl[
 \E\abs{\xi+K_b}^2
 +\int_a^b\norm{\widehat\Psi_s(0,0)}_{L^2}^2\dd s
 \Bigr].
\end{equation}

Set
$
 R_{a,b}:=
 \sup_{a\le t\le b}
 \bigl|\int_t^b r_s\dd s\bigr|.
$
Since \(\abs{K_b}\le R_{a,b}\) and
\[
 \sup_{a\le t\le b}\abs{K_t}
 \le
 \abs{K_b}
 +\sup_{a\le t\le b}\Bigl|\int_t^b r_s\dd s\Bigr|
 \le 2R_{a,b},
\]
\eqref{eq:shifted-origin} gives
\[
 \int_a^b\norm{\widehat\Psi_s(0,0)}_{L^2}^2\dd s
 \le 4L^2(b-a)\E R_{a,b}^2,
 \qquad
 \E\abs{\xi+K_b}^2
 \le 2\E\abs{\xi}^2+2\E R_{a,b}^2.
\]
Finally, since \(P_t=\widehat P_t-K_t\) and
\(\E\sup_{a\le t\le b}\abs{K_t}^2\le4\E R_{a,b}^2\),
\eqref{eq:shifted-final-estimate} yields
\eqref{eq:operator-cancellation}. The preceding estimates also show that
\(C_{L,b-a}\) is locally bounded in \(b-a\).
\end{proof}

\subsection{Flow Property and Moment Estimates}

We next establish two basic properties of the averaged system: the restart
flow property and uniform moment estimates. These estimates will be used
repeatedly in the subsequent analysis.

\begin{lemma}\label{lem:restart-flow}
Assume {\rm(H3)}. Let \(s\in[0,T]\) and
\(\xi\in L^2(\Ff_s;\R^n)\), and let
\((\bar X^{s,\xi},\bar U^{s,\xi},\bar V^{s,\xi})\) be the averaged solution
started from \((s,\xi)\). Then, for every \(t\in[s,T]\),
\begin{equation}\label{eq:restart-flow}
 \bar U_t^{s,\xi}
 =
 \bar U_t^{\,t,\bar X_t^{s,\xi}}
 \quad\text{a.s.}
\end{equation}
\end{lemma}

\begin{proof}
Fix \(t\in[s,T]\). On \([t,T]\), the restricted process $ \bigl(
 \bar X_q^{s,\xi},
 \bar U_q^{s,\xi},
 \bar V_q^{s,\xi}
 \bigr)_{t\le q\le T}$ satisfies
\begin{equation}\label{eq:restricted-averaged-system}
\left\{
\begin{aligned}
\dd\bar X_q^{s,\xi}
 &=
 \bar b\bigl(
 \bar X_q^{s,\xi},
 \bar U_q^{s,\xi},
 \Law(\bar X_q^{s,\xi},\bar U_q^{s,\xi})
 \bigr)\dd q
 +\sigma\bigl(
 \bar X_q^{s,\xi},
 \bar U_q^{s,\xi},
 \Law(\bar X_q^{s,\xi},\bar U_q^{s,\xi})
 \bigr)\dd W_q^1,\\
\dd\bar U_q^{s,\xi}
 &=
 -\bar F\bigl(
 \bar X_q^{s,\xi},
 \bar U_q^{s,\xi},
 \bar V_q^{s,\xi},
 \Law(\bar X_q^{s,\xi},\bar U_q^{s,\xi})
 \bigr)\dd q
 +\bar V_q^{s,\xi}\dd W_q^1,
 \qquad
 \bar U_T^{s,\xi}
 =
 \beta\bigl(
 \bar X_T^{s,\xi},
 \Law(\bar X_T^{s,\xi})
 \bigr).
\end{aligned}
\right.
\end{equation}
Thus it satisfies the averaged equation on \([t,T]\) with initial variable
\(\bar X_t^{s,\xi}\). Notice that all law terms in
\eqref{eq:restricted-averaged-system} are the unconditional laws of the
restricted solution itself, as required by the restarted McKean--Vlasov
equation. Moreover, it is adapted to
\[
 \bigl(
 \Ff_t\vee
 \sigma(W_r^1-W_t^1:t\le r\le q)
 \bigr)^{\Pp},
 \qquad q\in[t,T].
\]
Hence, by uniqueness in {\rm(H3)},
\[
 \bar U_q^{s,\xi}
 =
 \bar U_q^{\,t,\bar X_t^{s,\xi}},
 \qquad q\in[t,T],
 \quad\text{a.s.}
\]
Taking \(q=t\) gives \eqref{eq:restart-flow}. 

\end{proof}

\begin{lemma}\label{lem:restart-origin-bound}
Under {\rm(H1)--(H3)},
\begin{equation}\label{eq:restart-origin-bound}
 \sup_{0\le s\le T}
 \norm{\bar U_s^{s,0}}_{L^2}
 \le C_T.
\end{equation}
Consequently,
\begin{equation}\label{eq:restart-linear-growth}
 \norm{\bar U_s^{s,\xi}}_{L^2}
 \le C_T\bigl(1+\norm{\xi}_{L^2}\bigr),
 \qquad
 s\in[0,T].
\end{equation}
\end{lemma}

\begin{proof}
Fix \(0\le a<b\le T\) with \(b-a\le1\). Let
\((\bar X^{a,0},\bar U^{a,0},\bar V^{a,0})\)
denote the averaged solution on \([a,T]\) with deterministic initial
condition \(0\) at time \(a\). On \([a,b]\), abbreviate its restriction by
\((X,U,V)\).

The forward equation, the Burkholder--Davis--Gundy inequality, and the
growth estimates in Lemma~\ref{lem:averaged-coefficients} give
\begin{align*}
 \E\sup_{a\le t\le b}\abs{X_t}^2
 &\le
 C\E\Bigl(\int_a^b
 \abs{\bar b(X_t,U_t,\Law(X_t,U_t))}\dd t\Bigr)^2
 +C\E\int_a^b
 \abs{\sigma(X_t,U_t,\Law(X_t,U_t))}^2\dd t
\\
 &\le
 C(b-a)\Bigl(
 1+\E\sup_{a\le t\le b}\abs{X_t}^2
 +\E\sup_{a\le t\le b}\abs{U_t}^2
 \Bigr).
\end{align*}
Choose \(h_0\in(0,1]\) sufficiently small. Then, for \(b-a\le h_0\),
\begin{equation}\label{eq:restart-growth-forward}
 \E\sup_{a\le t\le b}\abs{X_t}^2
 \le
 C(b-a)\Bigl(
 1+\E\sup_{a\le t\le b}\abs{U_t}^2
 \Bigr).
\end{equation}

By Lemma~\ref{lem:restart-flow},
$
 U_b=\bar U_b^{\,b,X_b}.
$
Hence, by \eqref{eq:restart-lipschitz},
\[
 \E\abs{U_b}^2
 \le
 2L_{\rm rst}^2\E\abs{X_b}^2
 +2\norm{\bar U_b^{b,0}}_{L^2}^2.
\]
To estimate the BSDE on \([a,b]\), for \(t\in[a,b]\),
\(p\in L^2(\Ff_t;\R^p)\), and
\(q\in L^2(\Ff_t;\R^{p\times d_1})\), define
\[
 \Psi_t(p,q)
 :=
 \bar F(X_t,p,q,\Law(X_t,p))
 -\bar F(X_t,0,0,\Law(X_t,0)),
 \qquad
 r_t:=\bar F(X_t,0,0,\Law(X_t,0)).
\]
Using the coupling induced by
\(\bigl((X_t,p),(X_t,p')\bigr)\), we have
\[
 W_2^2(\Law(X_t,p),\Law(X_t,p'))
 \le
 \E\abs{p-p'}^2.
\]
Lemma~\ref{lem:averaged-coefficients} shows that \(\Psi_t(0,0)=0\),
that \eqref{eq:operator-lipschitz} holds with a uniform constant, and that
\[
 \E\abs{r_t}^2\le C(1+\E\abs{X_t}^2).
\]
Therefore, Lemma~\ref{lem:operator-cancellation} and Cauchy--Schwarz inequality yield
\begin{align*}
 &\E\sup_{a\le t\le b}\abs{U_t}^2
 +\E\int_a^b\abs{V_t}^2\dd t
\le
 C\E\abs{U_b}^2
 +C\E\sup_{a\le r\le b}
 \Bigl|\int_r^b r_t\dd t\Bigr|^2
\\
 &\quad\le
 C\Bigl[
 \E\abs{U_b}^2
 +(b-a)\int_a^b\E\abs{r_t}^2\dd t
 \Bigr]
\le
 C\Bigl[
 1+\E\sup_{a\le t\le b}\abs{X_t}^2
 +\norm{\bar U_b^{b,0}}_{L^2}^2
 \Bigr].
\end{align*}
Combining this estimate with \eqref{eq:restart-growth-forward}, and
decreasing \(h_0\) if necessary, gives
\[
 \E\sup_{a\le t\le b}\abs{X_t}^2
 +\E\sup_{a\le t\le b}\abs{U_t}^2
 +\E\int_a^b\abs{V_t}^2\dd t
 \le
 C\Bigl(
 1+\norm{\bar U_b^{b,0}}_{L^2}^2
 \Bigr),
 \qquad b-a\le h_0.
\]
Since \(U_a=\bar U_a^{a,0}\),
\[
 \norm{\bar U_a^{a,0}}_{L^2}^2
 \le
 C\bigl(
 1+\norm{\bar U_b^{b,0}}_{L^2}^2
 \bigr).
\]
Partition \([a,T]\) into at most \(1+T/h_0\) subintervals and iterate this
inequality backward. At the terminal time,
\[
 \bar U_T^{T,0}=\beta(0,\delta_0),
\]
which is finite by {\rm(H1)}. This proves
\eqref{eq:restart-origin-bound}. Finally,
\eqref{eq:restart-lipschitz} gives
\[
 \norm{\bar U_s^{s,\xi}}_{L^2}
 \le
 L_{\rm rst}\norm{\xi}_{L^2}
 +\norm{\bar U_s^{s,0}}_{L^2},
\]
and hence \eqref{eq:restart-linear-growth}.
\end{proof}

\begin{lemma}\label{lem:averaged-moment}
Under {\rm(H1)--(H3)}, for every \(s\in[0,T]\) and
\(\xi\in L^2(\Ff_s;\R^n)\),
\begin{align}
 &\E\sup_{s\le t\le T}\abs{\bar X_t^{s,\xi}}^2
 +\E\sup_{s\le t\le T}\abs{\bar U_t^{s,\xi}}^2
 +\E\int_s^T\abs{\bar V_t^{s,\xi}}^2\dd t
 \le
 C_T\bigl(1+\E\abs\xi^2\bigr).
 \label{eq:averaged-moment}
\end{align}
\end{lemma}

\begin{proof}
For each fixed \(t\in[s,T]\), Lemma~\ref{lem:restart-flow},
\eqref{eq:restart-lipschitz}, and
\eqref{eq:restart-origin-bound} imply
\begin{align*}
 \E\abs{\bar U_t^{s,\xi}}^2
 &=
 \E\abs{\bar U_t^{\,t,\bar X_t^{s,\xi}}}^2
 \le
 C_T\bigl(1+\E\abs{\bar X_t^{s,\xi}}^2\bigr).
\end{align*}
The forward SDE estimate and the Lipschitz and growth properties of
\(\bar b\) and \(\sigma\) give, for \(q\in[s,T]\),
\begin{align*}
 \E\sup_{s\le t\le q}\abs{\bar X_t^{s,\xi}}^2
 &\le
 C\bigl(1+\E\abs\xi^2\bigr)
 +C\int_s^q
 \bigl(
 \E\abs{\bar X_t^{s,\xi}}^2
 +\E\abs{\bar U_t^{s,\xi}}^2
 \bigr)\dd t
\\
 &\le
 C_T\bigl(1+\E\abs\xi^2\bigr)
 +C_T\int_s^q
 \E\sup_{s\le r\le t}\abs{\bar X_r^{s,\xi}}^2\dd t.
\end{align*}
Gronwall's inequality yields
\begin{equation}\label{eq:averaged-X-moment}
 \E\sup_{s\le t\le T}\abs{\bar X_t^{s,\xi}}^2
 \le C_T(1+\E\abs\xi^2).
\end{equation}

For the backward equation, use on \([s,T]\) the \(L^2\)-operator
decomposition from the proof of Lemma~\ref{lem:restart-origin-bound}.
The terminal condition satisfies
\[
 \E\abs{\beta(\bar X_T^{s,\xi},\Law(\bar X_T^{s,\xi}))}^2
 \le C\bigl(1+\E\abs{\bar X_T^{s,\xi}}^2\bigr),
\]
and
\[
 \E\abs{
 \bar F(\bar X_t^{s,\xi},0,0,\Law(\bar X_t^{s,\xi},0))
 }^2
 \le C\bigl(1+\E\abs{\bar X_t^{s,\xi}}^2\bigr).
\]
Lemma~\ref{lem:operator-cancellation} and
\eqref{eq:averaged-X-moment} therefore give
\[
 \E\sup_{s\le t\le T}\abs{\bar U_t^{s,\xi}}^2
 +\E\int_s^T\abs{\bar V_t^{s,\xi}}^2\dd t
 \le C_T(1+\E\abs\xi^2).
\]
Together with \eqref{eq:averaged-X-moment}, this proves the lemma.
\end{proof}

\subsection{Regularity and Stability of the Restarted Averaged System}
We next establish the integrated time regularity of the averaged solution and
the stability of restarted averaged flows with respect to their initial data.
These estimates will be used in the local-to-global error analysis in section \ref{sec:local-to-global}.
\begin{lemma}\label{lem:integrated-regularity}
Assume {\rm(H1)--(H3)}. For every \(s\in[0,T]\),
\(\xi\in L^2(\Ff_s;\R^n)\), and \(\delta\in(0,1]\), let
$
 s=r_0<r_1<\cdots<r_N=T
$
be any deterministic partition satisfying
$
 \max_k(r_{k+1}-r_k)\le\delta,
$
and set \(\pi_\delta(t)=r_k\) on \([r_k,r_{k+1})\). Then
\begin{align}
 &\int_s^T
 \E\bigl[
 \abs{\bar X_t^{s,\xi}
 -\bar X_{\pi_\delta(t)}^{s,\xi}}^2
 +\abs{\bar U_t^{s,\xi}
 -\bar U_{\pi_\delta(t)}^{s,\xi}}^2
 \bigr]\dd t
 \le
 C_T\delta\bigl(1+\E\abs\xi^2\bigr).
 \label{eq:integrated-regularity}
\end{align}
\end{lemma}

\begin{proof}
For \(t\in[r_k,r_{k+1}]\),
\[
 \bar X_t^{s,\xi}-\bar X_{r_k}^{s,\xi}
 =
 \int_{r_k}^t
 \bar b\bigl(
 \bar X_q^{s,\xi},
 \bar U_q^{s,\xi},
 \Law(\bar X_q^{s,\xi},\bar U_q^{s,\xi})
 \bigr)\dd q
 +\int_{r_k}^t
 \sigma\bigl(
 \bar X_q^{s,\xi},
 \bar U_q^{s,\xi},
 \Law(\bar X_q^{s,\xi},\bar U_q^{s,\xi})
 \bigr)
 \dd W_q^1.
\]
Cauchy--Schwarz and It\^o's isometry give
\[
 \begin{aligned}
 \E\abs{\bar X_t^{s,\xi}-\bar X_{r_k}^{s,\xi}}^2
 &\le
 2(t-r_k)\int_{r_k}^t
 \E\abs{
 \bar b\bigl(
 \bar X_q^{s,\xi},
 \bar U_q^{s,\xi},
 \Law(\bar X_q^{s,\xi},\bar U_q^{s,\xi})
 \bigr)
 }^2\dd q\\
 &\quad+
 2\int_{r_k}^t
 \E\abs{
 \sigma\bigl(
 \bar X_q^{s,\xi},
 \bar U_q^{s,\xi},
 \Law(\bar X_q^{s,\xi},\bar U_q^{s,\xi})
 \bigr)
 }^2\dd q.
 \end{aligned}
\]
Integrating in \(t\) and using Fubini,
\[
 \begin{aligned}
 \int_{r_k}^{r_{k+1}}
 \E\abs{\bar X_t^{s,\xi}-\bar X_{r_k}^{s,\xi}}^2\dd t\le
 2\delta^2\int_{r_k}^{r_{k+1}}
 \E\abs{
 \bar b\bigl(
 \bar X_q^{s,\xi},
 \bar U_q^{s,\xi},
 \Law(\bar X_q^{s,\xi},\bar U_q^{s,\xi})
 \bigr)
 }^2\dd q\\
 \qquad+
 2\delta\int_{r_k}^{r_{k+1}}
 \E\abs{
 \sigma\bigl(
 \bar X_q^{s,\xi},
 \bar U_q^{s,\xi},
 \Law(\bar X_q^{s,\xi},\bar U_q^{s,\xi})
 \bigr)
 }^2\dd q.
 \end{aligned}
\]
Summing in \(k\), using \(\delta<1\), the linear-growth estimates,
and Lemma~\ref{lem:averaged-moment}, we obtain
\begin{equation}\label{eq:integrated-X-regularity}
 \int_s^T
 \E\abs{
 \bar X_t^{s,\xi}
 -\bar X_{\pi_\delta(t)}^{s,\xi}
 }^2\dd t
 \le C_T\delta(1+\E\abs\xi^2).
\end{equation}

For the backward component, the equation gives
\[
 \bar U_t^{s,\xi}-\bar U_{r_k}^{s,\xi}
 =
 -\int_{r_k}^t
 \bar F\bigl(
 \bar X_q^{s,\xi},
 \bar U_q^{s,\xi},
 \bar V_q^{s,\xi},
 \Law(\bar X_q^{s,\xi},\bar U_q^{s,\xi})
 \bigr)\dd q
 +\int_{r_k}^t\bar V_q^{s,\xi}\dd W_q^1.
\]
Cauchy--Schwarz for the time integral, It\^o's isometry for the
stochastic integral, and Fubini's theorem give
\[
 \begin{aligned}
 &\int_{r_k}^{r_{k+1}}
 \E\abs{
 \bar U_t^{s,\xi}-\bar U_{r_k}^{s,\xi}
 }^2\dd t\\
 &\quad\le
 2\delta^2\int_{r_k}^{r_{k+1}}
 \E\abs{
 \bar F\bigl(
 \bar X_q^{s,\xi},
 \bar U_q^{s,\xi},
 \bar V_q^{s,\xi},
 \Law(\bar X_q^{s,\xi},\bar U_q^{s,\xi})
 \bigr)
 }^2\dd q
 +2\delta\int_{r_k}^{r_{k+1}}
 \E\abs{\bar V_q^{s,\xi}}^2\dd q.
 \end{aligned}
\]
The linear growth of \(\bar F\) and
\eqref{eq:averaged-moment} imply
\begin{equation}\label{eq:integrated-U-regularity}
 \int_s^T
 \E\abs{
 \bar U_t^{s,\xi}
 -\bar U_{\pi_\delta(t)}^{s,\xi}
 }^2\dd t
 \le C_T\delta(1+\E\abs\xi^2).
\end{equation}
Combining \eqref{eq:integrated-X-regularity} and
\eqref{eq:integrated-U-regularity} yields
\eqref{eq:integrated-regularity}. This completes the proof.
\end{proof}

\begin{lemma}
\label{lem:averaged-flow-stability}
Assume {\rm(H1)--(H3)}.
Let two averaged solutions be restarted at the same time \(s\), on the same
probability space and with the same \(W^1\), from
\(\xi,\xi'\in L^2(\Ff_s;\R^n)\).  Then
\begin{align}
 &\E\sup_{s\le t\le T}
 \abs{\bar X_t^{s,\xi}-\bar X_t^{s,\xi'}}^2
 +\E\sup_{s\le t\le T}
 \abs{\bar U_t^{s,\xi}-\bar U_t^{s,\xi'}}^2
+
 \E\int_s^T
 \abs{\bar V_t^{s,\xi}-\bar V_t^{s,\xi'}}^2\dd t
 \le
 C_T\E\abs{\xi-\xi'}^2.
 \label{eq:averaged-flow-stability}
\end{align}
\end{lemma}

\begin{proof}
Put
\[
 \Delta X_t=\bar X_t^{s,\xi}-\bar X_t^{s,\xi'},\quad
 \Delta U_t=\bar U_t^{s,\xi}-\bar U_t^{s,\xi'},\quad
 \Delta V_t=\bar V_t^{s,\xi}-\bar V_t^{s,\xi'}.
\]
For every fixed \(t\), Lemma~\ref{lem:restart-flow} and
\eqref{eq:restart-lipschitz} give
\begin{equation}\label{eq:flow-stability-U-by-X}
 \E\abs{\Delta U_t}^2
 \le L_{\rm rst}^2\E\abs{\Delta X_t}^2.
\end{equation}
Hence
\begin{align}
 &W_2^2\bigl(
 \Law(\bar X_t^{s,\xi},\bar U_t^{s,\xi}),
 \Law(\bar X_t^{s,\xi'},\bar U_t^{s,\xi'})
 \bigr)
\le
 \E\abs{\Delta X_t}^2+\E\abs{\Delta U_t}^2
 \le
 (1+L_{\rm rst}^2)\E\abs{\Delta X_t}^2.
 \label{eq:flow-stability-law}
\end{align}
Using \eqref{eq:averaged-b-lipschitz},
\eqref{eq:H1-sigma}, the Burkholder--Davis--Gundy inequality, and
\eqref{eq:flow-stability-U-by-X}--\eqref{eq:flow-stability-law}, we obtain
for \(q\in[s,T]\),
\begin{align*}
 \E\sup_{s\le t\le q}\abs{\Delta X_t}^2
 &\le
 C\E\abs{\xi-\xi'}^2
 +C\int_s^q
 \E\bigl(\abs{\Delta X_t}^2+\abs{\Delta U_t}^2\bigr)\dd t
\\
 &\qquad\qquad\qquad\qquad+
 C\int_s^q
 W_2^2\bigl(
 \Law(\bar X_t^{s,\xi},\bar U_t^{s,\xi}),
 \Law(\bar X_t^{s,\xi'},\bar U_t^{s,\xi'})
 \bigr)\dd t
\\
 &\le
 C\E\abs{\xi-\xi'}^2
 +C\int_s^q
 \E\sup_{s\le r\le t}\abs{\Delta X_r}^2\dd t.
\end{align*}
Gronwall's inequality gives
\begin{equation}\label{eq:flow-stability-X}
 \E\sup_{s\le t\le T}\abs{\Delta X_t}^2
 \le C_T\E\abs{\xi-\xi'}^2.
\end{equation}

Now we derive the backward stability estimate. By
\eqref{eq:H1-beta},
\begin{align}
 &\E\bigl|
 \beta(\bar X_T^{s,\xi},\Law(\bar X_T^{s,\xi}))
 -\beta(\bar X_T^{s,\xi'},\Law(\bar X_T^{s,\xi'}))
 \bigr|^2
\le4L^2\E\abs{\Delta X_T}^2.
 \label{eq:flow-stability-terminal}
\end{align}

For \(P\in L^2(\Ff_t;\R^p)\) and
\(Q\in L^2(\Ff_t;\R^{p\times d_1})\), define
\[
\begin{aligned}
 \Psi_t(P,Q)
 &:=
 \bar F\bigl(
 \bar X_t^{s,\xi},
 \bar U_t^{s,\xi'}+P,
 \bar V_t^{s,\xi'}+Q,
 \Law(\bar X_t^{s,\xi},\bar U_t^{s,\xi'}+P)
 \bigr)
 -
 \bar F\bigl(
 \bar X_t^{s,\xi},
 \bar U_t^{s,\xi'},
 \bar V_t^{s,\xi'},
 \Law(\bar X_t^{s,\xi},\bar U_t^{s,\xi'})
 \bigr),
\end{aligned}
\]
and
\[
\begin{aligned}
 r_t
 &:=
 \bar F\bigl(
 \bar X_t^{s,\xi},
 \bar U_t^{s,\xi'},
 \bar V_t^{s,\xi'},
 \Law(\bar X_t^{s,\xi},\bar U_t^{s,\xi'})
 \bigr)
 -
 \bar F\bigl(
 \bar X_t^{s,\xi'},
 \bar U_t^{s,\xi'},
 \bar V_t^{s,\xi'},
 \Law(\bar X_t^{s,\xi'},\bar U_t^{s,\xi'})
 \bigr).
\end{aligned}
\]
Then the difference BSDE has the form required in
Lemma~\ref{lem:operator-cancellation}, with
\[
 \Delta U_t
 =
 \Delta U_T
 +\int_t^T
 \bigl[\Psi_q(\Delta U_q,\Delta V_q)+r_q\bigr]\dd q
 -\int_t^T\Delta V_q\dd W_q^1.
\]
By \eqref{eq:averaged-F-lipschitz} and the coupling inequality,
\[
 \norm{\Psi_t(P,Q)-\Psi_t(P',Q')}_{L^2}
 \le
 C\bigl(
 \norm{P-P'}_{L^2}
 +\norm{Q-Q'}_{L^2}
 \bigr),
\]
and
\[
 \E\abs{r_t}^2
 \le
 C\E\abs{\Delta X_t}^2.
\]
Hence Lemma~\ref{lem:operator-cancellation},
\eqref{eq:flow-stability-terminal}, and
\eqref{eq:flow-stability-X} yield
\begin{align}
 &\E\sup_{s\le t\le T}\abs{\Delta U_t}^2
 +\E\int_s^T\abs{\Delta V_t}^2\dd t
 \le
 C_T\Big[
 \E\abs{\Delta X_T}^2
 +\E\sup_{s\le t\le T}
 \big|\int_t^T r_q\dd q\big|^2
 \Big]
 \notag\\
 &\quad\le
 C_T\Big[
 \E\abs{\Delta X_T}^2
 +\int_s^T\E\abs{\Delta X_q}^2\dd q
 \Big]
 \le
 C_T\E\abs{\xi-\xi'}^2.
 \label{eq:flow-stability-backward}
\end{align}
Combining \eqref{eq:flow-stability-X} and
\eqref{eq:flow-stability-backward} yields
\eqref{eq:averaged-flow-stability}.
\end{proof}

\section{Proof of the Global-In-Time Averaging Principle}
\label{sec:proof}

\subsection{Macroscopic and Microscopic Localization}

Before entering the proof, we first outline the main strategy. In contrast to the classical Khasminskii scheme, the argument relies on two distinct time discretizations. A macroscopic partition, chosen independently of \(\eps\), is used to localize the forward--backward coupling. Within each macroscopic interval, we then introduce a uniform microscopic mesh with step size of order \(\eps\). Along this microscopic mesh, the conditional law of the fast component is shown to track the successive invariant kernels, while a time-inhomogeneous Gordin decomposition controls the accumulated centered fluctuations over the whole macroscopic interval. The following subsections develop these ingredients and assemble them into the proof.

Let \(h_0\in(0,1]\), to be chosen after the local estimates have been
derived.  Choose an integer \(M\), independent of \(\eps\), such that
\begin{equation}\label{eq:macro-partition}
 0=t_0<t_1<\cdots<t_M=T,\qquad
 t_i=i\ell,\qquad
 \ell:=\frac{T}{M}<h_0.
\end{equation}
Write \(I_i=[t_i,t_{i+1}]\).  At the macroscopic grid points define
\begin{equation}\label{eq:grid-defect}
 R_{t_i}^\eps
 :=
 U_{t_i}^\eps-\bar U_{t_i}^{\,t_i,X_{t_i}^\eps}.
\end{equation}
Only grid-point defects are used; no joint measurability in \(t\) of a
lifted decoupling operator is needed.  Since
\[
 \bar U_T^{\,T,\xi}=\beta(\xi,\Law(\xi)),
\]
the terminal conditions in the original and averaged systems give
\begin{equation*}
 R_T^\eps
 =
 \beta(X_T^\eps,\Law(X_T^\eps))
 -\beta(X_T^\eps,\Law(X_T^\eps))
 =0.
\end{equation*}

Fix \(i\in\{0,\ldots,M-1\}\).  Let
\[
 (\widetilde X^i,\widetilde U^i,\widetilde V^i)
 :=
 (\bar X^{\,t_i,X_{t_i}^\eps},
  \bar U^{\,t_i,X_{t_i}^\eps},
  \bar V^{\,t_i,X_{t_i}^\eps})
\]
be the averaged restart solution, and retain its restriction to \(I_i\).
Then
\begin{equation}\label{eq:local-initial-match}
 \widetilde X_{t_i}^i=X_{t_i}^\eps,\qquad
 \widetilde U_{t_i}^i=\bar U_{t_i}^{\,t_i,X_{t_i}^\eps},
\end{equation}
and Lemma~\ref{lem:restart-flow} gives, for every fixed \(t\in I_i\),
\begin{equation}\label{eq:local-restart-flow}
 \widetilde U_t^i
 =
 \bar U_t^{\,t,\widetilde X_t^i}
 \quad\text{a.s.}
\end{equation}
Lemma~\ref{lem:averaged-moment} gives
\begin{align}
 \E\sup_{t\in I_i}
 \bigl(\abs{\widetilde X_t^i}^2+\abs{\widetilde U_t^i}^2\bigr)
 +\E\int_{I_i}\abs{\widetilde V_t^i}^2\dd t
 &\le C_T\bigl(1+\E\abs{X_{t_i}^\eps}^2+\E\abs{Y_{t_i}^\eps}^2\bigr).
 \label{eq:local-averaged-moment}
\end{align}

For sufficiently small \(\eps\), assume \(\eps<\ell\wedge1\).  Let
\[
 N_i:=\lceil\frac{\ell}{\eps}\rceil,
 \qquad
 \delta_\eps:=\frac{\ell}{N_i},
 \qquad
 s_{i,k}:=t_i+k\delta_\eps,
 \quad0\le k\le N_i,
\]
and let
\[
 J_{i,k}:=[s_{i,k},s_{i,k+1}),\qquad
 \pi_\eps(t):=s_{i,k}\quad\text{for }t\in J_{i,k}.
\]
Then
\begin{equation}\label{eq:number-blocks}
 N_i\le1+\frac{\ell}{\eps},
 \qquad
 \frac\eps2\le\delta_\eps\le\eps,
 \qquad
 \abs{J_{i,k}}=\delta_\eps.
\end{equation}
Lemma~\ref{lem:integrated-regularity}, applied to the local averaged flow,
implies
\begin{align}
 \int_{I_i}
 \E\bigl[
 \abs{\widetilde X_t^i
 -\widetilde X_{\pi_\eps(t)}^i}^2
 +\abs{\widetilde U_t^i
 -\widetilde U_{\pi_\eps(t)}^i}^2
 \bigr]\dd t
 &\le C_T\delta_\eps
 \bigl(1+\E\abs{X_{t_i}^\eps}^2
 +\E\abs{Y_{t_i}^\eps}^2\bigr)
 \notag\\
 &\le C_T\eps
 \bigl(1+\E\abs{X_{t_i}^\eps}^2
 +\E\abs{Y_{t_i}^\eps}^2\bigr).
 \label{eq:local-freezing-modulus}
\end{align}

For \(0\le k<N_i\), put
\[
 \eta_{i,k}
 :=
 \Law(\widetilde X_{s_{i,k}}^i,\widetilde U_{s_{i,k}}^i).
\]
On \(I_i\), define a continuous process
\(\widehat Y^{\eps,i}\) by
\(\widehat Y_{t_i}^{\eps,i}=Y_{t_i}^\eps\) and, on each
\(J_{i,k}\),
\begin{equation}\label{auxiliary-0}
\begin{aligned}
\dd\widehat Y_t^{\eps,i}
 &=
 \frac1\eps
 h\bigl(
 \widetilde X_{s_{i,k}}^i,\widehat Y_t^{\eps,i},\widetilde U_{s_{i,k}}^i,
 \Law((\widetilde X_{s_{i,k}}^i,\widetilde U_{s_{i,k}}^i),\widehat Y_t^{\eps,i})
 \bigr)\dd t
\\
 &\quad+
 \frac1{\sqrt\eps}
 g\bigl(
 \widetilde X_{s_{i,k}}^i,\widehat Y_t^{\eps,i},\widetilde U_{s_{i,k}}^i,
 \Law((\widetilde X_{s_{i,k}}^i,\widetilde U_{s_{i,k}}^i),\widehat Y_t^{\eps,i})
 \bigr)\dd W_t^2.
\end{aligned}
\end{equation}

We begin by deriving estimates for this auxiliary process and quantifying its deviation from the original one.

\begin{lemma}
\label{lem:auxiliary-fast-comparison}
Assume {\rm(H1)--(H3)}. Then, for every \(0\le i<M\),
\begin{equation}\label{eq:Yhat-moment}
 \sup_{t\in I_i}\E\abs{\widehat Y_t^{\eps,i}}^2
 \le C\bigl(
 1+\E\abs{X_{t_i}^\eps}^2+\E\abs{Y_{t_i}^\eps}^2
 \bigr),
\end{equation}
and
\begin{equation}\label{eq:fast-comparison-integrated}
 \int_{I_i}\E\abs{Y_t^\eps-\widehat Y_t^{\eps,i}}^2\dd t
 \le
 C\ell\bigl(
 \E\sup_{t\in I_i}\abs{X_t^\eps-\widetilde X_t^i}^2
 +\E\sup_{t\in I_i}\abs{U_t^\eps-\widetilde U_t^i}^2
 \bigr)
 +C\delta_\eps
 \bigl(
 1+\E\abs{X_{t_i}^\eps}^2+\E\abs{Y_{t_i}^\eps}^2
 \bigr).
\end{equation}
\end{lemma}

\begin{proof}
For \(t\in J_{i,k}\), Lemma~\ref{lem:frozen-lyapunov}, after the time
rescaling \(r=t/\eps\), gives
\begin{align*}
 \frac{\dd}{\dd t}\E\abs{\widehat Y_t^{\eps,i}}^2
 &\le
 -\frac{c_0}{\eps}\E\abs{\widehat Y_t^{\eps,i}}^2
 +\frac C\eps
 \bigl(
 1+\E\abs{\widetilde X_{s_{i,k}}^i}^2
 +\E\abs{\widetilde U_{s_{i,k}}^i}^2
 \bigr).
\end{align*}
By \eqref{eq:local-averaged-moment},
\[
 1+\E\abs{\widetilde X_{s_{i,k}}^i}^2
 +\E\abs{\widetilde U_{s_{i,k}}^i}^2
 \le
 C\bigl(
 1+\E\abs{X_{t_i}^\eps}^2
 +\E\abs{Y_{t_i}^\eps}^2
 \bigr),
\]
uniformly in \(k\). Since \(\widehat Y^{\eps,i}\) is continuous at the microscopic grid
points, the variation-of-constants estimate can be iterated across the
successive blocks. The exponential factors then concatenate, yielding,
for every \(t\in I_i\),
\begin{align*}
 \E\abs{\widehat Y_t^{\eps,i}}^2
 &\le
 \ee^{-c_0(t-t_i)/\eps}\E\abs{Y_{t_i}^\eps}^2
 +\frac{C\bigl(1+\E\abs{X_{t_i}^\eps}^2+\E\abs{Y_{t_i}^\eps}^2\bigr)}{\eps}
 \int_{t_i}^t\ee^{-c_0(t-r)/\eps}\dd r
\\
 &\le C\bigl(1+\E\abs{X_{t_i}^\eps}^2+\E\abs{Y_{t_i}^\eps}^2\bigr).
\end{align*}
This proves \eqref{eq:Yhat-moment}.

Set
$
 D_t^{\eps,i}:=Y_t^\eps-\widehat Y_t^{\eps,i}.
$
For \(t\in J_{i,k}\), It\^o's formula and
\eqref{eq:joint-dissipativity} yield
\begin{align}
 \frac{\dd}{\dd t}\E\abs{D_t^{\eps,i}}^2
 &\le
 -\frac{\kappa}{\eps}\E\abs{D_t^{\eps,i}}^2
 +\frac{K_1}{\eps}
 \E\bigl[
 \abs{X_t^\eps-\widetilde X_{s_{i,k}}^i}^2
 +\abs{U_t^\eps-\widetilde U_{s_{i,k}}^i}^2
 \bigr]
 \notag\\
 &\quad+
 \frac{K_2}{\eps}
 W_2^2\bigl(
 \Law((X_t^\eps,U_t^\eps),Y_t^\eps),
 \Law((\widetilde X_{s_{i,k}}^i,\widetilde U_{s_{i,k}}^i),\widehat Y_t^{\eps,i})
 \bigr).
 \label{eq:fast-comparison-first}
\end{align}
Using \(\alpha=\kappa-K_2\) in
\eqref{eq:fast-comparison-first} gives
\begin{align*}
 \frac{\dd}{\dd t}\E\abs{D_t^{\eps,i}}^2
 &\le
 -\frac{\alpha}{\eps}\E\abs{D_t^{\eps,i}}^2
 +\frac C\eps
 \E\bigl[
 \abs{X_t^\eps-\widetilde X_{\pi_\eps(t)}^i}^2
 +\abs{U_t^\eps-\widetilde U_{\pi_\eps(t)}^i}^2
 \bigr].
\end{align*}
Since \(D_{t_i}^{\eps,i}=0\), variation of constants gives
\[
 \E\abs{D_t^{\eps,i}}^2
 \le
 \frac C\eps\int_{t_i}^t
 \ee^{-\alpha(t-r)/\eps}
 \E\bigl[
 \abs{X_r^\eps-\widetilde X_{\pi_\eps(r)}^i}^2
 +\abs{U_r^\eps-\widetilde U_{\pi_\eps(r)}^i}^2
 \bigr]\dd r.
\]
Integrating in \(t\), applying Tonelli's theorem, and using
$
 \frac1\eps\int_r^{t_{i+1}}
 \ee^{-\alpha(t-r)/\eps}\dd t\le\frac1\alpha,
$
we obtain
\begin{align}
 \int_{I_i}\E\abs{D_t^{\eps,i}}^2\dd t
 &\le C\int_{I_i}
 \E\bigl[
 \abs{X_t^\eps-\widetilde X_{\pi_\eps(t)}^i}^2
 +\abs{U_t^\eps-\widetilde U_{\pi_\eps(t)}^i}^2
 \bigr]\dd t.
 \label{eq:fast-comparison-integrated-first}
\end{align}
For \(t\in J_{i,k}\),
\begin{align*}
 \E\bigl[
 \abs{X_t^\eps-\widetilde X_{\pi_\eps(t)}^i}^2
 +\abs{U_t^\eps-\widetilde U_{\pi_\eps(t)}^i}^2
 \bigr]
 &\le
 2\E\bigl[
 \abs{X_t^\eps-\widetilde X_t^i}^2
 +\abs{U_t^\eps-\widetilde U_t^i}^2
 \bigr]
\\
 &\quad+
 2\E\bigl[
 \abs{\widetilde X_t^i
 -\widetilde X_{\pi_\eps(t)}^i}^2
 +\abs{\widetilde U_t^i
 -\widetilde U_{\pi_\eps(t)}^i}^2
 \bigr].
\end{align*}
Substitution into \eqref{eq:fast-comparison-integrated-first}, followed by
\eqref{eq:local-freezing-modulus}, proves
\eqref{eq:fast-comparison-integrated}.
\end{proof}

\subsection{Centered Coefficients and Conditional Equilibrium Tracking}

In this subsection, we introduce the centered coefficients, whose estimates
form a key ingredient in the proof of the averaging principle. We first
establish pointwise second-moment bounds for these coefficients. We then
estimate the discrepancy between the conditional law of the auxiliary fast
process and the corresponding conditional invariant kernel \eqref{eq:integrated-equilibrium-tracking}. This conditional
equilibrium tracking estimate will be used in the next subsection to control
the conditional bias of the centered coefficients.

For \(t\in J_{i,k}\), abbreviate the auxiliary joint law by
\begin{equation*}
 \widehat\rho_t^{i,k}
 :=
 \Law((\widetilde X_{s_{i,k}}^i,\widetilde U_{s_{i,k}}^i),\widehat Y_t^{\eps,i}).
\end{equation*}
Define
\begin{align*}
 \mathfrak b_t^{i,k}
 &:=
 b(\widetilde X_{s_{i,k}}^i,\widehat Y_t^{\eps,i},\widetilde U_{s_{i,k}}^i,
 \widehat\rho_t^{i,k})
 -\bar b(\widetilde X_{s_{i,k}}^i,\widetilde U_{s_{i,k}}^i,\eta_{i,k}),
 \\
 \mathfrak F_t^{i,k}
 &:=
 F(\widetilde X_{s_{i,k}}^i,\widehat Y_t^{\eps,i},\widetilde U_{s_{i,k}}^i,
 (\widetilde V_t^i,0),\widehat\rho_t^{i,k})
-
 \bar F(\widetilde X_{s_{i,k}}^i,\widetilde U_{s_{i,k}}^i,
 \widetilde V_t^i,\eta_{i,k}).
\end{align*}

\begin{lemma}
\label{lem:centered-coefficient-moment}
Assume {\rm(H1)--(H3)}. Then, for every \(0\le i<M\),
\(0\le k<N_i\), and almost every \(t\in J_{i,k}\),
\begin{equation}\label{eq:centered-second-moment}
 \E\abs{\mathfrak b_t^{i,k}}^2
 +\E\abs{\mathfrak F_t^{i,k}}^2
 \le
 C\bigl(
 1+\E\abs{X_{t_i}^\eps}^2
 +\E\abs{Y_{t_i}^\eps}^2
 \bigr).
\end{equation}
\end{lemma}

\begin{proof}
We first prove a pointwise \(L^2\) bound that does not require any
pointwise estimate for \(\widetilde V^i\).  By
\eqref{eq:averaged-F-definition},
\begin{align*}
 \mathfrak F_t^{i,k}
 &=
 \int_{\R^m}
 \bigl[
 F(\widetilde X_{s_{i,k}}^i,\widehat Y_t^{\eps,i},\widetilde U_{s_{i,k}}^i,
 (\widetilde V_t^i,0),\widehat\rho_t^{i,k})
 -F(\widetilde X_{s_{i,k}}^i,z,\widetilde U_{s_{i,k}}^i,
 (\widetilde V_t^i,0),\Pi^{\eta_{i,k}})
 \bigr]
 \nu^{\widetilde X_{s_{i,k}}^i,\widetilde U_{s_{i,k}}^i;\eta_{i,k}}(\dd z).
\end{align*}
The same value of \(\widetilde V_t^i\) occurs in both terms and cancels in
the Lipschitz estimate.  Hence
\[
 \abs{\mathfrak F_t^{i,k}}
 \le
 L\int_{\R^m}
 \abs{\widehat Y_t^{\eps,i}-z}
 \nu^{\widetilde X_{s_{i,k}}^i,\widetilde U_{s_{i,k}}^i;\eta_{i,k}}(\dd z)
 +LW_2(\widehat\rho_t^{i,k},\Pi^{\eta_{i,k}}).
\]
\begin{samepage}
Likewise, the definition of \(\bar b\) gives the exact representation
\begin{align*}
 \mathfrak b_t^{i,k}
 &=
 \int_{\R^m}\bigl[
 b(\widetilde X_{s_{i,k}}^i,\widehat Y_t^{\eps,i},\widetilde U_{s_{i,k}}^i,
 \widehat\rho_t^{i,k})
 -b(\widetilde X_{s_{i,k}}^i,z,\widetilde U_{s_{i,k}}^i,\Pi^{\eta_{i,k}})
 \bigr]
 \nu^{\widetilde X_{s_{i,k}}^i,\widetilde U_{s_{i,k}}^i;\eta_{i,k}}(\dd z),
\end{align*}
\end{samepage}
and hence
\[
 \abs{\mathfrak b_t^{i,k}}
 \le
 L\int_{\R^m}\abs{\widehat Y_t^{\eps,i}-z}
 \nu^{\widetilde X_{s_{i,k}}^i,\widetilde U_{s_{i,k}}^i;\eta_{i,k}}(\dd z)
 +LW_2(\widehat\rho_t^{i,k},\Pi^{\eta_{i,k}}).
\]
Equations
\eqref{eq:Yhat-moment},
\eqref{eq:conditional-invariant-moment},
\eqref{eq:joint-invariant-moment}, and
\eqref{eq:local-averaged-moment} give
\eqref{eq:centered-second-moment}.
\end{proof}

We now introduce the random environment with respect to which the subsequent
conditional argument will be carried out. Its choice is dictated by two
competing requirements: it must contain enough information to make the
unfrozen process \(\widetilde V^i\) measurable, while revealing no future
increment of \(W^2\), so that the conditional independence needed below is
preserved. To make the construction transparent, we formulate it on path
space and first introduce the following Polish space.
\[
 \mathsf K_i
 :=
 \R^m\times C(I_i;\slowspace)
 \times L^2(I_i;\R^{p\times d_1})
\]
and the random element
\begin{equation*}
 \mathbf E_i
 :=
 \bigl(
 Y_{t_i}^\eps,
 (\widetilde X_t^i,\widetilde U_t^i)_{t\in I_i},
 (\widetilde V_t^i)_{t\in I_i}
 \bigr),
 \qquad
 \mathcal K_i:=\sigma(\mathbf E_i).
\end{equation*}
Here and below, as usual, \(\widetilde V^i\) also denotes its
\(L^2(I_i;\R^{p\times d_1})\)-equivalence class.  Whenever pointwise values
are needed, we fix once and for all a jointly measurable representative;
such a choice may be obtained from Lebesgue differentiation of the local
averages on \(L^2(I_i)\).  Thus \(\widetilde V^i\) may be taken
\(\mathcal K_i\otimes\mathcal B(I_i)\)-measurable.

The random element \(\mathbf E_i\) is measurable with respect to
$
 \Ff_{t_i}\vee
 \sigma\bigl(W_t^1-W_{t_i}^1:\ t\in I_i\bigr).
$
Hence
\[
 \mathcal K_i
 \subset
 \Ff_{t_i}\vee
 \sigma\bigl(W_t^1-W_{t_i}^1:\ t\in I_i\bigr).
\]
Therefore,
$
 \sigma\bigl(W_t^2-W_{t_i}^2:\ t\in I_i\bigr)
$ is independent of $
 \mathcal K_i.$

Let \(\Lambda_i:=\Law(\mathbf E_i)\).  The path
\((\widehat Y_t^{\eps,i})_{t\in I_i}\) belongs to \(C(I_i;\R^m)\).
Since \(\mathsf K_i\) and \(C(I_i;\R^m)\) are Polish spaces, there exists a
regular conditional probability kernel
\[
 (e,A)\longmapsto Q_i^e(A),
 \qquad
 e\in\mathsf K_i,\quad
 A\in\mathcal B(C(I_i;\R^m)),
\]
such that
\[
 \Pp\bigl(\mathbf E_i\in B,
 (\widehat Y_t^{\eps,i})_{t\in I_i}\in A\bigr)
 =
 \int_B Q_i^e(A)\,\Lambda_i(\dd e),
 \qquad
 B\in\mathcal B(\mathsf K_i).
\]
Thus \(Q_i^e\) is a regular conditional law of the entire path
\((\widehat Y_t^{\eps,i})_{t\in I_i}\) given \(\mathbf E_i=e\).  For
\(t\in I_i\), set
\begin{equation*}
 \lambda_{i,t}^e
 :=
 \bigl(\gamma\mapsto\gamma(t)\bigr)_\#Q_i^e.
\end{equation*}
Equivalently,
\[
 \lambda_{i,t}^{\mathbf E_i}(A)
 =
 \E\big[
 \mathbf 1_{\{\widehat Y_t^{\eps,i}\in A\}}
 \,|\,\mathcal K_i
 \big]
 \quad\text{a.s.},
 \qquad A\in\mathcal B(\R^m).
\]
The kernel may be chosen jointly measurable in \((e,t)\). The conditional laws \(\lambda_{i,t}^e\) introduced below are used solely
for the conditional mixing argument and do not replace the unconditional
law appearing in the auxiliary equation \eqref{auxiliary-0}.

For
\[
 e=(y,(x(\cdot),u(\cdot)),v)\in\mathsf K_i,
\]
define
\[
 P_{i,k}(e)
 :=
 \bigl(x(s_{i,k}),u(s_{i,k})\bigr),
 \qquad 0\le k<N_i.
\]
Thus
\[
 P_{i,k}(\mathbf E_i)
 =
 (\widetilde X_{s_{i,k}}^i,\widetilde U_{s_{i,k}}^i)
 \qquad\text{a.s.}
\]
For \(0\le k<N_i\) and \(t\in[s_{i,k},s_{i,k+1}]\), set
\begin{equation}\label{eq:Dik-process-definition}
 D_{i,k}(t)^2
 :=
 \int_{\mathsf K_i}W_2^2\bigl(
 \lambda_{i,t}^e,
 \nu^{P_{i,k}(e);\eta_{i,k}}
 \bigr)\Lambda_i(\dd e).
\end{equation}

By definition, \(D_{i,k}(t)\) measures the averaged Wasserstein
discrepancy between the conditional law of the auxiliary fast process
and the corresponding conditional invariant kernel. The following lemma
shows that, over each macroscopic interval \(I_i\), both this conditional
discrepancy and the discrepancy between the auxiliary joint law
\(\widehat\rho_t^{i,k}\) and the corresponding joint equilibrium measure
\(\Pi^{\eta_{i,k}}\) admit a uniform integrated estimate.

\begin{lemma}
\label{lem:conditional-equilibrium-tracking}
Assume {\rm(H1)--(H3)}. Then, for every \(0\le i<M\),
\begin{equation}\label{eq:integrated-equilibrium-tracking}
\begin{aligned}
\sum_{k=0}^{N_i-1}\int_{J_{i,k}}
\bigl[
D_{i,k}(t)^2
+
W_2^2\bigl(
\widehat\rho_t^{i,k},\Pi^{\eta_{i,k}}
\bigr)
\bigr]\dd t
&\le
2\sum_{k=0}^{N_i-1}
D_{i,k}(s_{i,k})^2
\int_{s_{i,k}}^{s_{i,k+1}}
\ee^{-\alpha(t-s_{i,k})/\eps}\dd t
\\
&\le
\frac{2\eps}{\alpha}
\sum_{k=0}^{N_i-1}
D_{i,k}(s_{i,k})^2
\le
C_T\eps
\bigl(
1+\E\abs{X_{t_i}^\eps}^2
+\E\abs{Y_{t_i}^\eps}^2
\bigr).
\end{aligned}
\end{equation}
\end{lemma}

\begin{proof}
We first derive a bound for the cumulative squared increments of the slow environment at the microscopic grid points. The SDE and BSDE on \(I_i\), the Cauchy--Schwarz and It\^o isometries,
the linear-growth estimates for the averaged coefficients, and
\eqref{eq:local-averaged-moment} give
\begin{align}
 \sum_{k=0}^{N_i-1}
 \E\bigl[
 \abs{\widetilde X_{s_{i,k+1}}^i-\widetilde X_{s_{i,k}}^i}^2
 +\abs{\widetilde U_{s_{i,k+1}}^i-\widetilde U_{s_{i,k}}^i}^2
 \bigr]
 &\le
 C\delta_\eps\int_{I_i}
 \E\bigl(
 1+\abs{\widetilde X_t^i}^2
 +\abs{\widetilde U_t^i}^2
 +\abs{\widetilde V_t^i}^2
 \bigr)\dd t
 \notag\\
 &\quad+
 C\int_{I_i}
 \E\bigl(
 1+\abs{\widetilde X_t^i}^2
 +\abs{\widetilde U_t^i}^2
 +\abs{\widetilde V_t^i}^2
 \bigr)\dd t
 \notag\\
 &\le C_T\bigl(1+\E\abs{X_{t_i}^\eps}^2+\E\abs{Y_{t_i}^\eps}^2\bigr).
 \label{eq:environment-endpoint-energy}
\end{align}

By \cite[Corollary~5.22]{Villani2009}, we may choose, for
\(\Lambda_i\)-a.e.\ \(e\in\mathsf K_i\), an optimal coupling
\[
 \Gamma_{i,k}^e
 \in
 \mathcal C\bigl(
 \lambda_{i,s_{i,k}}^e,
 \nu^{P_{i,k}(e);\eta_{i,k}}
 \bigr)
\]
such that
\[
 \int_{\R^m\times\R^m}|y-z|^2\,
 \Gamma_{i,k}^e(\dd y,\dd z)
 =
 W_2^2\bigl(
 \lambda_{i,s_{i,k}}^e,
 \nu^{P_{i,k}(e);\eta_{i,k}}
 \bigr).
\]
By measurable disintegration, there exists a probability kernel
\(K_{i,k}(e,y,\dd z)\), jointly measurable in \((e,y)\), such that
\[
\Gamma_{i,k}^e(\dd y,\dd z)
=
\lambda_{i,s_{i,k}}^e(\dd y)K_{i,k}(e,y,\dd z).
\]
By the randomization lemma \cite[Lemma 4.22]{Kallenberg2021}, after adjoining an independent
\(U_{i,k}\sim{\rm Unif}[0,1]\), there is a measurable map
\(\Phi_{i,k}\) such that
\[
Z_{s_{i,k}}
=
\Phi_{i,k}
\bigl(\mathbf E_i,\widehat Y_{s_{i,k}}^{\eps,i},U_{i,k}\bigr)
\]
satisfies
\[
\Law\bigl(
(\widehat Y_{s_{i,k}}^{\eps,i},Z_{s_{i,k}})
\mid \mathbf E_i=e
\bigr)
=
\Gamma_{i,k}^e
\]
for \(\Lambda_i\)-a.e. \(e\).
Equivalently,
\begin{equation}\label{weq0}
     \E\!\bigl[
 |\widehat Y_{s_{i,k}}^{\eps,i}-Z_{s_{i,k}}|^2
 \,\bigm|\,\mathbf E_i=e
 \bigr]
 =
 W_2^2\!\bigl(
 \lambda_{i,s_{i,k}}^e,
 \nu^{P_{i,k}(e);\eta_{i,k}}
 \bigr).
\end{equation}
Starting from the above coupled initial pair
\(
(\widehat Y_{s_{i,k}}^{\eps,i},Z_{s_{i,k}})
\),
let \(Z\) evolve synchronously with \(\widehat Y^{\eps,i}\) on
\(J_{i,k}\):
\[
\left\{
\begin{aligned}
d\widehat Y_t^{\eps,i}
&=
\frac1\eps
h\bigl(
\widetilde X_{s_{i,k}}^i,\widehat Y_t^{\eps,i},\widetilde U_{s_{i,k}}^i,
\widehat\rho_t^{i,k}
\bigr)\dd t
+
\frac1{\sqrt\eps}
g\bigl(
\widetilde X_{s_{i,k}}^i,\widehat Y_t^{\eps,i},\widetilde U_{s_{i,k}}^i,
\widehat\rho_t^{i,k}
\bigr)\dd W_t^2,
\\
dZ_t
&=
\frac1\eps
h\bigl(
\widetilde X_{s_{i,k}}^i,Z_t,\widetilde U_{s_{i,k}}^i,
\Pi^{\eta_{i,k}}
\bigr)\dd t
+
\frac1{\sqrt\eps}
g\bigl(
\widetilde X_{s_{i,k}}^i,Z_t,\widetilde U_{s_{i,k}}^i,
\Pi^{\eta_{i,k}}
\bigr)\dd W_t^2,
\qquad t\in J_{i,k}.
\end{aligned}
\right.
\]
Since
\[
 (P_{i,k})_\#\Lambda_i=\eta_{i,k},
 \qquad
 \mathcal L\bigl(
 Z_{s_{i,k}}\mid\mathbf E_i=e
 \bigr)
 =
 \nu^{P_{i,k}(e);\eta_{i,k}},
\]
we have
\[
\begin{aligned}
\mathcal L\bigl(
(\widetilde X_{s_{i,k}}^i,\widetilde U_{s_{i,k}}^i),Z_{s_{i,k}}
\bigr)(\dd p,\dd z)
&=
\int_{\mathsf K_i}
\delta_{P_{i,k}(e)}(\dd p)\,
\nu^{P_{i,k}(e);\eta_{i,k}}(\dd z)\,
\Lambda_i(\dd e)
\\
&=
\eta_{i,k}(\dd p)\,
\nu^{p;\eta_{i,k}}(\dd z)
=
\Pi^{\eta_{i,k}}(\dd p,\dd z).
\end{aligned}
\]
Since
$
 \sigma\bigl(
 W_t^2-W_{s_{i,k}}^2:\ t\in J_{i,k}
 \bigr)
$ is independent of $
 \sigma\bigl(
 \mathbf E_i,\widehat Y_{s_{i,k}}^{\eps,i},U_{i,k}
 \bigr),
$
for \(\Lambda_i\)-a.e.\ \(e\), let
\(\Pp^e:=\Pp(\,\cdot\,\mid\mathbf E_i=e)\).
Since the future \(W^2\)-increments are independent of
\(\sigma(\mathbf E_i)\), under \(\Pp^e\),
\[
 B_r^{i,k}
 :=
 W_{s_{i,k}+r}^2-W_{s_{i,k}}^2,
 \qquad r\ge0,
\]
is a Brownian motion, independent of \(Z_{s_{i,k}}\).
Hence, with \(\bar Z_r:=Z_{s_{i,k}+r}\), under $\Pp^e$, $\bar Z$ solves
\[
 d\bar Z_r
 =
 \frac1\eps
 h\bigl(P_{i,k}(e),\bar Z_r,\Pi^{\eta_{i,k}}\bigr)\dd r
 +
 \frac1{\sqrt\eps}
 g\bigl(P_{i,k}(e),\bar Z_r,\Pi^{\eta_{i,k}}\bigr)\dd B_r^{i,k},
\]
with
$
 \mathcal L_{\Pp^e}(\bar Z_0)
 =
 \nu^{P_{i,k}(e);\eta_{i,k}}.
$
Hence, writing $Q_t^{p;\eta,*}$ for the frozen transition semigroup,
\[
\begin{aligned}
\mathcal L(Z_t\mid\mathbf E_i=e)
&=
Q_{(t-s_{i,k})/\eps}^{P_{i,k}(e);\eta_{i,k},*}
\nu^{P_{i,k}(e);\eta_{i,k}}
=
\nu^{P_{i,k}(e);\eta_{i,k}},
\\
\mathcal L\bigl(
(\widetilde X_{s_{i,k}}^i,\widetilde U_{s_{i,k}}^i),Z_t
\bigr)
&=
\eta_{i,k}(\dd p)\,
\nu^{p;\eta_{i,k}}(\dd z)
=
\Pi^{\eta_{i,k}}(\dd p,\dd z),
\qquad t\in J_{i,k}.
\end{aligned}
\]
Note that
\[
 \mathcal L\bigl(
 (\widetilde X_{s_{i,k}}^i,\widetilde U_{s_{i,k}}^i),
 \widehat Y_t^{\eps,i}
 \bigr)
 =
 \widehat\rho_t^{i,k}.
\]
Hence
\[
 W_2^2\bigl(
 \widehat\rho_t^{i,k},\Pi^{\eta_{i,k}}
 \bigr)
 \le
 \E\bigl|
 \widehat Y_t^{\eps,i}-Z_t
 \bigr|^2,
\]
and Assumption {\rm(H2)} gives
\[
\begin{aligned}
\frac{\dd}{\dd t}
\E\bigl|\widehat Y_t^{\eps,i}-Z_t\bigr|^2
&\le
-\frac{\kappa}{\eps}
\E\bigl|\widehat Y_t^{\eps,i}-Z_t\bigr|^2
+
\frac{K_2}{\eps}
W_2^2\bigl(
\widehat\rho_t^{i,k},\Pi^{\eta_{i,k}}
\bigr)
\\
&\le
-\frac{\alpha}{\eps}
\E\bigl|\widehat Y_t^{\eps,i}-Z_t\bigr|^2.
\end{aligned}
\]
Moreover, by \eqref{weq0},
\[
 \E\bigl|
 \widehat Y_{s_{i,k}}^{\eps,i}-Z_{s_{i,k}}
 \bigr|^2
 =
 D_{i,k}(s_{i,k})^2.
\]
Since, conditionally on \(\mathbf E_i=e\),
\((\widehat Y_t^{\eps,i},Z_t)\) is a coupling of
\(\lambda_{i,t}^e\) and
\(\nu^{P_{i,k}(e);\eta_{i,k}}\), we have
\[
 D_{i,k}(t)^2
 \le
 \E\bigl|
 \widehat Y_t^{\eps,i}-Z_t
 \bigr|^2
 \le
 \ee^{-\alpha(t-s_{i,k})/\eps}D_{i,k}(s_{i,k})^2.
\]
Thus
\begin{equation}\label{eq:block-equilibrium-tracking}
 D_{i,k}(t)
 \le
 \ee^{-\alpha(t-s_{i,k})/(2\eps)}D_{i,k}(s_{i,k}).
\end{equation}
Furthermore, integrating optimal conditional couplings at time \(t\)
over \(\Lambda_i\) gives
\begin{equation}\label{eq:block-joint-law-tracking}
 W_2^2\bigl(
 \widehat\rho_t^{i,k},\Pi^{\eta_{i,k}}
 \bigr)
 \le
 \int_{\mathsf K_i}
 W_2^2\bigl(
 \lambda_{i,t}^e,
 \nu^{P_{i,k}(e);\eta_{i,k}}
 \bigr)\Lambda_i(\dd e)
 =
 D_{i,k}(t)^2.
\end{equation}

At \(s_{i,k+1}\), the triangle inequality and
Lemma~\ref{lem:conditional-invariant-stability} yield
\[
\begin{aligned}
D_{i,k+1}(s_{i,k+1})
&\le
D_{i,k}(s_{i,k+1})
+
\bigl[
\int_{\mathsf K_i}
W_2^2\bigl(
\nu^{P_{i,k}(e);\eta_{i,k}},
\nu^{P_{i,k+1}(e);\eta_{i,k+1}}
\bigr)\Lambda_i(\dd e)
\bigr]^{1/2}
\\
&\le
\ee^{-\alpha\delta_\eps/(2\eps)}D_{i,k}(s_{i,k})
+
C\bigl[
\bigl(
\int_{\mathsf K_i}
|P_{i,k+1}(e)-P_{i,k}(e)|^2\Lambda_i(\dd e)
\bigr)^{1/2}
+
W_2(\eta_{i,k},\eta_{i,k+1})
\bigr].
\end{aligned}
\]
By the definition of \(P_{i,k}\),
\begin{align*}
 \int_{\mathsf K_i}
 |P_{i,k+1}(e)-P_{i,k}(e)|^2\Lambda_i(\dd e)
 &=
 \E\bigl[
 \abs{\widetilde X_{s_{i,k+1}}^i-\widetilde X_{s_{i,k}}^i}^2
 +\abs{\widetilde U_{s_{i,k+1}}^i-\widetilde U_{s_{i,k}}^i}^2
 \bigr],
 \\
 W_2^2(\eta_{i,k},\eta_{i,k+1})
 &\le
 \E\bigl[
 \abs{\widetilde X_{s_{i,k+1}}^i-\widetilde X_{s_{i,k}}^i}^2
 +\abs{\widetilde U_{s_{i,k+1}}^i-\widetilde U_{s_{i,k}}^i}^2
 \bigr].
\end{align*}
Therefore
\begin{align}
 D_{i,k+1}(s_{i,k+1})
 &\le
 \ee^{-\alpha\delta_\eps/(2\eps)}D_{i,k}(s_{i,k})
 +C\Bigl(\E\bigl[
 \abs{\widetilde X_{s_{i,k+1}}^i-\widetilde X_{s_{i,k}}^i}^2
 +\abs{\widetilde U_{s_{i,k+1}}^i-\widetilde U_{s_{i,k}}^i}^2
 \bigr]\Bigr)^{1/2}
 \notag\\
 &\le
 \ee^{-\alpha/4}D_{i,k}(s_{i,k})
 +C\Bigl(\E\bigl[
 \abs{\widetilde X_{s_{i,k+1}}^i-\widetilde X_{s_{i,k}}^i}^2
 +\abs{\widetilde U_{s_{i,k+1}}^i-\widetilde U_{s_{i,k}}^i}^2
 \bigr]\Bigr)^{1/2}.
 \label{eq:block-tracking-recursion}
\end{align}
Squaring \eqref{eq:block-tracking-recursion} and choosing $\theta>0$ such that
$(1+\theta)\ee^{-\alpha/2}<1$, we obtain
\[
D_{i,k+1}(s_{i,k+1})^2
\le
(1+\theta)\ee^{-\alpha/2}D_{i,k}(s_{i,k})^2
+C\E\bigl[
\abs{\widetilde X_{s_{i,k+1}}^i-\widetilde X_{s_{i,k}}^i}^2
+\abs{\widetilde U_{s_{i,k+1}}^i-\widetilde U_{s_{i,k}}^i}^2
\bigr].
\]
Summing over $k$ and absorbing the first term on the right, and then using
\eqref{eq:environment-endpoint-energy} together with the moment estimates and
the invariant-kernel moment bound \eqref{eq:conditional-invariant-moment}, we obtain
\begin{align*}
 \sum_{k=0}^{N_i-1}D_{i,k}(s_{i,k})^2
 &\le
 C D_{i,0}(s_{i,0})^2
 +C\sum_{k=0}^{N_i-2}\E\bigl[
 \abs{\widetilde X_{s_{i,k+1}}^i-\widetilde X_{s_{i,k}}^i}^2
 +\abs{\widetilde U_{s_{i,k+1}}^i-\widetilde U_{s_{i,k}}^i}^2
 \bigr]
 \\
 &\le C_T\bigl(1+\E\abs{X_{t_i}^\eps}^2+\E\abs{Y_{t_i}^\eps}^2\bigr).
\end{align*}
Finally, by \eqref{eq:block-equilibrium-tracking} and \eqref{eq:block-joint-law-tracking},
the preceding bound proves \eqref{eq:integrated-equilibrium-tracking}.
\end{proof}

\subsection{Conditional Bias and the Gordin Decomposition}

We now prove the \(O(\eps)\) maximal bounds
\eqref{eq:centered-forward-integral}--\eqref{eq:centered-backward-integral}
for the centered coefficients. To this end, we decompose each centered
coefficient into two parts: a conditional bias and a conditionally centered
fluctuation. The former arises from the fact that, under the frozen random
environment, the auxiliary fast dynamics has not yet reached the corresponding
conditional equilibrium. This term is controlled by
Lemma~\ref{lem:conditional-equilibrium-tracking}. To handle the remaining
conditionally centered fluctuation, inspired by the Gordin
martingale--coboundary decomposition
\cite{Gordin1969,GordinPeligrad2011}, we introduce a Gordin-type
decomposition that separates it into a martingale term and a boundary
corrector.

Let \(k(t)\) be the block index containing \(t\), and let
\(\mathfrak a_t^i\) denote either
\(\mathfrak b_t^{i,k(t)}\) or
\(\mathfrak F_t^{i,k(t)}\).

\begin{lemma}
\label{lem:maximal-centered-integral}
Assume {\rm(H1)--(H3)}. For either choice of \(\mathfrak a^i\),
\begin{equation}\label{eq:centered-forward-integral}
 \E\sup_{r\in I_i}
 \bigl|
 \int_{t_i}^r
 \mathfrak a_t^{i,k(t)}\dd t
 \bigr|^2
 \le C\eps \bigl(1+\E\abs{X_{t_i}^\eps}^2+\E\abs{Y_{t_i}^\eps}^2\bigr),
\end{equation}
and
\begin{equation}\label{eq:centered-backward-integral}
 \E\sup_{r\in I_i}
 \bigl|
 \int_r^{t_{i+1}}
 \mathfrak a_t^{i,k(t)}\dd t
 \bigr|^2
 \le C\eps \bigl(1+\E\abs{X_{t_i}^\eps}^2+\E\abs{Y_{t_i}^\eps}^2\bigr).
\end{equation}
\end{lemma}

\begin{proof}
Set
\begin{equation*}
 \mu_t^i:=\E[\mathfrak a_t^i\mid\mathcal K_i],
 \qquad
 \zeta_t^i:=\mathfrak a_t^i-\mu_t^i.
\end{equation*}
Fix \(t\in J_{i,k}\) and write
\(e=(y_0,(x(\cdot),u(\cdot)),v)\in\mathsf K_i\).
Let \(\Gamma_{i,t}^e\) be an optimal coupling of
\(\lambda_{i,t}^e\) and
\(\nu^{P_{i,k}(e);\eta_{i,k}}\).

Consider first the case \(\mathfrak a=\mathfrak F\). Conditioning on
\(\mathbf E_i=e\) and using the definition of \(\bar F\), we obtain
\[
\begin{aligned}
\mu_t^i(e)
={}&
\int_{\R^m}
F\bigl(
P_{i,k}(e),y,(v(t),0),\widehat\rho_t^{i,k}
\bigr)
\lambda_{i,t}^e(dy)
\\
&-
\int_{\R^m}
F\bigl(
P_{i,k}(e),z,(v(t),0),\Pi^{\eta_{i,k}}
\bigr)
\nu^{P_{i,k}(e);\eta_{i,k}}(dz)
\\
={}&
\int_{\R^m\times\R^m}
\Big[
F\bigl(
P_{i,k}(e),y,(v(t),0),\widehat\rho_t^{i,k}
\bigr)
\\
&\hspace{35mm}
-
F\bigl(
P_{i,k}(e),z,(v(t),0),\Pi^{\eta_{i,k}}
\bigr)
\Big]\Gamma_{i,t}^e(dy,dz).
\end{aligned}
\]
Hence, by the Lipschitz property of \(F\),
\[
\begin{aligned}
|\mu_t^i(e)|
&\le
L\int_{\R^m\times\R^m}|y-z|\,\Gamma_{i,t}^e(dy,dz)
+
L W_2\bigl(
\widehat\rho_t^{i,k},\Pi^{\eta_{i,k}}
\bigr)
\\
&\le
L W_2\bigl(
\lambda_{i,t}^e,
\nu^{P_{i,k}(e);\eta_{i,k}}
\bigr)
+
L W_2\bigl(
\widehat\rho_t^{i,k},\Pi^{\eta_{i,k}}
\bigr).
\end{aligned}
\]
The case \(\mathfrak a=\mathfrak b\) is identical, with the
\(v\)-argument absent.
Consequently, \eqref{eq:Dik-process-definition} and
\eqref{eq:block-joint-law-tracking} yield
\begin{equation*}
 \E\abs{\mu_t^i}^2
 \le C D_{i,k}(t)^2.
\end{equation*}
Thus, by Cauchy--Schwarz and \eqref{eq:integrated-equilibrium-tracking},
\begin{equation}\label{eq:conditional-bias-maximal}
 \E\sup_{r\in I_i}
 \bigl|
 \int_{t_i}^r\mu_t^i\dd t
 \bigr|^2
 \le
 \ell\int_{I_i}\E\abs{\mu_t^i}^2\dd t
 \le C_T\eps \bigl(1+\E\abs{X_{t_i}^\eps}^2+\E\abs{Y_{t_i}^\eps}^2\bigr).
\end{equation}

It remains to control \(\zeta^i\).  Define the completed filtration
\begin{equation*}
 \Hh_{i,r}
 :=
 \bigl(
 \mathcal K_i\vee
 \sigma(W_t^2-W_{t_i}^2:t_i\le t\le r)
 \bigr)^{\Pp},
 \qquad r\in I_i.
\end{equation*}

Fix \(e=(y_0,(x(\cdot),u(\cdot)),v(\cdot))\in\mathsf K_i\).
For \(t\in J_{i,j}\cap[r,t_{i+1}]\), let \(Y_t^{r,y,e}\) solve
\begin{equation*}
 \left\{
 \begin{aligned}
 dY_t^{r,y,e}
 &=
 \frac1\eps
 h\bigl(
 P_{i,j}(e),Y_t^{r,y,e},\widehat\rho_t^{i,j}
 \bigr)\dd t
 +
 \frac1{\sqrt\eps}
 g\bigl(
 P_{i,j}(e),Y_t^{r,y,e},\widehat\rho_t^{i,j}
 \bigr)\dd W_t^2,
 \\
 Y_r^{r,y,e}&=y .
 \end{aligned}
 \right.
\end{equation*}
Thus the coefficients may change at the microscopic endpoints, while the
law argument is the same deterministic unconditional law flow
\(\widehat\rho_t^{i,j}\).  Define the associated time-inhomogeneous
transition kernel by
\begin{equation*}
 K_{r,s}^e(y,A)
 :=
 \Pp\bigl(
 Y_s^{r,y,e}\in A\mid \mathbf E_i=e
 \bigr),
 \qquad
 K_{r,s}^e\phi(y)
 :=
 \int_{\R^m}\phi(z)K_{r,s}^e(y,\dd z).
\end{equation*}
If \(Y^{r,y,e}\) and \(Y^{r,y',e}\) are driven by the same \(W^2\),
then It\^o's formula and {\rm(H2)} yield
\[
\E\abs{Y_t^{r,y,e}-Y_t^{r,y',e}}^2
\le
\abs{y-y'}^2
-\frac{\kappa}{\eps}
\int_r^t
\E\abs{Y_s^{r,y,e}-Y_s^{r,y',e}}^2\,\dd s,
\qquad t\in[r,s].
\]
Hence,
\begin{equation}\label{eq:conditional-kernel-contraction}
 \E\abs{Y_s^{r,y,e}-Y_s^{r,y',e}}^2
 \le
 \ee^{-\kappa(s-r)/\eps}\abs{y-y'}^2,
 \qquad t_i\le r\le s\le t_{i+1}.
\end{equation}

For \(s\in J_{i,k}\), define \(f_s^e\) according to the choice of
\(\mathfrak a_s^i\) by
\begin{equation*}
 f_s^e(y)
 :=
 \begin{cases}
 b\bigl(
 x(s_{i,k}),y,u(s_{i,k}),\widehat\rho_s^{i,k}
 \bigr)
 -
 \bar b\bigl(
 x(s_{i,k}),u(s_{i,k}),\eta_{i,k}
 \bigr),
 & \mathfrak a_s^i=\mathfrak b_s^{i,k},
 \\[1mm]
 F\bigl(
 x(s_{i,k}),y,u(s_{i,k}),(v(s),0),\widehat\rho_s^{i,k}
 \bigr)
 -
 \bar F\bigl(
 x(s_{i,k}),u(s_{i,k}),v(s),\eta_{i,k}
 \bigr),
 & \mathfrak a_s^i=\mathfrak F_s^{i,k}.
 \end{cases}
\end{equation*}
Then
\[
 \mathfrak a_s^i
 =
 f_s^{\mathbf E_i}(\widehat Y_s^{\eps,i}),
 \qquad
 \abs{f_s^e(y)-f_s^e(y')}
 \le L\abs{y-y'}.
\]
Hence, by \eqref{eq:conditional-kernel-contraction},
\begin{align}
 \abs{K_{r,s}^ef_s^e(y)-K_{r,s}^ef_s^e(y')}
 &\le
 L\E\abs{Y_s^{r,y,e}-Y_s^{r,y',e}}
 \le
 L\ee^{-\kappa(s-r)/(2\eps)}\abs{y-y'}.
 \label{eq:conditional-transition-lipschitz}
\end{align}

Next, for every bounded Borel \(\phi\),
\begin{align*}
 \int_{\R^m}\phi(z)\lambda_{i,s}^e(\dd z)
 &=
 \E\bigl[
 \phi(\widehat Y_s^{\eps,i})\mid\mathbf E_i=e
 \bigr]
 \\
 &=
 \E\bigl[
 K_{r,s}^e\phi(\widehat Y_r^{\eps,i})
 \mid\mathbf E_i=e
 \bigr]
 =
 \int_{\R^m}
 K_{r,s}^e\phi(z)\lambda_{i,r}^e(\dd z),
\end{align*}
and therefore, after choosing consistent versions,
\[
 \lambda_{i,s}^e
 =
 K_{r,s}^{e,*}\lambda_{i,r}^e,
 \qquad
 t_i\le r\le s\le t_{i+1}.
\]
for \(\Lambda_i\)-a.e. \(e\).

Since
\[
 \mu_s^i(e)
 =
 \int_{\R^m} f_s^e(z)\lambda_{i,s}^e(\dd z)
 =
 \int_{\R^m}
 K_{r,s}^e f_s^e(z)\lambda_{i,r}^e(\dd z),
\]
fix \(e\) outside a \(\Lambda_i\)-null set and work under the regular
conditional law given \(\mathbf E_i=e\).  Under this law,
\(\widehat Y^{\eps,i}\) is a time-inhomogeneous Markov process with
transition kernel \(K_{r,s}^e\).  Hence
\[
 \E\!\bigl[
 f_s^e(\widehat Y_s^{\eps,i})
 \,\bigm|\,
 \sigma(W_t^2-W_{t_i}^2:t_i\le t\le r)
 \bigr]
 =
 K_{r,s}^e f_s^e(\widehat Y_r^{\eps,i}).
\]
Since \(\sigma(\mathbf E_i)\subset\Hh_{i,r}\), evaluating the preceding
identity at the realized environment \(e=\mathbf E_i\) gives
\begin{align*}
 \E[\zeta_s^i\mid\Hh_{i,r}]
 &=
 K_{r,s}^{\mathbf E_i}f_s^{\mathbf E_i}
 (\widehat Y_r^{\eps,i})
 -
 \int_{\R^m}
 K_{r,s}^{\mathbf E_i}f_s^{\mathbf E_i}(z)
 \lambda_{i,r}^{\mathbf E_i}(\dd z).
\end{align*}
Thus \eqref{eq:conditional-transition-lipschitz} implies
\begin{align*}
 \norm{\E[\zeta_s^i\mid\Hh_{i,r}]}_{L^2}^2
 &\le
 L^2\ee^{-\kappa(s-r)/\eps}
 \E\bigl[
 \int_{\R^m}
 \abs{\widehat Y_r^{\eps,i}-z}^2
 \lambda_{i,r}^{\mathbf E_i}(\dd z)
 \bigr]
 \\
 &\le
 2L^2\ee^{-\kappa(s-r)/\eps}
 \E\abs{\widehat Y_r^{\eps,i}}^2
 \le
 C\ee^{-\kappa(s-r)/\eps}\bigl(1+\E\abs{X_{t_i}^\eps}^2+\E\abs{Y_{t_i}^\eps}^2\bigr),
\end{align*}
where the last inequality follows from \eqref{eq:Yhat-moment}.  Hence
\begin{equation}\label{eq:global-conditional-mixing}
 \norm{\E[\zeta_s^i\mid\Hh_{i,r}]}_{L^2}
 \le
 C\ee^{-\kappa(s-r)/2\eps}\bigl(1+\E\abs{X_{t_i}^\eps}^2+\E\abs{Y_{t_i}^\eps}^2\bigr)^{1/2},
 \qquad t_i\le r\le s\le t_{i+1}.
\end{equation}
Finally,
\eqref{eq:centered-second-moment} and Jensen inequality yield
\begin{equation}\label{eq:zeta-second-moment}
 \sup_{t\in I_i}\E\abs{\zeta_t^i}^2
 \le C\bigl(1+\E\abs{X_{t_i}^\eps}^2+\E\abs{Y_{t_i}^\eps}^2\bigr).
\end{equation}

For \(0\le k<N_i\), put
\begin{equation*}
 Z_{i,k}:=\int_{s_{i,k}}^{s_{i,k+1}}\zeta_t^i\dd t,
 \qquad
 \mathcal C_{i,k}
 :=
 \E\bigl[
 \sum_{j=k}^{N_i-1}Z_{i,j}
 \mathrel{\big|}\Hh_{i,s_{i,k}}
 \bigr],
 \qquad \mathcal C_{i,N_i}:=0.
\end{equation*}
By Minkowski's inequality and \eqref{eq:global-conditional-mixing},
\begin{align}
 \norm{\mathcal C_{i,k}}_{L^2}
 &\le
 \sum_{j=k}^{N_i-1}
 \int_{J_{i,j}}
 \norm{\E[\zeta_t^i\mid\Hh_{i,s_{i,k}}]}_{L^2}\dd t
 \notag\\
 &\le
 C\bigl(1+\E\abs{X_{t_i}^\eps}^2+\E\abs{Y_{t_i}^\eps}^2\bigr)^{1/2}
 \sum_{j=k}^{N_i-1}
 \int_{J_{i,j}}
 \ee^{-\kappa(t-s_{i,k})/(2\eps)}\dd t
 \notag\\
 &=
 C\bigl(1+\E\abs{X_{t_i}^\eps}^2+\E\abs{Y_{t_i}^\eps}^2\bigr)^{1/2}
 \int_{s_{i,k}}^{t_{i+1}}
 \ee^{-\kappa(t-s_{i,k})/(2\eps)}\dd t
 \notag\\
 &\le
 C\eps
 \bigl(1+\E\abs{X_{t_i}^\eps}^2+\E\abs{Y_{t_i}^\eps}^2\bigr)^{1/2},
 \label{eq:gordin-corrector-bound}
\end{align}
and \eqref{eq:zeta-second-moment} gives
\begin{equation}\label{eq:gordin-block-second-moment}
 \E\abs{Z_{i,k}}^2
 \le
 \delta_\eps\int_{J_{i,k}}\E\abs{\zeta_t^i}^2\dd t
 \le C\eps^2\bigl(1+\E\abs{X_{t_i}^\eps}^2+\E\abs{Y_{t_i}^\eps}^2\bigr).
\end{equation}
Define
\begin{equation*}
 \Delta M_{i,k+1}
 :=Z_{i,k}+\mathcal C_{i,k+1}-\mathcal C_{i,k}.
\end{equation*}
The tower property and the definition of \(\mathcal C_{i,k}\) show that
\[
 \E[\Delta M_{i,k+1}\mid\Hh_{i,s_{i,k}}]=0.
\]
Thus \(M_{i,j}:=\sum_{k=0}^{j-1}\Delta M_{i,k+1}\) is a discrete
martingale.  Equations \eqref{eq:gordin-corrector-bound}--
\eqref{eq:gordin-block-second-moment} and
\eqref{eq:number-blocks} imply
\[
 \sum_{k=0}^{N_i-1}\E\abs{\Delta M_{i,k+1}}^2
 \le C_T\eps \bigl(1+\E\abs{X_{t_i}^\eps}^2+\E\abs{Y_{t_i}^\eps}^2\bigr).
\]
In particular,
\[
 M_{i,N_i}
 =\sum_{k=0}^{N_i-1}Z_{i,k}-\mathcal C_{i,0},
 \qquad
 \E\abs{M_{i,N_i}}^2
 =\sum_{k=0}^{N_i-1}\E\abs{\Delta M_{i,k+1}}^2
 \le C_T\eps \bigl(1+\E\abs{X_{t_i}^\eps}^2+\E\abs{Y_{t_i}^\eps}^2\bigr).
\]
Doob's inequality and the identity
\[
 \sum_{k=0}^{j-1}Z_{i,k}
 =M_{i,j}+\mathcal C_{i,0}-\mathcal C_{i,j}
\]
give
\begin{equation}\label{eq:gordin-grid-maximal}
 \E\max_{0\le j\le N_i}
 \bigl|
 \sum_{k=0}^{j-1}Z_{i,k}
 \bigr|^2
 \le C_T\eps \bigl(1+\E\abs{X_{t_i}^\eps}^2+\E\abs{Y_{t_i}^\eps}^2\bigr),
\end{equation}
where we also used
\[
 \E\max_{0\le k\le N_i}\abs{\mathcal C_{i,k}}^2
 \le
 \sum_{k=0}^{N_i}\E\abs{\mathcal C_{i,k}}^2
 \le C_T\eps \bigl(1+\E\abs{X_{t_i}^\eps}^2+\E\abs{Y_{t_i}^\eps}^2\bigr).
\]
For the part of an integral inside one block,
\begin{align*}
 &\E\max_{0\le k<N_i}
 \sup_{s_{i,k}\le r\le s_{i,k+1}}
 \bigl|
 \int_{s_{i,k}}^r\zeta_t^i\dd t
 \bigr|^2
 \\
 &\quad\le
 \delta_\eps\sum_{k=0}^{N_i-1}
 \int_{J_{i,k}}\E\abs{\zeta_t^i}^2\dd t
 \le C_T\eps \bigl(1+\E\abs{X_{t_i}^\eps}^2+\E\abs{Y_{t_i}^\eps}^2\bigr).
\end{align*}
Combining this estimate with \eqref{eq:gordin-grid-maximal} yields
\begin{equation}\label{eq:centered-fluctuation-maximal}
 \E\sup_{r\in I_i}
 \bigl|
 \int_{t_i}^r\zeta_t^i\dd t
 \bigr|^2
 \le C_T\eps \bigl(1+\E\abs{X_{t_i}^\eps}^2+\E\abs{Y_{t_i}^\eps}^2\bigr).
\end{equation}

Finally, \eqref{eq:conditional-bias-maximal} and
\eqref{eq:centered-fluctuation-maximal} prove
\eqref{eq:centered-forward-integral}.
For backward tails, subtract the forward integral from the integral over
the whole interval.  This gives \eqref{eq:centered-backward-integral}.
\end{proof}

\subsection{Global Strong Averaging Principle}\label{sec:local-to-global}

We now complete the proof of the strong averaging principle. Combining the
centered-integral estimates obtained above with the local stability estimates
for the forward--backward system, we derive coupled error bounds on each
macroscopic interval and propagate them along the partition. We then compare
the restarted averaged flows with the global averaged solution, which yields
the global strong error estimate in Theorem~\ref{thm:main}.

\begin{lemma}
\label{lem:local-forward-backward-comparison}
Suppose that {\rm(H1)--(H3)} hold. Then there exist
\(h_0\in(0,1]\) and constants \(C_0,C_1>0\), independent of \(i\) and
\(\eps\), such that, whenever the macroscopic partition
\eqref{eq:macro-partition} satisfies \(\ell<h_0\),
\begin{align}
 &\E\sup_{t\in I_i}\abs{X_t^\eps-\widetilde X_t^i}^2
 +\E\sup_{t\in I_i}\abs{U_t^\eps-\widetilde U_t^i}^2
 +\E\int_{I_i}\abs{V_t^\eps-(\widetilde V_t^i,0)}^2\dd t
 \notag\\
 &\quad\le
 C_0\E\abs{R_{t_{i+1}}^\eps}^2
 +C_1\eps
 \bigl(
 1+\E\abs{X_{t_i}^\eps}^2+\E\abs{Y_{t_i}^\eps}^2
 \bigr),
 \qquad 0\le i<M.
 \label{eq:local-comparison}
\end{align}
\end{lemma}

\begin{proof}
Define, for \(r\in I_i\),
\begin{align*}
 \mathcal R_{i,b}^\eps(r)
 &:=
 \int_{t_i}^r
 \bigl[
 b\bigl(
 X_t^\eps,Y_t^\eps,U_t^\eps,
 \Law((X_t^\eps,U_t^\eps),Y_t^\eps)
 \bigr)
 -\bar b\bigl(
 \widetilde X_t^i,\widetilde U_t^i,
 \Law(\widetilde X_t^i,\widetilde U_t^i)
 \bigr)
 \bigr]\dd t.
\end{align*}
For \(t\in J_{i,k}\), the integrand is the sum of the following three
terms:
\begin{align*}
 B_{i,1}(t)
 &:=
 b\bigl(
 X_t^\eps,Y_t^\eps,U_t^\eps,
 \Law((X_t^\eps,U_t^\eps),Y_t^\eps)
 \bigr)-
 b\bigl(
 \widetilde X_{s_{i,k}}^i,\widehat Y_t^{\eps,i},\widetilde U_{s_{i,k}}^i,
 \widehat\rho_t^{i,k}
 \bigr),
\\
 B_{i,2}(t)&:=\mathfrak b_t^{i,k}, \ \ 
 B_{i,3}(t)
 :=
 \bar b(\widetilde X_{s_{i,k}}^i,\widetilde U_{s_{i,k}}^i,\eta_{i,k})
-
 \bar b\bigl(
 \widetilde X_t^i,\widetilde U_t^i,
 \Law(\widetilde X_t^i,\widetilde U_t^i)
 \bigr).
\end{align*}
By \eqref{eq:H1-bhg},
\begin{align*}
 \E\abs{B_{i,1}(t)}^2
 &\le
 C\E\bigl[
 \abs{X_t^\eps-\widetilde X_{s_{i,k}}^i}^2
 +\abs{U_t^\eps-\widetilde U_{s_{i,k}}^i}^2
 +\abs{Y_t^\eps-\widehat Y_t^{\eps,i}}^2
 \bigr].
\end{align*}
Decompose the first two differences through
\((\widetilde X_t^i,\widetilde U_t^i)\), integrate over \(I_i\), and use
\eqref{eq:local-freezing-modulus} and
\eqref{eq:fast-comparison-integrated}.  This yields
\begin{equation}\label{eq:forward-B1}
 \int_{I_i}\E\abs{B_{i,1}(t)}^2\dd t
 \le
 C\ell\bigl(
 \E\sup_{t\in I_i}\abs{X_t^\eps-\widetilde X_t^i}^2
 +\E\sup_{t\in I_i}\abs{U_t^\eps-\widetilde U_t^i}^2
 \bigr)
 +C\delta_\eps \bigl(1+\E\abs{X_{t_i}^\eps}^2+\E\abs{Y_{t_i}^\eps}^2\bigr).
\end{equation}

For the third term, \eqref{eq:averaged-b-lipschitz} and
\eqref{eq:local-freezing-modulus} give
\begin{equation}\label{eq:forward-B3}
 \int_{I_i}\E\abs{B_{i,3}(t)}^2\dd t
 \le C\delta_\eps \bigl(1+\E\abs{X_{t_i}^\eps}^2+\E\abs{Y_{t_i}^\eps}^2\bigr).
\end{equation}
For \(j=1,3\), Cauchy--Schwarz implies
\[
 \E\sup_{r\in I_i}
 \bigl|\int_{t_i}^rB_{i,j}(t)\dd t\bigr|^2
 \le
 \ell\int_{I_i}\E\abs{B_{i,j}(t)}^2\dd t.
\]
Since \(\ell\le1\), equations
\eqref{eq:forward-B1}--\eqref{eq:forward-B3} and the centered estimate
\eqref{eq:centered-forward-integral} imply
\begin{equation}\label{eq:forward-remainder-estimate}
 \E\sup_{r\in I_i}\abs{\mathcal R_{i,b}^\eps(r)}^2
 \le
 C\eps \bigl(1+\E\abs{X_{t_i}^\eps}^2+\E\abs{Y_{t_i}^\eps}^2\bigr)
 +C\ell\bigl(
 \E\sup_{t\in I_i}\abs{X_t^\eps-\widetilde X_t^i}^2
 +\E\sup_{t\in I_i}\abs{U_t^\eps-\widetilde U_t^i}^2
 \bigr).
\end{equation}

Note that the initial values of \(X^\eps\) and \(\widetilde X^i\) agree at \(t_i\).
Then
\begin{align*}
 X_r^\eps-\widetilde X_r^i
 &=
 \mathcal R_{i,b}^\eps(r)
+
 \int_{t_i}^r
 \bigl[
 \sigma(X_t^\eps,U_t^\eps,\Law(X_t^\eps,U_t^\eps))
 -\sigma(\widetilde X_t^i,\widetilde U_t^i,
 \Law(\widetilde X_t^i,\widetilde U_t^i))
 \bigr]\dd W_t^1.
\end{align*}
The Burkholder--Davis--Gundy inequality,
\eqref{eq:H1-sigma}, and
\eqref{eq:forward-remainder-estimate} consequently give
\begin{equation}\label{eq:local-forward-error}
 \E\sup_{r\in I_i}
 \abs{X_r^\eps-\widetilde X_r^i}^2
 \le
 C\eps \bigl(1+\E\abs{X_{t_i}^\eps}^2+\E\abs{Y_{t_i}^\eps}^2\bigr)
 +C\ell\bigl(
 \E\sup_{t\in I_i}\abs{X_t^\eps-\widetilde X_t^i}^2
 +\E\sup_{t\in I_i}\abs{U_t^\eps-\widetilde U_t^i}^2
 \bigr).
\end{equation}

For \(t\in I_i\), define
\begin{align*}
 r_{i,F}^\eps(t)
 &:=
 F\bigl(
 X_t^\eps,Y_t^\eps,\widetilde U_t^i,
 (\widetilde V_t^i,0),
 \Law((X_t^\eps,\widetilde U_t^i),Y_t^\eps)
 \bigr)
-
 \bar F\bigl(
 \widetilde X_t^i,\widetilde U_t^i,\widetilde V_t^i,
 \Law(\widetilde X_t^i,\widetilde U_t^i)
 \bigr).
\end{align*}
On \(J_{i,k}\), decompose
\[
 r_{i,F}^\eps(t)=F_{i,1}(t)+F_{i,2}(t)+F_{i,3}(t),
\]
where
\begin{align*}
 F_{i,1}(t)
 &:=
 F\bigl(
 X_t^\eps,Y_t^\eps,\widetilde U_t^i,
 (\widetilde V_t^i,0),
 \Law((X_t^\eps,\widetilde U_t^i),Y_t^\eps)
 \bigr)
-
 F\bigl(
 \widetilde X_{s_{i,k}}^i,\widehat Y_t^{\eps,i},\widetilde U_{s_{i,k}}^i,
 (\widetilde V_t^i,0),\widehat\rho_t^{i,k}
 \bigr),
\\
 F_{i,2}(t)&:=\mathfrak F_t^{i,k}, \qquad
 F_{i,3}(t)
 :=
 \bar F(\widetilde X_{s_{i,k}}^i,\widetilde U_{s_{i,k}}^i,
 \widetilde V_t^i,\eta_{i,k})
-
 \bar F\bigl(
 \widetilde X_t^i,\widetilde U_t^i,\widetilde V_t^i,
 \Law(\widetilde X_t^i,\widetilde U_t^i)
 \bigr).
\end{align*}
The law difference in the first term satisfies the explicit coupling bound
\begin{align}
 W_2^2\bigl(
 \Law((X_t^\eps,\widetilde U_t^i),Y_t^\eps),
 \widehat\rho_t^{i,k}
 \bigr)
 &\le
 \E\bigl[
 \abs{X_t^\eps-\widetilde X_{s_{i,k}}^i}^2
 +\abs{\widetilde U_t^i-\widetilde U_{s_{i,k}}^i}^2
 +\abs{Y_t^\eps-\widehat Y_t^{\eps,i}}^2
 \bigr].
 \label{eq:backward-law-coupling}
\end{align}
The same value of \(\widetilde V_t^i\) occurs on both sides of
\(F_{i,1}\).  Thus \eqref{eq:H1-F},
\eqref{eq:backward-law-coupling},
\eqref{eq:local-freezing-modulus}, and
\eqref{eq:fast-comparison-integrated} give
\begin{equation}\label{eq:backward-F1}
 \int_{I_i}\E\abs{F_{i,1}(t)}^2\dd t
 \le
 C\ell\Bigl(
 \E\sup_{t\in I_i}\abs{X_t^\eps-\widetilde X_t^i}^2
 +\E\sup_{t\in I_i}\abs{U_t^\eps-\widetilde U_t^i}^2
 \Bigr)
 +C\delta_\eps \bigl(1+\E\abs{X_{t_i}^\eps}^2+\E\abs{Y_{t_i}^\eps}^2\bigr).
\end{equation}
Consequently, \eqref{eq:averaged-F-lipschitz} and
\eqref{eq:local-freezing-modulus} imply
\begin{equation}\label{eq:backward-F3}
 \int_{I_i}\E\abs{F_{i,3}(t)}^2\dd t
 \le C\delta_\eps \bigl(1+\E\abs{X_{t_i}^\eps}^2+\E\abs{Y_{t_i}^\eps}^2\bigr).
\end{equation}
For \(j=1,3\),
\[
 \E\sup_{r\in I_i}
 \bigl|\int_r^{t_{i+1}}F_{i,j}(t)\dd t\bigr|^2
 \le
 \ell\int_{I_i}\E\abs{F_{i,j}(t)}^2\dd t.
\]
Combining \eqref{eq:backward-F1}--\eqref{eq:backward-F3} with
\eqref{eq:centered-backward-integral} yields
\begin{equation}\label{eq:backward-remainder-estimate}
 \E\sup_{r\in I_i}
 \bigl|\int_r^{t_{i+1}}r_{i,F}^\eps(t)\dd t\bigr|^2
 \le
 C\eps \bigl(1+\E\abs{X_{t_i}^\eps}^2+\E\abs{Y_{t_i}^\eps}^2\bigr)
 +C\ell\Bigl(
 \E\sup_{t\in I_i}\abs{X_t^\eps-\widetilde X_t^i}^2
 +\E\sup_{t\in I_i}\abs{U_t^\eps-\widetilde U_t^i}^2
 \Bigr).
\end{equation}

Put
\[
 \Delta U_t:=U_t^\eps-\widetilde U_t^i,\qquad
 \Delta V_t:=V_t^\eps-(\widetilde V_t^i,0).
\]
For square-integrable
\(p\in L^2(\Ff_t;\R^p)\) and
\(q\in L^2(\Ff_t;\R^{p\times(d_1+d_2)})\), define
\begin{align*}
 \Psi_t^i(p,q)
 &:=
 F\bigl(
 X_t^\eps,Y_t^\eps,\widetilde U_t^i+p,
 (\widetilde V_t^i,0)+q,
 \Law((X_t^\eps,\widetilde U_t^i+p),Y_t^\eps)
 \bigr)
 \\
 &\qquad\qquad\qquad-
 F\bigl(
 X_t^\eps,Y_t^\eps,\widetilde U_t^i,
 (\widetilde V_t^i,0),
 \Law((X_t^\eps,\widetilde U_t^i),Y_t^\eps)
 \bigr).
\end{align*}
Then
\(\Psi_t^i(0,0)=0\).  For \(p,p'\in L^2(\mathcal F_t;\R^p)\),
the synchronous coupling that keeps \(X_t^\eps\),
\(\widetilde U_t^i\), and \(Y_t^\eps\) unchanged gives
\[
 W_2^2\bigl(
 \Law((X_t^\eps,\widetilde U_t^i+p),Y_t^\eps),
 \Law((X_t^\eps,\widetilde U_t^i+p'),Y_t^\eps)
 \bigr)
 \le
 \E\abs{p-p'}^2.
\]
Hence, by \eqref{eq:H1-F},
\[
 \norm{\Psi_t^i(p,q)-\Psi_t^i(p',q')}_{L^2}
 \le
 C\bigl(
 \norm{p-p'}_{L^2}+\norm{q-q'}_{L^2}
 \bigr).
\]
This is the precise \(L^2\)-operator property required by
Lemma~\ref{lem:operator-cancellation}.

The difference of the backward equations satisfies
\begin{align}
 \Delta U_t
 &=
 U_{t_{i+1}}^\eps-\widetilde U_{t_{i+1}}^i
 +\int_t^{t_{i+1}}
 \bigl[
 \Psi_s^i(\Delta U_s,\Delta V_s)
 +r_{i,F}^\eps(s)
 \bigr]\dd s
-
 \int_t^{t_{i+1}}\Delta V_s\dd(W_s^1,W_s^2),
 \quad t\in I_i.
 \label{eq:difference-bsde}
\end{align}
By \eqref{eq:grid-defect} and \eqref{eq:local-restart-flow},
\begin{align*}
 U_{t_{i+1}}^\eps-\widetilde U_{t_{i+1}}^i
 &=
 R_{t_{i+1}}^\eps
 +\bar U_{t_{i+1}}^{\,t_{i+1},X_{t_{i+1}}^\eps}
 -\bar U_{t_{i+1}}^{\,t_{i+1},
 \widetilde X_{t_{i+1}}^i}.
\end{align*}
The restart stability assumption yields
\begin{align}
 \E\abs{U_{t_{i+1}}^\eps-\widetilde U_{t_{i+1}}^i}^2
 &\le
 2\E\abs{R_{t_{i+1}}^\eps}^2
 +2L_{\rm rst}^2
 \E\abs{X_{t_{i+1}}^\eps
 -\widetilde X_{t_{i+1}}^i}^2.
 \label{eq:local-terminal-bound}
\end{align}

Apply Lemma~\ref{lem:operator-cancellation} to
\eqref{eq:difference-bsde}.  Using
\eqref{eq:backward-remainder-estimate} and
\eqref{eq:local-terminal-bound}, we obtain
\begin{align*}
 &\E\sup_{t\in I_i}\abs{U_t^\eps-\widetilde U_t^i}^2
 +\E\int_{I_i}
 \abs{V_t^\eps-(\widetilde V_t^i,0)}^2\dd t
 \\
 &\quad\le
 C\E\abs{R_{t_{i+1}}^\eps}^2
 +C\E\sup_{t\in I_i}
 \abs{X_t^\eps-\widetilde X_t^i}^2+C\eps \bigl(1+\E\abs{X_{t_i}^\eps}^2+\E\abs{Y_{t_i}^\eps}^2\bigr)
 \\
 &\qquad+
 C\ell\Bigl(
 \E\sup_{t\in I_i}\abs{X_t^\eps-\widetilde X_t^i}^2
 +\E\sup_{t\in I_i}\abs{U_t^\eps-\widetilde U_t^i}^2
 +\E\int_{I_i}\abs{V_t^\eps-(\widetilde V_t^i,0)}^2\dd t
 \Bigr).
\end{align*}
Substituting the forward estimate \eqref{eq:local-forward-error}
into the preceding backward estimate and choosing
\(h_0\) in \eqref{eq:macro-partition} sufficiently small,
we obtain \eqref{eq:local-comparison}.
\end{proof}

\paragraph{Global recursions and conclusion.}

At the left endpoint of \(I_i\), \eqref{eq:local-initial-match} gives
\[
 R_{t_i}^\eps
 =
 U_{t_i}^\eps-\widetilde U_{t_i}^i.
\]
Hence
\(\E\abs{R_{t_i}^\eps}^2\le
\E\sup_{t\in I_i}\abs{U_t^\eps-\widetilde U_t^i}^2\).
Therefore \eqref{eq:local-comparison} yields directly
\[
 \E\abs{R_{t_i}^\eps}^2
 \le
 C_0\E\abs{R_{t_{i+1}}^\eps}^2
 +C_1\eps \bigl(1+\E\abs{X_{t_i}^\eps}^2+\E\abs{Y_{t_i}^\eps}^2\bigr),
 \qquad \E\abs{R_T^\eps}^2=0.
\]
Iterating backward over the fixed number \(M\) of intervals gives
\begin{equation}\label{eq:defect-max}
 \max_{0\le i\le M}\E\abs{R_{t_i}^\eps}^2
 \le
 C_T\eps
 \max_{0\le i\le M}
 \bigl(1+\E\abs{X_{t_i}^\eps}^2
 +\E\abs{Y_{t_i}^\eps}^2\bigr).
\end{equation}

We next derive the complementary forward recursion.  By
\eqref{eq:local-averaged-moment},
\eqref{eq:local-comparison}, for sufficiently small
\(\eps\), we have
\begin{align}
 &\E\sup_{t\in I_i}
 \bigl(\abs{X_t^\eps}^2+\abs{U_t^\eps}^2\bigr)
 +\E\int_{I_i}\abs{V_t^\eps}^2\dd t
\le
 C\bigl(1+\E\abs{X_{t_i}^\eps}^2
 +\E\abs{Y_{t_i}^\eps}^2+\E\abs{R_{t_{i+1}}^\eps}^2\bigr).
 \label{eq:original-local-moment}
\end{align}

The similar calculation in Lemma~\ref{lem:frozen-lyapunov}, applied to the
original fast equation and rescaled by \(\eps\), gives
\begin{align*}
 \frac{\dd}{\dd t}\E\abs{Y_t^\eps}^2
 \le
 -\frac{c_0}{\eps}\E\abs{Y_t^\eps}^2
 +\frac C\eps
 \bigl(
 1+\E\abs{X_t^\eps}^2+\E\abs{U_t^\eps}^2
 \bigr).
\end{align*}
Variation of constants on \(I_i\) yields
\begin{align*}
 \E\abs{Y_{t_{i+1}}^\eps}^2
 &\le
 \ee^{-c_0\ell/\eps}\E\abs{Y_{t_i}^\eps}^2
+
 \frac C\eps\int_{t_i}^{t_{i+1}}
 \ee^{-c_0(t_{i+1}-t)/\eps}
 \bigl(
 1+\E\abs{X_t^\eps}^2+\E\abs{U_t^\eps}^2
 \bigr)\dd t.
\end{align*}
Since
\[
 \frac1\eps\int_{t_i}^{t_{i+1}}
 \ee^{-c_0(t_{i+1}-t)/\eps}\dd t\le\frac1{c_0},
\]
\eqref{eq:original-local-moment} gives
\begin{equation*}
 \E\abs{Y_{t_{i+1}}^\eps}^2
 \le C\bigl(1+\E\abs{X_{t_i}^\eps}^2
 +\E\abs{Y_{t_i}^\eps}^2+\E\abs{R_{t_{i+1}}^\eps}^2\bigr).
\end{equation*}
The same bound for \(X_{t_{i+1}}^\eps\) follows directly from
\eqref{eq:original-local-moment}.  Consequently, for constants \(C_2,C_3\),
\begin{equation*}
 1+\E\abs{X_{t_{i+1}}^\eps}^2
 +\E\abs{Y_{t_{i+1}}^\eps}^2
 \le
 C_2\bigl(1+\E\abs{X_{t_i}^\eps}^2
 +\E\abs{Y_{t_i}^\eps}^2\bigr)
 +C_3\E\abs{R_{t_{i+1}}^\eps}^2.
\end{equation*}
Iterating forward over the fixed number of macroscopic intervals gives
\begin{equation}\label{eq:moment-max-before-absorption}
 \max_{0\le i\le M}
 \bigl(1+\E\abs{X_{t_i}^\eps}^2
 +\E\abs{Y_{t_i}^\eps}^2\bigr)
 \le
 C_T\Bigl(
 1+\abs x^2+\abs y^2
 +\max_{0\le i\le M}\E\abs{R_{t_i}^\eps}^2
 \Bigr).
\end{equation}
Substituting \eqref{eq:defect-max} into
\eqref{eq:moment-max-before-absorption},
\begin{align*}
 \max_{0\le i\le M}
 \bigl(1+\E\abs{X_{t_i}^\eps}^2
 +\E\abs{Y_{t_i}^\eps}^2\bigr)
 &\le
 C_T(1+\abs x^2+\abs y^2)
 +C_T\eps\max_{0\le i\le M}
 \bigl(1+\E\abs{X_{t_i}^\eps}^2
 +\E\abs{Y_{t_i}^\eps}^2\bigr).
\end{align*}
Since \(\eps\to0\), there exists \(\eps_T>0\) such that for any $0<\eps\leq\eps_{T}$,
$
 C_T\eps\le\frac12.
$
Then we have
\begin{equation*}
 \max_{0\le i\le M}
 \bigl(1+\E\abs{X_{t_i}^\eps}^2
 +\E\abs{Y_{t_i}^\eps}^2\bigr)
 \le C_T(1+\abs x^2+\abs y^2).
\end{equation*}
Equations \eqref{eq:defect-max} and
\eqref{eq:local-comparison} then imply
\begin{equation}\label{eq:defect-local-final}
\begin{aligned}
 &\max_{0\le i\le M}\E\abs{R_{t_i}^\eps}^2
 +\max_{0\le i<M}\biggl[
 \E\sup_{t\in I_i}\abs{X_t^\eps-\widetilde X_t^i}^2
 +\E\sup_{t\in I_i}\abs{U_t^\eps-\widetilde U_t^i}^2
 +\E\int_{I_i}\abs{V_t^\eps-(\widetilde V_t^i,0)}^2\dd t
 \biggr]
 \\
 &\quad
 \le
 C_T\eps(1+\abs x^2+\abs y^2).
\end{aligned}
\end{equation}

The restriction of the global averaged solution
\((\bar X,\bar U,\bar V)\) to \([t_i,T]\) is the averaged restart flow from
\((t_i,\bar X_{t_i})\).  Lemma~\ref{lem:averaged-flow-stability}, applied
to this flow and the local flow started from \(X_{t_i}^\eps\), gives
\begin{align}
 &\E\sup_{t\in I_i}\abs{\widetilde X_t^i-\bar X_t}^2
 +\E\sup_{t\in I_i}\abs{\widetilde U_t^i-\bar U_t}^2
+
 \E\int_{I_i}\abs{\widetilde V_t^i-\bar V_t}^2\dd t
 \le
 C_T\E\abs{X_{t_i}^\eps-\bar X_{t_i}}^2.
 \label{eq:local-to-global-flow}
\end{align}

At the right endpoint of \(I_i\),
\[
 X_{t_{i+1}}^\eps-\bar X_{t_{i+1}}
 =
 (X_{t_{i+1}}^\eps-\widetilde X_{t_{i+1}}^i)
 +(\widetilde X_{t_{i+1}}^i-\bar X_{t_{i+1}}).
\]
Using \eqref{eq:defect-local-final} and
\eqref{eq:local-to-global-flow}, we obtain
\[
 \E\abs{X_{t_{i+1}}^\eps-\bar X_{t_{i+1}}}^2
 \le C_T\eps(1+\abs x^2+\abs y^2)
 +C_T\E\abs{X_{t_i}^\eps-\bar X_{t_i}}^2.
\]
Since \(X_0^\eps=\bar X_0=x\) and \(M\) is fixed,
\begin{equation}\label{eq:global-grid-X}
 \max_{0\le i\le M}
 \E\abs{X_{t_i}^\eps-\bar X_{t_i}}^2
 \le
 C_T\eps(1+\abs x^2+\abs y^2).
\end{equation}

For \(t\in I_i\), decompose
\[
 X_t^\eps-\bar X_t
 =
 (X_t^\eps-\widetilde X_t^i)
 +(\widetilde X_t^i-\bar X_t).
\]
There are only \(M\) macroscopic intervals.  Summing the supremum estimates
from \eqref{eq:defect-local-final},
\eqref{eq:local-to-global-flow}, and
\eqref{eq:global-grid-X}, we obtain
\begin{equation}\label{eq:global-X-final}
 \E\sup_{0\le t\le T}\abs{X_t^\eps-\bar X_t}^2
 \le
 C_T\eps(1+\abs x^2+\abs y^2).
\end{equation}

The same decomposition gives, on \(I_i\),
\[
 U_t^\eps-\bar U_t
 =
 (U_t^\eps-\widetilde U_t^i)
 +(\widetilde U_t^i-\bar U_t)
\]
and
\[
 V_t^\eps-(\bar V_t,0)
 =
 \bigl[V_t^\eps-(\widetilde V_t^i,0)\bigr]
 +(\widetilde V_t^i-\bar V_t,0).
\]
Using again the fixed number of intervals and
\eqref{eq:defect-local-final},
\eqref{eq:local-to-global-flow}, and
\eqref{eq:global-grid-X}, we conclude that
\begin{align}
 &\E\sup_{0\le t\le T}\abs{U_t^\eps-\bar U_t}^2
 +\E\int_0^T
 \bigl[
 \abs{V_t^{1,\eps}-\bar V_t}^2
 +\abs{V_t^{2,\eps}}^2
 \bigr]\dd t
 \le
 C_T\eps(1+\abs x^2+\abs y^2).
 \label{eq:global-UV-final}
\end{align}
Combining \eqref{eq:global-X-final} and
\eqref{eq:global-UV-final} proves
\eqref{eq:main-estimate}.  This completes the proof of
Theorem~\ref{thm:main}.

\section{Special Cases and Conditions for Global Averaging}
\label{sec:consequences}

This section discusses several important special cases of
Theorem~\ref{thm:main} and clarifies the role of \({\rm(H3)}\).
We first consider the implications for forward fast--slow
McKean--Vlasov equations and for systems whose forward dynamics
is decoupled from the backward variable.
We then show that global-in-time averaging can fail in a stochastic
mean-field example for which the frozen fast motion is exponentially
stable and both systems associated with the prescribed boundary data
are uniquely solvable.
Finally, we give conditions under which the restart stability
in \({\rm(H3)}\) follows from standard well-posedness criteria
for mean-field FBSDEs.

\subsection{Forward and Backward Specializations}
\label{sec:specializations}

We first consider the purely forward case. Let
\begin{equation*}
\left\{
\begin{aligned}
 \dd X_t^\eps
 &=
 b\bigl(X_t^\eps,Y_t^\eps,\Law(X_t^\eps,Y_t^\eps)\bigr)\dd t
 +\sigma\bigl(X_t^\eps,\Law(X_t^\eps)\bigr)\dd W_t^1,
 \\
 \dd Y_t^\eps
 &=
 \frac1\eps
 h\bigl(X_t^\eps,Y_t^\eps,\Law(X_t^\eps,Y_t^\eps)\bigr)\dd t
 +\frac1{\sqrt\eps}
 g\bigl(X_t^\eps,Y_t^\eps,\Law(X_t^\eps,Y_t^\eps)\bigr)\dd W_t^2,
 \\
 X_0^\eps&=x,\qquad Y_0^\eps=y.
\end{aligned}
\right.
\end{equation*}
For \(\eta\in\Ptwo(\R^n)\), freeze a random variable
\(\Theta\sim\eta\) and consider
\[
 \dd Z_t
 =h\bigl(\Theta,Z_t,\Law(\Theta,Z_t)\bigr)\dd t
 +g\bigl(\Theta,Z_t,\Law(\Theta,Z_t)\bigr)\dd B_t.
\]
Under the forward versions of \({\rm(H1)}\)--\({\rm(H2)}\), its invariant
joint law has a disintegration
\[
 \Pi^\eta(\dd x,\dd z)
 =\eta(\dd x)\nu^{x;\eta}(\dd z).
\]
Set
\[
 \bar b(x,\eta)
 :=\int_{\R^m}b(x,z,\Pi^\eta)\nu^{x;\eta}(\dd z),
\]
and let \(\bar X\) solve
\begin{equation*}
 \dd\bar X_t
 =\bar b\bigl(\bar X_t,\Law(\bar X_t)\bigr)\dd t
 +\sigma\bigl(\bar X_t,\Law(\bar X_t)\bigr)\dd W_t^1,
 \qquad \bar X_0=x.
\end{equation*}

\begin{corollary}
\label{cor:forward-mf}
Suppose that the forward versions of \({\rm(H1)}\)--\({\rm(H2)}\) hold.
Then, for every \(T>0\), there exist \(C_T<\infty\) and \(\eps_T>0\)
such that
\begin{equation*}
 \E\sup_{0\le t\le T}\abs{X_t^\eps-\bar X_t}^2
 \le
 C_T(1+\abs x^2+\abs y^2)\eps,
 \qquad 0<\eps\le\eps_T.
\end{equation*}
Moreover, this convergence order is optimal in general.
\end{corollary}

\begin{remark}
\label{rem:forward-literature}
{\rm We now compare Corollary~\ref{cor:forward-mf} with existing averaging
principles for two-time-scale McKean--Vlasov equations.  R\"ockner, Sun, and
Xie~\cite{RocknerSunXie2021} allow the coefficients to depend on the slow
marginal law, whereas the frozen fast equation does not depend on the law of
the fast variable.  In the time-homogeneous globally Lipschitz setting,
Corollary~\ref{cor:forward-mf} extends their model to coefficients depending
on the full joint law
$
\Law(X_t^\eps,Y_t^\eps).
$
Moreover, \cite{RocknerSunXie2021} obtains a convergence rate of order
\(1/3\).  Corollary~\ref{cor:forward-mf} improves this to the optimal order
\(1/2\) and does not require the spatial and Lions differentiability
assumptions used in the Poisson-equation approach of
\cite{RocknerSunXie2021}.  Hong, Li, Liu, and Sun~\cite{HongLiLiuSun2023}
further studied the central limit theorem for this class of systems.

For infinite-dimensional models, Hong, Li, and Liu~\cite{HongLiLiu2022}
established a strong averaging principle for fast--slow McKean--Vlasov
SPDEs in a variational framework, again with the law dependence entering
through the slow component.  In all the models mentioned above, the fast
equation does not depend on its own distribution.  Consequently, the
underlying arguments remain, to some extent, close to those used in the
classical averaging principle.

When the fast dynamics depends on its own distribution, to the best of our
knowledge, the first systematic study is due to Li, Wu, and
Xie~\cite{LiWuXie2024}.  They developed a Poisson-equation theory on the
Wasserstein space and applied it to the corresponding diffusion approximation
problem.  Their model, however, does not allow the fast equation to depend on
the state of the slow component.  Hong, Hu, Liu, and
Yang~\cite{HongHuLiuYang2026} subsequently studied the large deviation
principle for a related class of models by means of a lifted semigroup
approach.  More recently, Hou, Li, and Xie~\cite{HouLiXie2024} established
an averaging principle for genuinely fully coupled two-time-scale
McKean--Vlasov systems.

Compared with \cite{HouLiXie2024}, our approach does not require the
nondegeneracy condition on the diffusion coefficient and imposes no Lions
differentiability assumptions on the coefficients with respect to the
distribution variables. In addition, our argument works directly with the nonlinear McKean--Vlasov dynamics, without introducing an outer non-autonomous approximation that replaces the mean-field interaction by a sequence of classical SDEs. In this sense, the averaging principle is derived more directly from the intrinsic mixing structure of the nonlinear fast dynamics.  In
particular, it yields the optimal convergence rate of order \(1/2\) without
using a Poisson-equation argument.  To the best of our knowledge, obtaining
the optimal rate by such a method is new even in the classical averaging
framework.  For instance, time-discretization arguments of the type used in
\cite{LiuRocknerSunXie2020} typically yield only a convergence rate of order
\(1/4\).}
\end{remark}

We next consider a specialization in which the forward fast--slow dynamics
forms a closed system and can therefore be studied independently of the
backward equation.  More precisely, suppose that
\eqref{eq:original-mf-system} takes the form
\begin{equation}\label{eq:decoupled-forward-mf-system}
\left\{
\begin{aligned}
\dd X_t^\eps
 &=
 b\bigl(
 X_t^\eps,Y_t^\eps,
 \Law(X_t^\eps,Y_t^\eps)
 \bigr)\dd t
 +
 \sigma\bigl(
 X_t^\eps,\Law(X_t^\eps)
 \bigr)\dd W_t^1,
\\
\dd Y_t^\eps
 &=
 \frac1\eps
 h\bigl(
 X_t^\eps,Y_t^\eps,
 \Law(X_t^\eps,Y_t^\eps)
 \bigr)\dd t
 +
 \frac1{\sqrt\eps}
 g\bigl(
 X_t^\eps,Y_t^\eps,
 \Law(X_t^\eps,Y_t^\eps)
 \bigr)\dd W_t^2,
\\
\dd U_t^\eps
 &=
 -F\bigl(
 X_t^\eps,Y_t^\eps,U_t^\eps,V_t^\eps,
 \Law(X_t^\eps,U_t^\eps,Y_t^\eps)
 \bigr)\dd t
 +
 V_t^{1,\eps}\dd W_t^1
 +
 V_t^{2,\eps}\dd W_t^2,
\\
X_0^\eps&=x,\qquad
Y_0^\eps=y,\qquad
U_T^\eps
=
\beta\bigl(
X_T^\eps,\Law(X_T^\eps)
\bigr).
\end{aligned}
\right.
\end{equation}

\begin{corollary}
\label{cor:decoupled-mf-bsde}
Assume \({\rm(H1)}\)--\({\rm(H2)}\), and suppose that
\eqref{eq:original-mf-system} has the form
\eqref{eq:decoupled-forward-mf-system}.  Then \({\rm(H3)}\) holds on every
finite time interval, and the estimate \eqref{eq:main-estimate} follows.
\end{corollary}

\begin{proof}
The pair \((X^\eps,Y^\eps)\) is a closed fast--slow McKean--Vlasov system
and is uniquely solvable for each fixed \(\eps>0\).  Once
\((X^\eps,Y^\eps)\) is determined, the equation for
\((U^\eps,V^\eps)\) is a Lipschitz mean-field BSDE and is uniquely solvable
on every finite time interval; see, for instance,
\cite[Theorem~A.1]{Li2018}.

The same decoupling is inherited by the averaged system.  In particular,
the averaged forward equation is independent of the backward component and
can therefore be solved first.  Given the averaged forward process, the
backward equation is again a Lipschitz mean-field BSDE.  Standard stability
estimates for McKean--Vlasov SDEs and mean-field BSDEs then yield, uniformly
in \(s\),
\[
\norm{\bar U_s^{s,\xi}-\bar U_s^{s,\xi'}}_{L^2}
\le
C_T\norm{\xi-\xi'}_{L^2}.
\]
Hence \({\rm(H3)}\) holds, and Theorem~\ref{thm:main} applies.
\end{proof}

If all measure arguments are suppressed, Theorem~\ref{thm:main} becomes a
global-in-time averaging theorem for a classical coupled fast--slow FBSDE.
In particular, after the corresponding coefficient restrictions, the systems
considered in
\cite{JiLiu2026Potential,JiLiu2026Optimal,ShengWuYinZong2026,XuLian2023}
are law-independent special cases of the present formulation.  Therefore,
our main result may be viewed as an extension of the results in the above
references to the mean-field setting.

\subsection{Failure of Global-in-Time Averaging without Restart Stability}
\label{sec:counterexample-restart}

In this subsection, we give an example showing that the averaging
principle for coupled fast--slow mean-field forward--backward
systems need not hold on arbitrary finite time intervals.
The example extends the deterministic construction in
\cite[Example~4.1]{SWY_GlobalFBSDE} to the mean-field setting.

\begin{example}
\label{ex:critical-horizon}
{\rm Fix \(0<T\leq1\). Consider the mean-field
forward-backward stochastic system
\begin{equation}
\label{eq:counter-original}
\begin{cases}
 \mathrm dX_t^\varepsilon
 =
 \displaystyle\frac12Y_t^\varepsilon\,\mathrm dt
 +\mathrm dW_t^1,
 \\[1mm]
 \mathrm dY_t^\varepsilon
 =
 \displaystyle\frac1\varepsilon
 \bigl(
 1+U_t^\varepsilon-Y_t^\varepsilon
 +\tfrac12\mathbb EY_t^\varepsilon
 \bigr)\,\mathrm dt
 +\sqrt{\frac2\varepsilon}\,\mathrm dW_t^2,
 \\[1mm]
 \mathrm dU_t^\varepsilon
 =
 V_t^{1,\varepsilon}\,\mathrm dW_t^1
 +V_t^{2,\varepsilon}\,\mathrm dW_t^2,
 \\[1mm]
 X_0^\varepsilon=-1,\qquad Y_0^\varepsilon=0,
 \\[1mm]
 U_T^\varepsilon
 =
 -X_T^\varepsilon
 +2\mathbb EX_T^\varepsilon
 -\rho\bigl(\mathbb EX_T^\varepsilon\bigr),
\end{cases}
\end{equation}
in which \(\rho(x)=\lvert x\rvert/(1+\lvert x\rvert)\), and
\(W^1,W^2\) are independent one-dimensional Brownian motions.
All coefficients are globally Lipschitz, and (H2) holds with
\(\kappa=1\), \(K_1=2\), and \(K_2=1/2\).

Define
\(\Delta_{\varepsilon,t}
=2\varepsilon\bigl(1-e^{-t/(2\varepsilon)}\bigr)\).
Since \(1-e^{-x}<x\) for \(x>0\),
\(0<\Delta_{\varepsilon,T}<T\).
For any square-integrable adapted solution,
\(u^\varepsilon:=\mathbb EU_t^\varepsilon\) is independent of \(t\).
Writing \(m_t^\varepsilon:=\mathbb EX_t^\varepsilon\), we have
\[
 \mathbb EY_t^\varepsilon
 =
 2(1+u^\varepsilon)
 \bigl(1-e^{-t/(2\varepsilon)}\bigr),
 \qquad
 m_t^\varepsilon
 =
 -1+\bigl(t-\Delta_{\varepsilon,t}\bigr)
 (1+u^\varepsilon).
\]
The terminal condition therefore gives
\begin{equation}
\label{eq:counter-original-terminal}
 \bigl(1-T+\Delta_{\varepsilon,T}\bigr)
 (1+u^\varepsilon)
 +\rho(m_T^\varepsilon)
 =0.
\end{equation}
Since \(1-T+\Delta_{\varepsilon,T}>0\), necessarily
\(u^\varepsilon<-1\). Set
\(q^\varepsilon:=-(1+u^\varepsilon)>0\). Then
\[
 m_T^\varepsilon
 =
 -1-\bigl(T-\Delta_{\varepsilon,T}\bigr)q^\varepsilon<0.
\]
Substituting this expression into
\eqref{eq:counter-original-terminal} yields
\begin{align}
 &\bigl(1-T+\Delta_{\varepsilon,T}\bigr)
 \bigl(T-\Delta_{\varepsilon,T}\bigr)
 (q^\varepsilon)^2
 +\bigl(2-3T+3\Delta_{\varepsilon,T}\bigr)q^\varepsilon
 -1
 =0.
\label{eq:counter-original-polynomial}
\end{align}
Clearly, \eqref{eq:counter-original-polynomial} has exactly one
positive root. Thus the means are uniquely determined.

For the centered components, write
\[
 \widehat X_t^\varepsilon
 :=X_t^\varepsilon-m_t^\varepsilon,
 \qquad
 \widehat Y_t^\varepsilon
 :=Y_t^\varepsilon-\mathbb EY_t^\varepsilon,
 \qquad
 \widehat U_t^\varepsilon
 :=U_t^\varepsilon-u^\varepsilon,
\]
and define
\[
 K_\varepsilon(r):=1-e^{-r/\varepsilon},
 \qquad
 D_\varepsilon(t)
 :=
 1+\frac12
 \bigl(T-t-\varepsilon K_\varepsilon(T-t)\bigr)
 \geq1.
\]
The centered terminal condition is
\(\widehat U_T^\varepsilon=-\widehat X_T^\varepsilon\).
Since
\(\widehat U_t^\varepsilon
=-\mathbb E[\widehat X_T^\varepsilon\mid\mathcal F_t]\),
a direct computation gives
\begin{equation}
\label{eq:counter-centered-relation}
 D_\varepsilon(t)\widehat U_t^\varepsilon
 =
 -\widehat X_t^\varepsilon
 -\frac{\varepsilon}{2}
 K_\varepsilon(T-t)\widehat Y_t^\varepsilon.
\end{equation}
Substitution into the centered forward equations gives a linear
SDE with a unique square-integrable solution.
It\^o's formula verifies the backward equation, with
\begin{equation}
\label{eq:counter-martingale-integrands}
 \widehat U_0^\varepsilon=0,
 \qquad
 V_t^{1,\varepsilon}
 =
 -\frac1{D_\varepsilon(t)},
 \qquad
 V_t^{2,\varepsilon}
 =
 -\sqrt{\frac{\varepsilon}{2}}\,
 \frac{K_\varepsilon(T-t)}{D_\varepsilon(t)}.
\end{equation}
It follows that \eqref{eq:counter-original} has a unique
square-integrable adapted solution for every \(0<T\leq1\) and
every \(\varepsilon>0\).

\medskip
\noindent
{\rm\textbf{The averaged system.}}
For a frozen slow pair
\(\Theta=(\Theta^x,\Theta^u)\) with law \(\eta\), the frozen fast
equation is
\[
 \mathrm dZ_t
 =
 \bigl(
 1+\Theta^u-Z_t+\tfrac12\mathbb EZ_t
 \bigr)\,\mathrm dt
 +\sqrt2\,\mathrm dB_t,
\]
where \(B\) is a Brownian motion independent of \((\Theta,Z_0)\).
Writing
\(m_u(\eta):=\int_{\mathbb R^2}u\,\eta(\mathrm d(x,u))\),
its unique invariant joint law is
\[
 \Pi^\eta(\mathrm d(x,u),\mathrm dz)
 =
 \eta(\mathrm d(x,u))\nu^{x,u;\eta}(\mathrm dz),
 \qquad
 \nu^{x,u;\eta}
 =
 \mathcal N\bigl(2+u+m_u(\eta),1\bigr).
\]
Moreover, two solutions with the same frozen environment and
Brownian motion satisfy
\[
 \mathbb E\lvert Z_t-Z_t'\rvert^2
 \leq
 e^{-t}\mathbb E\lvert Z_0-Z_0'\rvert^2.
\]
Consequently,
\[
 \bar b(x,u,\eta)
 =
 \frac12\int_{\mathbb R}z\,\nu^{x,u;\eta}(\mathrm dz)
 =
 1+\frac12u+\frac12m_u(\eta).
\]
It follows that the limiting forward-backward system can be
written as
\begin{equation}
\label{eq:counter-averaged}
\begin{cases}
 \mathrm d\bar X_t
 =
 \bigl(
 1+\tfrac12\bar U_t+\tfrac12\mathbb E\bar U_t
 \bigr)\,\mathrm dt
 +\mathrm dW_t^1,
 \\[1mm]
 \mathrm d\bar U_t
 =
 \bar V_t\,\mathrm dW_t^1,
 \\[1mm]
 \bar X_0=-1,
 \\[1mm]
 \bar U_T
 =
 -\bar X_T
 +2\mathbb E\bar X_T
 -\rho\bigl(\mathbb E\bar X_T\bigr).
\end{cases}
\end{equation}
Let us write
\(\bar u:=\mathbb E\bar U_t\) and
\(\bar m_t:=\mathbb E\bar X_t=-1+t(1+\bar u)\).
By the terminal condition,
\[
 (1-T)(1+\bar u)+\rho(\bar m_T)=0.
\]
If \(T<1\), then necessarily \(\bar u<-1\). Setting
\(\bar q:=-(1+\bar u)>0\), we obtain
\(\bar m_T=-1-T\bar q<0\) and hence
\begin{equation}
\label{eq:counter-averaged-polynomial}
 T(1-T)\bar q^2+(2-3T)\bar q-1=0.
\end{equation}
Then it is easy to verify that
\eqref{eq:counter-averaged-polynomial} has exactly one positive
root. Thus the means are uniquely determined for every \(T<1\).

If \(T=1\), then \(\bar m_1=\bar u\), and the terminal condition
gives \(\rho(\bar u)=0\). Consequently,
\begin{equation}
\label{eq:counter-averaged-critical}
 \bar u=0,
 \qquad
 \bar m_t=t-1.
\end{equation}
For every \(0<T\leq1\), define
\(D_0(t):=1+\frac12(T-t)\).
The centered components are uniquely determined by
\[
 \bar U_t-\bar u
 =
 -\int_0^t\frac1{D_0(s)}\,\mathrm dW_s^1,
 \qquad
 \bar X_t-\bar m_t
 =
 -D_0(t)(\bar U_t-\bar u),
 \qquad
 \bar V_t=-\frac1{D_0(t)}.
\]
Therefore, the averaged system is uniquely solvable for every
\(0<T\leq1\).

\medskip
\noindent
\rm{\textbf{Averaging on $[0,1)$.}}
Fix \(T<1\). Since
\(\Delta_{\varepsilon,T}\to0\), the coefficients in
\eqref{eq:counter-original-polynomial} converge to those in
\eqref{eq:counter-averaged-polynomial}.
Therefore, a direct computation yields
\[
 \lvert q^\varepsilon-\bar q\rvert
 \leq C_T\Delta_{\varepsilon,T}
 \leq C_T\varepsilon.
\]
Using
\[
 u^\varepsilon=-1-q^\varepsilon,
 \qquad
 m_t^\varepsilon
 =
 -1-\bigl(t-\Delta_{\varepsilon,t}\bigr)q^\varepsilon,
\]
together with
\[
 \bar u=-1-\bar q,
 \qquad
 \bar m_t=-1-t\bar q,
\]
we obtain
\[
 \lvert u^\varepsilon-\bar u\rvert
 +
 \sup_{0\leq t\leq T}
 \lvert m_t^\varepsilon-\bar m_t\rvert
 \leq C_T\varepsilon.
\]
Moreover,
\[
 \sup_{0\leq t\leq T}
 \lvert D_\varepsilon(t)-D_0(t)\rvert
 \leq\frac{\varepsilon}{2},
 \qquad
 D_\varepsilon(t)\geq1,
 \qquad
 D_0(t)\geq1.
\]
By \eqref{eq:counter-martingale-integrands}, It\^o's isometry,
and Doob's inequality,
\[
 \mathbb E\sup_{0\leq t\leq T}
 \lvert U_t^\varepsilon-\bar U_t\rvert^2
 +
 \mathbb E\int_0^T
 \bigl(
 \lvert V_t^{1,\varepsilon}-\bar V_t\rvert^2
 +\lvert V_t^{2,\varepsilon}\rvert^2
 \bigr)\,\mathrm dt
 \leq C_T\varepsilon.
\]
The centered fast equation gives
\[
 \widehat Y_t^\varepsilon
 =
 \frac1\varepsilon\int_0^t
 e^{-(t-s)/\varepsilon}
 \widehat U_s^\varepsilon\,\mathrm ds
 +
 \sqrt{\frac2\varepsilon}
 \int_0^t e^{-(t-s)/\varepsilon}\,\mathrm dW_s^2.
\]
Integration by parts and Doob's inequality therefore yield
\[
 \mathbb E\sup_{0\leq t\leq T}
 \lvert\widehat U_t^\varepsilon\rvert^2
 +
 \varepsilon\mathbb E\sup_{0\leq t\leq T}
 \lvert\widehat Y_t^\varepsilon\rvert^2
 \leq C_T,
 \qquad 0<\varepsilon\leq1.
\]
Using these estimates together with
\eqref{eq:counter-centered-relation} and the representation of
\(\bar X\), we obtain
\begin{align*}
 &\mathbb E\sup_{0\leq t\leq T}
 \lvert X_t^\varepsilon-\bar X_t\rvert^2
 +
 \mathbb E\sup_{0\leq t\leq T}
 \lvert U_t^\varepsilon-\bar U_t\rvert^2
 +
 \mathbb E\int_0^T
 \bigl(
 \lvert V_t^{1,\varepsilon}-\bar V_t\rvert^2
 +\lvert V_t^{2,\varepsilon}\rvert^2
 \bigr)\,\mathrm dt
 \leq C_T\varepsilon.
\end{align*}
Thus averaging holds on every fixed interval with \(T<1\).

\medskip
\noindent
\rm{\textbf{Failure at $T=1$.}}
Let \(T=1\) and write
\(\Delta_\varepsilon=\Delta_{\varepsilon,1}\).
Equation \eqref{eq:counter-original-polynomial} becomes
\[
 \Delta_\varepsilon(1-\Delta_\varepsilon)
 (q^\varepsilon)^2
 +(3\Delta_\varepsilon-1)q^\varepsilon-1=0.
\]
Its positive root satisfies
\[
 \Delta_\varepsilon q^\varepsilon
 =
 \frac{
  1-3\Delta_\varepsilon
  +\sqrt{1-2\Delta_\varepsilon+5\Delta_\varepsilon^2}
 }{
  2(1-\Delta_\varepsilon)
 }
 \longrightarrow1
\]
as \(\varepsilon\to0\).
Since \(\Delta_\varepsilon\sim2\varepsilon\),
\[
 q^\varepsilon\sim\frac1{2\varepsilon},
 \qquad
 \mathbb EU_t^\varepsilon
 =
 -1-q^\varepsilon
 \longrightarrow-\infty.
\]
In particular, \(U_0^\varepsilon=-1-q^\varepsilon\), whereas
\(\bar U_0=0\) by \eqref{eq:counter-averaged-critical}.
Hence
\[
 \mathbb E\lvert U_0^\varepsilon-\bar U_0\rvert^2
 =
 (1+q^\varepsilon)^2
 \sim\frac1{4\varepsilon^2}
 \longrightarrow\infty.
\]
Thus averaging fails at \(T=1\), although both the original and
averaged systems remain uniquely solvable.}
\end{example}

\begin{remark}\label{rem:failure of averaging}
{\rm The essential mechanism behind this counterexample is the failure of
Assumption~{\rm (H3)}. Indeed, restart the averaged system on $[s,1]$
from $\bar X_s=\eta\in\mathbb R$:
\[
\begin{cases}
 \dd\bar X_t
 =
 \bigl(1+\tfrac12\bar U_t+\tfrac12\mathbb E\bar U_t\bigr)\dd t
 +\dd W_t^1,
 \\[1mm]
 \dd\bar U_t=\bar V_t\dd W_t^1,
 \\[1mm]
 \bar U_1
 =
 -\bar X_1+2\mathbb E\bar X_1
 -\rho\bigl(\mathbb E\bar X_1\bigr).
\end{cases}
\]
Writing $y:=\mathbb E\bar U_t$ and $z:=\mathbb E\bar X_1$ gives
\[
 z=\eta+(1-s)(1+y),
 \qquad
 y=z-\rho(z),
\]
and hence
\[
 sz+(1-s)\rho(z)=\eta+1-s.
\]
For $s=0$ and $\eta=-\frac12$, this reduces to
$\rho(z)=\frac12$, which has the two solutions $z=1$ and $z=-1$.
Define
\[
 D_0(t):=1+\frac12(1-t),
 \qquad
 M_t:=\int_0^t\frac1{D_0(r)}\dd W_r^1.
\]
Accordingly, the averaged system admits the two distinct solutions
\[
\begin{aligned}
 \bigl(\bar X_t^1,\bar U_t^1\bigr)
 =
 \big(-\frac12+\frac32t+D_0(t)M_t,\frac12-M_t\big),
 \bigl(\bar X_t^2,\bar U_t^2\bigr)
 =
 \big(-\frac12-\frac12t+D_0(t)M_t,-\frac32-M_t\big),
\end{aligned}
\]
with $\bar V_t^1=\bar V_t^2=-D_0(t)^{-1}$.
The restarted averaged system is therefore not globally well posed.
In particular, the same restarted initial condition gives rise to
the two distinct backward values $\bar U_0^1=\frac12$ and
$\bar U_0^2=-\frac32$, so the restart map
\(\xi\mapsto\bar U_s^{s,\xi}\) is not even single valued.
Thus, the uniform restart stability required in \({\rm(H3)}\)
cannot hold.}
\end{remark}

\subsection{Sufficient Conditions}
\label{sec:H3-sufficient}

Assumption \({\rm(H3)}\) has two logically distinct parts: well-posedness of
the original \(\eps\)-system and uniform stability of all restarts of the
averaged system.  The counterexample above shows that the second part cannot
be replaced by well-posedness for only the single initial condition appearing
in the main theorem.  We next give several ways to verify the restart part.
Throughout this subsection, well-posedness of the original system for every
sufficiently small fixed \(\eps>0\) is assumed unless it is obtained from the
stated structure, as in Corollary~\ref{cor:decoupled-mf-bsde}.

The first criterion is convenient when a regular mean-field decoupling field
is available.

\begin{proposition}
\label{prop:master-field-H3}
Assume that the original system is uniquely solvable for every sufficiently
small fixed \(\eps>0\).  Suppose that the averaged equation is uniquely solvable for every
\(s\in[0,T]\) and \(\xi\in L^2(\Ff_s;\R^n)\).  Assume that there exist a
measurable function
\[
 \mathcal U:[0,T]\times\R^n\times\Ptwo(\R^n)\longrightarrow\R^p
\]
and \(L_{\mathcal U}>0\) such that
\begin{equation}\label{decouplingfield0}
     \bar U_s^{s,\xi}
 =\mathcal U\bigl(s,\xi,\Law(\xi)\bigr)
\end{equation}
and
\begin{equation}\label{eq:master-field-lipschitz}
 \abs{\mathcal U(t,x,\mu)-\mathcal U(t,x',\mu')}
 \le
 L_{\mathcal U}\bigl(\abs{x-x'}+W_2(\mu,\mu')\bigr).
\end{equation}
Then \({\rm(H3)}\) holds with
\(L_{\rm rst}=2L_{\mathcal U}\).
\end{proposition}

\begin{proof}
For \(\mu=\Law(\xi)\) and \(\mu'=\Law(\xi')\),
\eqref{eq:master-field-lipschitz} and
\(W_2(\mu,\mu')\le\norm{\xi-\xi'}_{L^2}\) give
\begin{align*}
 \norm{\bar U_s^{s,\xi}-\bar U_s^{s,\xi'}}_{L^2}
 &\le
 L_{\mathcal U}
 \bigl(\norm{\xi-\xi'}_{L^2}+W_2(\mu,\mu')\bigr)\le2L_{\mathcal U}\norm{\xi-\xi'}_{L^2}.
\end{align*}
This is \eqref{eq:restart-lipschitz}.
\end{proof}

\begin{remark}\label{rem:decoupling}
    {\rm This condition is satisfied, for instance, on sufficiently short time
intervals under the regularity assumptions of
\cite{ChassagneuxCrisanDelarue2022}.  More precisely, if the coefficients
of the averaged mean-field FBSDE are globally Lipschitz in the Euclidean
and measure variables, then, on a sufficiently short time interval, the
associated decoupling field exists and is uniformly Lipschitz in both the
state and measure variables.  We emphasize, however, that extending such a regular decoupling
field to an arbitrary time horizon generally requires additional structural
conditions; smoothness of the coefficients alone is not sufficient for the
available continuation arguments.}
\end{remark}

Even when a decoupling field is not constructed, stochastic stability of the
restarted equations can be checked directly.  This is the stability criterion for the solution map
\(\xi\mapsto \bar U_s^{s,\xi}\) underlying
\cite[Proposition~2.2]{ReisingerStockingerZhang2020}, and it is exactly the
formulation adopted in \({\rm(H3)}\). The following
two consequences of that theory are particularly useful here.

\begin{proposition}
\label{prop:short-time-H3}
Assume that the original system is uniquely solvable for every sufficiently
small fixed \(\eps>0\).  There exists \(T_0>0\), depending only on the
Lipschitz constants of the averaged coefficients and the terminal condition,
such that \({\rm(H3)}\) holds whenever \(T\le T_0\).
\end{proposition}

\begin{proof}
The averaged system is a coupled mean-field FBSDE with globally
Lipschitz coefficients.  Moreover, its forward diffusion coefficient does
not depend on the backward martingale integrand, so that the compatibility
condition \(L_z^\sigma L^g<1\) in
\cite[Corollary~2.3]{ReisingerStockingerZhang2020} is automatically
satisfied with \(L_z^\sigma=0\).  Applying that result on each interval
\([s,T]\), with \(T-s\le T_0\), yields a constant \(C\), independent of
\(s\), such that
$
\norm{\bar U_s^{s,\xi}-\bar U_s^{s,\xi'}}_{L^2}
\le
C\norm{\xi-\xi'}_{L^2}.
$
Hence \({\rm(H3)}\) holds.
\end{proof}

For arbitrary finite time horizons, we use the generalized monotonicity
criterion of
\cite[Corollary~2.4]{ReisingerStockingerZhang2020}.  We state their
conditions below for the present averaged coefficients.  For \(i=1,2\), let
\[
 \Theta_i=(\xi_i,\upsilon_i,\zeta_i)
 \in L^2\bigl(\Omega;
 \R^n\times\R^p\times\R^{p\times d_1}\bigr),
\]
and define
\[
\begin{aligned}
 B_i&:=\bar b\bigl(\xi_i,\upsilon_i,\Law(\xi_i,\upsilon_i)\bigr),
 &\qquad
 \Sigma_i&:=\sigma\bigl(\xi_i,\upsilon_i,\Law(\xi_i,\upsilon_i)\bigr),\\
 F_i&:=\bar F\bigl(\xi_i,\upsilon_i,\zeta_i,
                   \Law(\xi_i,\upsilon_i)\bigr),
 &
 \Gamma_i&:=\beta\bigl(\xi_i,\Law(\xi_i)\bigr).
\end{aligned}
\]

\begin{proposition}[Generalized monotonicity]
\label{prop:monotonicity-H3}
Assume that the original system is uniquely solvable for every sufficiently
small fixed \(\eps>0\).  Suppose that there exist a matrix
\(G\in\R^{p\times n}\), constants
\(\alpha_1,\beta_1,\beta_2,L_\phi\ge0\), and nonnegative functionals
\(\phi_1(\xi_1,\xi_2)\) and \(\phi_2(\Theta_1,\Theta_2)\) such that
\begin{align*}
 &\E\ip{B_1-B_2}{G^\top(\upsilon_1-\upsilon_2)}
 +\E\ip{\Sigma_1-\Sigma_2}{G^\top(\zeta_1-\zeta_2)}
 \\[-1mm]
 &\qquad
 -\E\ip{F_1-F_2}{G(\xi_1-\xi_2)}
 \le
 -\beta_1\phi_1(\xi_1,\xi_2)
 -\beta_2\phi_2(\Theta_1,\Theta_2),
 \\
 &\E\ip{\Gamma_1-\Gamma_2}{G(\xi_1-\xi_2)}
 \ge
 \alpha_1\phi_1(\xi_1,\xi_2).
\end{align*}
Assume, in addition, that one of the following alternatives holds:
\begin{enumerate}
 \item[\rm(i)] \(\beta_2>0\) and
 \begin{equation*}
  \norm{B_1-B_2}_{L^2}^2
  +\norm{\Sigma_1-\Sigma_2}_{L^2}^2
  \le
  L_\phi\Bigl(
   \norm{\xi_1-\xi_2}_{L^2}^2
   +\phi_2(\Theta_1,\Theta_2)
  \Bigr).
 \end{equation*}

 \item[\rm(ii)] \(\alpha_1,\beta_1>0\).  Set
 \[
  \widehat F_i
  :=
  \bar F\bigl(
   \xi_i,\upsilon_2,\zeta_2,
   \Law(\xi_i,\upsilon_2)
  \bigr).
 \]
 Then
 \begin{equation*}
  \norm{\widehat F_1-\widehat F_2}_{L^2}^2
  +\norm{\Gamma_1-\Gamma_2}_{L^2}^2
  \le
  L_\phi\phi_1(\xi_1,\xi_2).
 \end{equation*}
\end{enumerate}
Then \({\rm(H3)}\) holds on every finite time interval.
\end{proposition}

\begin{proof}
The assumptions above are precisely the generalized monotonicity conditions
of \cite[Corollary~2.4]{ReisingerStockingerZhang2020}, specialized to the
present averaged coefficients.  Their continuation argument and stability
estimate yield unique solvability for every restarted averaged system and,
uniformly in \(s\),
\[
 \norm{\bar U_s^{s,\xi}-\bar U_s^{s,\xi'}}_{L^2}
 \le
 C_T\norm{\xi-\xi'}_{L^2}.
\]
Hence \({\rm(H3)}\) follows.
\end{proof}

The classical \(G\)-monotonicity condition for coupled
mean-field FBSDEs, as in
\cite[Assumption~(A.1)]{BensoussanYamZhang2015}, is recovered by taking
\(G\) of full rank and
\[
 \phi_1(\xi_1,\xi_2)
 =
 \norm{G(\xi_1-\xi_2)}_{L^2}^2,
 \qquad
 \phi_2(\Theta_1,\Theta_2)
 =
 \norm{G^\top(\upsilon_1-\upsilon_2)}_{L^2}^2
 +
 \norm{G^\top(\zeta_1-\zeta_2)}_{L^2}^2;
\]
see also \cite[Remark~2.4]{ReisingerStockingerZhang2020}.

\begin{remark}
{\rm
At the end of this section, we further explain why the present approach is
particularly useful for the study of two-scale mean-field FBSDEs. One might
naturally attempt to establish the averaging principle by means of the
Poisson equation method. For instance, in \cite{BriandHu1999}, an averaging
principle was obtained for coupled forward-backward stochastic differential
equations, with a coupling structure similar to the one considered here but
without distribution dependence, by exploiting a Poisson equation argument.
Their analysis, however, relies essentially on regularity properties of the
quasilinear PDE associated with the coupled FBSDE; see, in particular,
Lemma 3.1 therein. In the present setting, this amounts to requiring suitable
regularity of the decoupling field in \eqref{decouplingfield0}. As pointed
out in Remark \ref{rem:decoupling}, such regularity is in general difficult
to establish and does not follow merely from smoothness of the coefficients.
Consequently, the Poisson equation approaches on Wasserstein space developed
in \cite{HouLiXie2024,LiWuXie2024} do not appear to extend directly to the
averaging problem for two-scale mean-field FBSDEs. The present method avoids
this regularity requirement. We finally emphasize that one of the main
advantages of the Poisson equation approach is its ability to yield the
optimal convergence rate. The method developed here retains this advantage
and leads to the same optimal rate.
}
\end{remark}

\section{Applications in Mean-Field Control}
\label{sec:mf-control}

We apply Theorem~\ref{thm:main} to a class of mean-field stochastic
control problems by replacing the implicit mean-field optimality
condition with a fast dissipative relaxation.
This extends the fast-relaxation construction in
\cite[Section~5]{SWY_GlobalFBSDE} to the mean-field setting.
Under the assumptions below, the resulting fast--slow system is
globally well posed, and the averaged Pontryagin system satisfies
the restart stability required in \({\rm(H3)}\).
We also present a nonquadratic numerical example to illustrate
the approximation of the optimal control by the fast relaxation.

\subsection{Mean-Field Control and Fast Pontryagin Relaxation}
\label{sec:mf-control-problem}

In this section, the backward variable plays the role of the adjoint
process associated with the state variable.  We therefore take \(p=n\).
For any admissible filtration \(\mathbb G\), define the corresponding
Hilbert space of controls by
\[
 \mathcal A_{\mathbb G}(s,T)
 :=
 \mathcal H_{\mathbb G}^2(s,T;\R^m).
\]
In particular, we write
\[
 \mathcal A:=\mathcal A_{\mathbb F}(0,T),
 \qquad
 \mathbb F^{s,1}:=(\Ff_t^{s,1})_{s\le t\le T}.
\]
The former will be used for the original fast--slow system, whereas the
latter is the natural filtration for the restarted averaged control
problem.

Let
\[
 A,\widehat A\in\R^{n\times n},
 \qquad
 B,\widehat B\in\R^{n\times m},
 \qquad
 c\in\R^n,
\]
and let \(S_j\in\R^{n\times n}\) and \(s_j\in\R^n\),
\(1\le j\le d_1\), be fixed.  For
\(\mu\in\Ptwo(\R^n\times\R^m)\), define the mean state and control by
\[
 m_x(\mu):=\int_{\R^{n+m}}x\,\mu(\dd x,\dd a),
 \qquad
 m_a(\mu):=\int_{\R^{n+m}}a\,\mu(\dd x,\dd a).
\]
We then define the affine drift and diffusion coefficients by
\begin{align}
 &\qquad \qquad \qquad b(x,a,\mu)
 :=
 Ax+Ba+\widehat A\,m_x(\mu)+\widehat B\,m_a(\mu)+c,
 \label{eq:mf-control-drift}\\
 &\sigma^j(x)
 :=
 S_jx+s_j,
 \qquad 1\le j\le d_1, \qquad  \sigma(x)
 :=
 \bigl(
 \sigma^1(x),\ldots,\sigma^{d_1}(x)
 \bigr)
 \in\R^{n\times d_1}.
 \notag
\end{align}
For \(\alpha\in\mathcal A\), consider
\begin{equation}
 \dd X_t^\alpha
 =
 b\bigl(
 X_t^\alpha,\alpha_t,\Law(X_t^\alpha,\alpha_t)
 \bigr)\dd t
 +\sigma(X_t^\alpha)\dd W_t^1,
 \qquad
 X_0^\alpha=x.
 \label{eq:mf-controlled-state}
\end{equation}
Let
\[
 \ell:\R^n\times\R^m\times\Ptwo(\R^{n+m})\longrightarrow\R,
 \qquad
 \varphi:\R^n\times\Ptwo(\R^n)\longrightarrow\R,
\]
and define
\begin{equation}
 J_x(\alpha)
 :=
 \E\bigl[
 \varphi\bigl(X_T^\alpha,\Law(X_T^\alpha)\bigr)
 +\int_0^T
 \ell\bigl(
 X_t^\alpha,\alpha_t,\Law(X_t^\alpha,\alpha_t)
 \bigr)\dd t
 \bigr],
 \qquad
 \mathcal V(x):=\inf_{\alpha\in\mathcal A}J_x(\alpha).
 \label{eq:mf-control-cost}
\end{equation}

We use the Lions derivative convention; see, for example,
\cite{CarmonaDelarue2018}. For
\(\mu\in\Ptwo(\R^n\times\R^m)\), the Lions derivative of
\(\ell(x,a,\cdot)\) is an \(\R^{n+m}\)-valued map.  We decompose it
according to the state and control coordinates as
\[
 \partial_\mu\ell
 =
 \bigl(
 \partial_\mu^x\ell,
 \partial_\mu^a\ell
 \bigr),
\]
where
\(\partial_\mu^x\ell\in\R^n\) and
\(\partial_\mu^a\ell\in\R^m\).
We then define the complete lifted derivatives by
\begin{align}
 L_x(x,a,\mu)
 &:=
 \ell_x(x,a,\mu)
 +\int_{\R^{n+m}}
 \partial_\mu^x\ell(x',a',\mu)(x,a)\,
 \mu(\dd x',\dd a'),
 \label{eq:mf-control-Lx}\\
 L_a(x,a,\mu)
 &:=
 \ell_a(x,a,\mu)
 +\int_{\R^{n+m}}
 \partial_\mu^a\ell(x',a',\mu)(x,a)\,
 \mu(\dd x',\dd a'),
 \notag\\
 \Psi(x,\nu)
 &:=
 \varphi_x(x,\nu)
 +\int_{\R^n}
 \partial_\nu\varphi(x',\nu)(x)\,\nu(\dd x').
 \notag
\end{align}
Thus, for square-integrable random variables \((X,\alpha)\),
the Fr\'echet gradient of
\[
 {\boldsymbol L}(X,\alpha)
 :=
 \E\bigl[
 \ell(X,\alpha,\Law(X,\alpha))
 \bigr]
\]
is
\[
 D{\boldsymbol L}(X,\alpha)
 =
 \bigl(
 L_x(X,\alpha,\Law(X,\alpha)),
 L_a(X,\alpha,\Law(X,\alpha))
 \bigr).
\]
Similarly, the gradient of
\(
 {\boldsymbol\Phi}(X)
 :=
 \E[\varphi(X,\Law(X))]
\)
is \(\Psi(X,\Law(X))\).

\vspace{0.3cm}

We impose the following two assumptions.

\medskip
\noindent
\textbf{(C1) Lifted regularity and convexity.}
The functions \(\ell\) and \(\varphi\) are continuously differentiable
in their Euclidean variables and Lions differentiable in their measure
variables.  There exists a constant \(C>0\) such that, for all
\(x,x'\in\R^n\), \(a,a'\in\R^m\), and
\(\mu,\mu'\in\Ptwo(\R^{n+m})\),
\begin{align*}
 &\abs{L_x(x,a,\mu)-L_x(x',a',\mu')}
 +\abs{L_a(x,a,\mu)-L_a(x',a',\mu')}
 \le
 C\bigl(
 \abs{x-x'}+\abs{a-a'}+W_2(\mu,\mu')
 \bigr),
\end{align*}
and, for all \(x,x'\in\R^n\) and
\(\nu,\nu'\in\Ptwo(\R^n)\),
\begin{equation*}
 \abs{\Psi(x,\nu)-\Psi(x',\nu')}
 \le
 C\bigl(
 \abs{x-x'}+W_2(\nu,\nu')
 \bigr).
\end{equation*}
Moreover,
\begin{align*}
 \abs{L_x(x,a,\mu)}
 +\abs{L_a(x,a,\mu)}
 &\le
 C\Bigl(
 1+\abs{x}+\abs{a}
 +
 \bigl[
 \int_{\R^{n+m}}
 \bigl(\abs{x'}^2+\abs{a'}^2\bigr)
 \mu(\dd x',\dd a')
 \bigr]^{1/2}
 \Bigr),
 \\
 \abs{\Psi(x,\nu)}
 &\le
 C\bigl(
 1+\abs{x}
 +
 \bigl[
 \int_{\R^n}
 \abs{x'}^2\,\nu(\dd x')
 \bigr]^{1/2}
 \bigr).
\end{align*}

There exists \(\lambda>0\) such that, on every rich probability space
and for all square-integrable \((X_i,\alpha_i)\), \(i=1,2\),
\begin{align}
 &\E\bigl\langle
 L_x\bigl(
 X_1,\alpha_1,\Law(X_1,\alpha_1)
 \bigr)
 -
 L_x\bigl(
 X_2,\alpha_2,\Law(X_2,\alpha_2)
 \bigr),
 X_1-X_2
 \bigr\rangle
 \notag\\
 &\quad+
 \E\bigl\langle
 L_a\bigl(
 X_1,\alpha_1,\Law(X_1,\alpha_1)
 \bigr)
 -
 L_a\bigl(
 X_2,\alpha_2,\Law(X_2,\alpha_2)
 \bigr),
 \alpha_1-\alpha_2
 \bigr\rangle
 \notag\\
 &\qquad\ge
 \lambda\E\abs{\alpha_1-\alpha_2}^2,
 \label{eq:mf-control-lifted-strong-convexity}
\end{align}
and
\begin{equation}
 \E\bigl\langle
 \Psi\bigl(X_1,\Law(X_1)\bigr)
 -
 \Psi\bigl(X_2,\Law(X_2)\bigr),
 X_1-X_2
 \bigr\rangle
 \ge0.
 \label{eq:mf-control-terminal-convexity}
\end{equation}
Thus \({\boldsymbol L}\) is convex and uniformly strongly convex in the
control direction, whereas \({\boldsymbol\Phi}\) is convex. Equivalently, these conditions correspond to \(L\)-convexity of the
underlying mean-field cost functionals, with uniform strong convexity in
the control direction for the running cost; see, e.g.,
\cite[Section~5.5]{CarmonaDelarue2018}.

For
\(\rho\in\Ptwo(\R^n\times\R^n\times\R^m)\), whose coordinates are ordered as
\((x,u,a)\), let
\(\pr_{x,a}(x,u,a):=(x,a)\) and put
\[
 \mu_\rho:=(\pr_{x,a})_\#\rho,
 \qquad
 \overline u_\rho
 :=
 \int_{\R^{2n+m}}u'\,\rho(\dd x',\dd u',\dd a').
\]
For
\[
 v=(v^1,v^2)
 \in\R^{n\times d_1}\times\R^{n\times d_2},
 \qquad
 v^1=(v^{1,1},\ldots,v^{1,d_1}),
\]
define the Pontryagin residual and adjoint generator by
\begin{align}
 \mathcal R(x,a,u,\rho)
 &:=
 L_a(x,a,\mu_\rho)
 +B^\top u+\widehat B^\top\overline u_\rho,
 \label{eq:mf-control-residual}\\
 \mathcal G(x,a,u,v,\rho)
 &:=
 L_x(x,a,\mu_\rho)
 +A^\top u+\widehat A^\top\overline u_\rho
 +\sum_{j=1}^{d_1}S_j^\top v^{1,j}.
 \label{eq:mf-control-adjoint-generator}
\end{align}
Here the integral contributions contained in \(L_x\) and \(L_a\) arise
from the dependence of the cost on the joint law. The terms involving
\(\widehat A\) and \(\widehat B\) come from the dependence of the drift on
the mean state and control.

\medskip
\noindent
\textbf{(C2) Strong monotonicity of the Pontryagin residual.}
There exist constants
\[
 \kappa>K_2\ge0,
 \qquad K_1\ge0,
\]
such that
\begin{align}
 &2\bigl\langle a-a',
 \mathcal R(x,a,u,\rho)
 -\mathcal R(x',a',u',\rho')
 \bigr\rangle
 \notag\\
 &\qquad\ge
 \kappa\abs{a-a'}^2
 -K_1\bigl(\abs{x-x'}^2+\abs{u-u'}^2\bigr)
 -K_2W_2^2(\rho,\rho').
 \label{eq:mf-control-residual-gap}
\end{align}
Condition \eqref{eq:mf-control-residual-gap} requires the strong
monotonicity of \(\mathcal R\) in the control variable to dominate the
variations caused by its dependence on \(x\), \(u\), and the law variable
\(\rho\). In particular, the strict inequality \(\kappa>K_2\) ensures
that a positive monotonicity margin remains after the contribution of the
law dependence is taken into account. For instance,
\eqref{eq:mf-control-residual-gap} follows if \(\mathcal R\) is uniformly
strongly monotone in \(a\) and Lipschitz in \((x,u,\rho)\), with the
monotonicity in \(a\) sufficiently strong relative to its dependence on
\(\rho\), after applying Young's inequality.

For a given control \(\alpha\), let \((U^\alpha,V^\alpha)\) be the adjoint
pair:
\begin{equation}
\left\{
\begin{aligned}
 \dd U_t^\alpha
 &=
 -\mathcal G\bigl(
 X_t^\alpha,\alpha_t,U_t^\alpha,V_t^\alpha,
 \Law(X_t^\alpha,U_t^\alpha,\alpha_t)
 \bigr)\dd t
 +V_t^{1,\alpha}\dd W_t^1+V_t^{2,\alpha}\dd W_t^2,
\\
 U_T^\alpha
 &=
 \Psi\bigl(X_T^\alpha,\Law(X_T^\alpha)\bigr).
\end{aligned}
\right.
\label{eq:mf-control-adjoint}
\end{equation}
The affine structure of \eqref{eq:mf-control-drift} and the definitions
\eqref{eq:mf-control-Lx}--\eqref{eq:mf-control-adjoint-generator} yield the
first variation
\begin{equation}
 DJ_x(\alpha)[\gamma]
 =
 \E\int_0^T
 \bigl\langle
 \mathcal R\bigl(
 X_t^\alpha,\alpha_t,U_t^\alpha,
 \Law(X_t^\alpha,U_t^\alpha,\alpha_t)
 \bigr),
 \gamma_t
 \bigr\rangle\dd t,
 \qquad \gamma\in\mathcal A.
 \label{eq:mf-control-first-variation}
\end{equation}
Indeed, differentiating the lifted running and terminal costs gives
\(L_x\), \(L_a\), and \(\Psi\), and the standard adjoint argument yields
\eqref{eq:mf-control-first-variation}. Hence, if \(\alpha^*\) is optimal,
the first-order condition \(DJ_x(\alpha^*)=0\) is equivalent to
\begin{equation}
 \mathcal R\bigl(
 X_t^{\alpha^*},\alpha_t^*,U_t^{\alpha^*},
 \Law(X_t^{\alpha^*},U_t^{\alpha^*},\alpha_t^*)
 \bigr)
 =0,
 \qquad
 \dd t\otimes\dd\mathbb{P}\text{-a.e.}
 \label{eq:mf-control-stationarity}
\end{equation}
This is the stationarity condition in the mean-field stochastic maximum
principle; see, for example, \cite[Section~4.1]{CarmonaDelarue2015} and
\cite[Section~6.3.1]{CarmonaDelarue2018}.

Instead of solving the implicit, self-consistent stationarity equation in
\eqref{eq:mf-control-stationarity} at every time, we introduce
\begin{equation}
\left\{
\begin{aligned}
 \dd X_t^\eps
 &=
 \bigl[
 A X_t^\eps
 +B Y_t^\eps
 +\widehat A\,\E X_t^\eps
 +\widehat B\,\E Y_t^\eps
 +c
 \bigr]\dd t
 +\sum_{j=1}^{d_1}
 \bigl(S_jX_t^\eps+s_j\bigr)\dd W_t^{1,j},
\\
 \dd Y_t^\eps
 &=
 -\frac1\eps
 \bigl[
 L_a\bigl(
 X_t^\eps,Y_t^\eps,\Law(X_t^\eps,Y_t^\eps)
 \bigr)
 +B^\top U_t^\eps
 +\widehat B^\top\E U_t^\eps
 \bigr]\dd t,
\\
 \dd U_t^\eps
 &=
 -\bigl[
 L_x\bigl(
 X_t^\eps,Y_t^\eps,\Law(X_t^\eps,Y_t^\eps)
 \bigr)
 +A^\top U_t^\eps
 +\widehat A^\top\E U_t^\eps
 +\sum_{j=1}^{d_1}
 S_j^\top V_t^{1,\eps,j}
 \bigr]\dd t
 \\
 &\hspace{7cm}
 +V_t^{1,\eps}\dd W_t^1
 +V_t^{2,\eps}\dd W_t^2,
\\
 X_0^\eps
 &=x,
 \qquad
 Y_0^\eps=y,
 \qquad
 U_T^\eps
 =
 \Psi\bigl(
 X_T^\eps,\Law(X_T^\eps)
 \bigr).
\end{aligned}
\right.
\label{eq:mf-control-relaxation}
\end{equation}
Every adapted solution \(Y^\eps\) belongs to \(\mathcal A\), and
\((U^\eps,V^\eps)\) is the adjoint pair associated with the control
\(Y^\eps\). By \eqref{eq:mf-control-first-variation}, the corresponding
Pontryagin residual represents the gradient \(DJ_x(Y^\eps)\) in
\(\mathcal A\). Hence the fast equation can be written as
\begin{equation}
 \eps\dot Y^\eps+DJ_x(Y^\eps)=0,
 \qquad Y_0^\eps=y.
 \label{eq:mf-control-gradient-flow}
\end{equation}
Thus the fast equation is driven by the gradient \(DJ_x\) of the
mean-field control cost.

We next identify its frozen equilibrium. For
\(\eta\in\Ptwo(\R^n\times\R^n)\), there exists a unique measurable map
\[
 \Gamma(\,\cdot\,,\,\cdot\,,\eta):
 \R^n\times\R^n\longrightarrow\R^m
\]
such that, if \((X,U)\sim\eta\), then
\begin{equation*}
 \Pi^\eta
 :=
 \Law\bigl(
 X,U,\Gamma(X,U,\eta)
 \bigr)
 =
 \bigl(
 (x,u)\mapsto
 (x,u,\Gamma(x,u,\eta))
 \bigr)_\#\eta,
\end{equation*}
and
\begin{equation}
 \mathcal R\bigl(
 x,\Gamma(x,u,\eta),u,\Pi^\eta
 \bigr)
 =0
 \quad\text{for }\eta\text{-a.e. }(x,u).
 \label{eq:mf-control-self-consistent-stationarity}
\end{equation}
To justify this construction, fix
\(\eta\in\Ptwo(\R^n\times\R^n)\). For a probability measure
\(\rho\) with \((x,u)\)-marginal \(\eta\), let
\(\gamma_\rho(x,u)\) denote the unique solution of
\[
 \mathcal R\bigl(x,\gamma_\rho(x,u),u,\rho\bigr)=0.
\]
Indeed, by \eqref{eq:mf-control-residual-gap}, the map
\(a\mapsto\mathcal R(x,a,u,\rho)\) is strongly monotone, so that this
zero is unique. Define
\[
 \mathcal T_\eta(\rho)
 :=
 \Law\bigl(
 X,U,\gamma_\rho(X,U)
 \bigr),
 \qquad (X,U)\sim\eta.
\]
If \(\rho\) and \(\rho'\) have the same \((x,u)\)-marginal \(\eta\),
then \eqref{eq:mf-control-residual-gap}, applied to
\(\gamma_\rho(x,u)\) and \(\gamma_{\rho'}(x,u)\), gives
\[
 \bigl|
 \gamma_\rho(x,u)-\gamma_{\rho'}(x,u)
 \bigr|^2
 \le
 \frac{K_2}{\kappa}W_2^2(\rho,\rho').
\]
Using the coupling induced by the same random variable
\((X,U)\sim\eta\), we obtain
\[
 W_2^2\bigl(
 \mathcal T_\eta(\rho),
 \mathcal T_\eta(\rho')
 \bigr)
 \le
 \frac{K_2}{\kappa}W_2^2(\rho,\rho').
\]
Since \(K_2/\kappa<1\), the map \(\mathcal T_\eta\) is a contraction.
Denote its unique fixed point by \(\Pi^\eta\), and define
\[
 \Gamma(x,u,\eta)
 :=
 \gamma_{\Pi^\eta}(x,u).
\]
By the fixed-point property,
\[
 \Pi^\eta
 =
 \mathcal T_\eta(\Pi^\eta)
 =
 \Law\bigl(
 X,U,\Gamma(X,U,\eta)
 \bigr),
\]
and hence
\[
 \mathcal R\bigl(
 x,\Gamma(x,u,\eta),u,\Pi^\eta
 \bigr)=0
 \quad\text{for }\eta\text{-a.e. }(x,u).
\]
Since the frozen fast flow is contractive and has the unique equilibrium
\(\Gamma(x,u,\eta)\), its unique invariant probability measure is the
Dirac mass at this equilibrium. Hence its conditional invariant kernel is
\begin{equation}
 \nu^{x,u;\eta}
 =
 \delta_{\Gamma(x,u,\eta)}.
 \label{eq:mf-control-dirac-kernel}
\end{equation}
Moreover, Lemma~\ref{lem:conditional-invariant-stability} and
\eqref{eq:mf-control-dirac-kernel} yield
\begin{equation*}
 \bigl|
 \Gamma(x,u,\eta)
 -
 \Gamma(x',u',\eta')
 \bigr|
 \le
 C\bigl(
 \abs{x-x'}
 +\abs{u-u'}
 +W_2(\eta,\eta')
 \bigr).
\end{equation*}

Set
\[
 \eta_t^*:=\Law(\bar X_t,\bar U_t),
 \qquad
 \alpha_t^*:=\Gamma(\bar X_t,\bar U_t,\eta_t^*).
\]
Then
\[
 \Pi^{\eta_t^*}
 =
 \Law(\bar X_t,\bar U_t,\alpha_t^*),
\]
and the averaged equation associated with
\eqref{eq:mf-control-relaxation} is precisely
\begin{equation}
\left\{
\begin{aligned}
 \dd\bar X_t
 &=
 \bigl[
 A\bar X_t
 +B\alpha_t^*
 +\widehat A\,\E\bar X_t
 +\widehat B\,\E\alpha_t^*
 +c
 \bigr]\dd t
 +\sum_{j=1}^{d_1}
 \bigl(S_j\bar X_t+s_j\bigr)\dd W_t^{1,j},
\\
 \dd\bar U_t
 &=
 -\bigl[
 L_x\bigl(
 \bar X_t,\alpha_t^*,\Law(\bar X_t,\alpha_t^*)
 \bigr)
 +A^\top\bar U_t
 +\widehat A^\top\E\bar U_t
 +\sum_{j=1}^{d_1}
 S_j^\top\bar V_t^j
 \bigr]\dd t
 +\bar V_t\dd W_t^1,
\\
 L_a&\bigl(
 \bar X_t,\alpha_t^*,\Law(\bar X_t,\alpha_t^*)
 \bigr)
 +B^\top\bar U_t
 +\widehat B^\top\E\bar U_t=0,
\\
 \bar X_0&=x,
 \qquad
 \bar U_T=\Psi(\bar X_T,\Law(\bar X_T)).
\end{aligned}
\right.
\label{eq:mf-control-pontryagin}
\end{equation}
Thus, in the averaging limit, the fast residual dynamics is replaced by
the mean-field Pontryagin stationarity condition
\[
 \mathcal R(
 \bar X_t,\alpha_t^*,\bar U_t,\Pi^{\eta_t^*}
 )=0.
\]

\subsection{Well-Posedness, Averaging, and Near-Optimality}
\label{sec:mf-control-H3}

We next establish the well-posedness of the mean-field control system and its averaged counterpart. The argument is based on the structural monotonicity of the control problem.

\begin{proposition}
\label{prop:mf-control-well-posedness}
Assume \({\rm(C1)}\)--\({\rm(C2)}\).  Then, for every \(T>0\) and every
\(\eps>0\), system \eqref{eq:mf-control-relaxation} has a unique
square-integrable adapted solution.  For every \(s\in[0,T]\) and
\(\xi\in L^2(\Ff_s;\R^n)\), the restarted version of
\eqref{eq:mf-control-pontryagin}, adapted to
\((\Ff_t^{s,1})_{s\le t\le T}\) and satisfying \(\bar X_s=\xi\), has a
unique solution.  Moreover, uniformly in \(s\),
\begin{equation}
 \norm{\bar U_s^{s,\xi}-\bar U_s^{s,\xi'}}_{L^2}
 \le C_T\norm{\xi-\xi'}_{L^2}.
 \label{eq:mf-control-restart-estimate}
\end{equation}
Consequently, \({\rm(H3)}\) holds for
\eqref{eq:mf-control-relaxation}.  Finally, the control
\(\alpha^*\) in \eqref{eq:mf-control-pontryagin} is the unique minimizer
of \eqref{eq:mf-control-cost}.
\end{proposition}

\begin{proof}

Fix \(s\in[0,T]\) and an initial state
\(\xi\in L^2(\Ff_s;\R^n)\).
For \(\alpha\in\mathcal A_{\mathbb F^{s,1}}(s,T)\), let
\(X^{s,\xi,\alpha}\) denote the corresponding state process on \([s,T]\)
with \(X_s^{s,\xi,\alpha}=\xi\), and define the reduced cost
\[
J_{s,\xi}(\alpha)
:=
\E\left[
\varphi\bigl(
X_T^{s,\xi,\alpha},
\Law(X_T^{s,\xi,\alpha})
\bigr)
+
\int_s^T
\ell\bigl(
X_t^{s,\xi,\alpha},
\alpha_t,
\Law(X_t^{s,\xi,\alpha},\alpha_t)
\bigr)\dd t
\right].
\]
Set
\[
 \Delta X:=X^{s,\xi,\alpha}-X^{s,\xi',\alpha'},
 \qquad
 \Delta U:=U^{s,\xi,\alpha}-U^{s,\xi',\alpha'},
 \qquad
 \Delta V:=V^{s,\xi,\alpha}-V^{s,\xi',\alpha'}.
\]
Standard stability estimates for the affine state equation give
\[
 \norm{\Delta X}_{\mathcal S^2(s,T)}
 \le
 C_T\bigl(
 \norm{\xi-\xi'}_{L^2}
 +\norm{\alpha-\alpha'}_{\mathcal A_{\mathbb F^{s,1}}(s,T)}
 \bigr).
\]
Using the Lipschitz continuity of the derivatives of the running and
terminal costs, the corresponding linear mean-field adjoint BSDE satisfies
\[
 \norm{\Delta U}_{\mathcal S^2(s,T)}
 +\norm{\Delta V}_{\mathcal H^2(s,T)}
 \le
 C_T\bigl(
 \norm{\xi-\xi'}_{L^2}
 +\norm{\alpha-\alpha'}_{\mathcal A_{\mathbb F^{s,1}}(s,T)}
 \bigr).
\]
Applying the restarted version of
\eqref{eq:mf-control-first-variation} and using the Lipschitz continuity
of the Pontryagin residual, together with the preceding state and adjoint
estimates, yields
\begin{equation*}
 \norm{DJ_{s,\xi}(\alpha)-DJ_{s,\xi'}(\alpha')}
 _{\mathcal A_{\mathbb F^{s,1}}(s,T)}
 \le
 C_T\bigl(
 \norm{\xi-\xi'}_{L^2}
 +\norm{\alpha-\alpha'}_{\mathcal A_{\mathbb F^{s,1}}(s,T)}
 \bigr).
\end{equation*}

Let \(\alpha^1,\alpha^2\in\mathcal A_{\mathbb F^{s,1}}(s,T)\), and denote
their associated state processes by \(X^1,X^2\), respectively.
Applying It\^o's formula to the corresponding state--adjoint duality
pairings and using \eqref{eq:mf-control-residual} and
\eqref{eq:mf-control-adjoint-generator}, we obtain
\begin{align*}
 &\bigl\langle
 DJ_{s,\xi}(\alpha^1)-DJ_{s,\xi}(\alpha^2),
 \alpha^1-\alpha^2
 \bigr\rangle_{\mathcal A_{\mathbb F^{s,1}}(s,T)}
 \\
 &\quad=
 \E\bigl\langle
 \Psi(X_T^1,\Law(X_T^1))-\Psi(X_T^2,\Law(X_T^2)),
 X_T^1-X_T^2
 \bigr\rangle
 \\
 &\qquad+
 \int_s^T
 \E\bigl\langle
 D{\boldsymbol L}(X_t^1,\alpha_t^1)
 -D{\boldsymbol L}(X_t^2,\alpha_t^2),
 (X_t^1-X_t^2,\alpha_t^1-\alpha_t^2)
 \bigr\rangle\dd t
 \\
 &\quad\ge
 \lambda
 \norm{\alpha^1-\alpha^2}_{\mathcal A_{\mathbb F^{s,1}}(s,T)}^2.
\end{align*}
Thus \(DJ_{s,\xi}\) is globally Lipschitz and strongly monotone on
\(\mathcal A_{\mathbb F^{s,1}}(s,T)\).

We first solve the original system. We use the
monotone-operator argument similar to
\cite[Appendix~A]{SWY_GlobalFBSDE}. Applying the same state--adjoint stability and duality arguments with
\(s=0\) and the full filtration \(\mathbb F\), we obtain, for all
\(\alpha,\alpha'\in\mathcal A\),
\begin{align}
 \norm{DJ_x(\alpha)-DJ_x(\alpha')}_{\mathcal A}
 &\le
 C_T\norm{\alpha-\alpha'}_{\mathcal A},
 \label{eq:mf-control-full-gradient-stability}\\
 \bigl\langle
 DJ_x(\alpha)-DJ_x(\alpha'),
 \alpha-&\alpha'
 \bigr\rangle_{\mathcal A}
 \ge
 \lambda\norm{\alpha-\alpha'}_{\mathcal A}^2.
 \label{eq:mf-control-full-gradient-monotonicity}
\end{align}

Define the time-derivative operator
\(\mathcal D:\operatorname{Dom}(\mathcal D)\subset\mathcal A
\to\mathcal A\) by
\[
 \operatorname{Dom}(\mathcal D)
 =
 \Bigl\{
 z\in\mathcal A:
 z_t=\int_0^t\nu_r\dd r
 \text{ for some }\nu\in\mathcal A
 \Bigr\},
 \qquad
 \mathcal Dz=\nu.
\]
For \(z_1,z_2\in\operatorname{Dom}(\mathcal D)\),
\begin{equation}
 \bigl\langle
 \mathcal Dz_1-\mathcal Dz_2,
 z_1-z_2
 \bigr\rangle_{\mathcal A}
 =
 \frac12\E\abs{z_{1,T}-z_{2,T}}^2
 \ge0.
 \label{eq:mf-control-derivative-monotonicity}
\end{equation}
Moreover, for every \(\rho>0\) and \(f\in\mathcal A\), define
\[
 \bigl(R_\rho f\bigr)_t
 :=
 \frac1\rho
 \int_0^t
 \ee^{-(t-r)/\rho}f_r\dd r.
\]
Then
\[
 R_\rho f\in\operatorname{Dom}(\mathcal D),
 \qquad
 \mathcal D R_\rho f
 =
 \frac1\rho\bigl(f-R_\rho f\bigr).
\]
Thus, it is straightforward to verify that
\[
 R_\rho
 =
 (I+\rho\mathcal D)^{-1}
 :
 \mathcal A\to\mathcal A.
\]
Since
\[
 f-g
 =
 R_\rho f-R_\rho g
 +
 \rho\bigl(
 \mathcal D R_\rho f-\mathcal D R_\rho g
 \bigr),
\]
\eqref{eq:mf-control-derivative-monotonicity} gives
\begin{align*}
 \bigl\langle
 R_\rho f-R_\rho g,
 f-g
 \bigr\rangle_{\mathcal A}
 &=
 \norm{R_\rho f-R_\rho g}_{\mathcal A}^2
 +
 \frac{\rho}{2}
 \E\abs{
 (R_\rho f)_T-(R_\rho g)_T
 }^2
 \\
 &\ge
 \norm{R_\rho f-R_\rho g}_{\mathcal A}^2.
\end{align*}
Thus \(R_\rho\) is firmly nonexpansive
\cite[Proposition~4.4(iv)]{BauschkeCombettes2017} and has full domain.
By \cite[Proposition~23.8(iii)]{BauschkeCombettes2017},
\(\rho\mathcal D\) is maximally monotone on \(\mathcal A\) for every
\(\rho>0\).

Viewing \(y\in\R^m\) as a constant element of \(\mathcal A\), define
\[
 \mathcal B_y:\mathcal A\to\mathcal A,
 \qquad
 \mathcal B_y(z):=DJ_x(y+z).
\]
By \eqref{eq:mf-control-full-gradient-stability} and
\eqref{eq:mf-control-full-gradient-monotonicity},
\(\mathcal B_y\) is continuous and monotone with full domain.  Hence
\cite[Corollary~20.28]{BauschkeCombettes2017} implies that
\(\mathcal B_y\) is maximally monotone.  The maximal-monotone sum theorem
\cite[Corollary~25.5(i)]{BauschkeCombettes2017} therefore shows that
\[
 \mathcal M_\eps
 :=
 \eps\mathcal D+\mathcal B_y,
 \qquad
 \operatorname{Dom}(\mathcal M_\eps)
 =
 \operatorname{Dom}(\mathcal D),
\]
is maximally monotone on \(\mathcal A\).  In addition,
\eqref{eq:mf-control-full-gradient-monotonicity} and
\eqref{eq:mf-control-derivative-monotonicity} yield
\begin{align*}
 &\bigl\langle
 \mathcal M_\eps z_1-\mathcal M_\eps z_2,
 z_1-z_2
 \bigr\rangle_{\mathcal A}
 \\
 &\qquad=
 \frac{\eps}{2}
 \E\abs{z_{1,T}-z_{2,T}}^2
 +
 \bigl\langle
 DJ_x(y+z_1)-DJ_x(y+z_2),
 z_1-z_2
 \bigr\rangle_{\mathcal A}
 \\
 &\qquad\ge
 \lambda\norm{z_1-z_2}_{\mathcal A}^2.
\end{align*}
Thus \(\mathcal M_\eps\) is also strongly monotone.  By
\cite[Proposition~22.11(ii)]{BauschkeCombettes2017},
\(\mathcal M_\eps\) is surjective.  Since it is maximally and strongly
monotone, \cite[Proposition~23.35]{BauschkeCombettes2017} implies that
\(\mathcal M_\eps\) has a unique zero
\(z^\eps\in\operatorname{Dom}(\mathcal D)\).

Set
\[
 Y^\eps=y+z^\eps.
\]
Then \(Y_0^\eps=y\), and
\(\mathcal M_\eps z^\eps=0\) gives
\[
 \eps\dot Y_t^\eps
 +
 \bigl(DJ_x(Y^\eps)\bigr)_t
 =
 0.
\]
Let \(X^\eps\) be the state process
\eqref{eq:mf-controlled-state} associated with the control \(Y^\eps\),
and let \((U^\eps,V^\eps)\) be the corresponding adjoint pair
\eqref{eq:mf-control-adjoint}.  By
\eqref{eq:mf-control-first-variation},
\[
 \bigl(DJ_x(Y^\eps)\bigr)_t
 =
 \mathcal R\bigl(
 X_t^\eps,Y_t^\eps,U_t^\eps,
 \Law(X_t^\eps,U_t^\eps,Y_t^\eps)
 \bigr),
 \qquad
 \dd t\otimes\dd\mathbb P\text{-a.e.}
\]
Together with the definition
\eqref{eq:mf-control-residual}, this shows that
\((X^\eps,Y^\eps,U^\eps,V^\eps)\) solves
\eqref{eq:mf-control-relaxation}.  Conversely, if
\((X,Y,U,V)\) solves \eqref{eq:mf-control-relaxation}, then
\(z=Y-y\in\operatorname{Dom}(\mathcal D)\), and
\eqref{eq:mf-control-first-variation} gives
\(\mathcal M_\eps z=0\).  The uniqueness of this zero, together with
the uniqueness of \eqref{eq:mf-controlled-state} and
\eqref{eq:mf-control-adjoint} for a given control, yields the uniqueness
of the square-integrable adapted solution.

We next consider the restarted averaged system.  Set
\[
 a_i=\Gamma(\xi_i,\upsilon_i,\eta_i),
 \qquad
 \eta_i=\Law(\xi_i,\upsilon_i),
 \qquad
 \mu_i=\Law(\xi_i,a_i),
 \qquad i=1,2,
\]
and denote the corresponding coefficients of
\eqref{eq:mf-control-pontryagin} by
\(\bar b_i,\sigma_i,\bar F_i\).  The stationarity condition
\eqref{eq:mf-control-self-consistent-stationarity} gives
\begin{align*}
 &\E\langle\bar b_1-\bar b_2,\upsilon_1-\upsilon_2\rangle
 +\E\langle\sigma_1-\sigma_2,\zeta_1-\zeta_2\rangle
 -\E\langle\bar F_1-\bar F_2,\xi_1-\xi_2\rangle
 \\
 &\quad=
 -\E\bigl\langle
 L_x(\xi_1,a_1,\mu_1)-L_x(\xi_2,a_2,\mu_2),
 \xi_1-\xi_2
 \bigr\rangle
 \\
 &\qquad
 -\E\bigl\langle
 L_a(\xi_1,a_1,\mu_1)-L_a(\xi_2,a_2,\mu_2),
 a_1-a_2
 \bigr\rangle
 \le
 -\lambda\E\abs{a_1-a_2}^2,
\end{align*}
where the last inequality follows from
\eqref{eq:mf-control-lifted-strong-convexity}.  Furthermore, the affine
form of the forward coefficients gives
\[
 \norm{\bar b_1-\bar b_2}_{L^2}^2
 +
 \norm{\sigma_1-\sigma_2}_{L^2}^2
 \le
 C\left(
 \norm{\xi_1-\xi_2}_{L^2}^2
 +
 \E\abs{a_1-a_2}^2
 \right).
\]
Together with \eqref{eq:mf-control-terminal-convexity}, the generalized
monotonicity conditions of Proposition~\ref{prop:monotonicity-H3} therefore
hold with \(G=I_n\) and
\[
 \phi_2(\Theta_1,\Theta_2)=\E\abs{a_1-a_2}^2.
\]
Since the original system has already been shown to be uniquely solvable
for every \(\eps>0\), Proposition~\ref{prop:monotonicity-H3} yields unique
solvability of every restarted averaged system and, uniformly in \(s\),
\[
 \norm{\bar U_s^{s,\xi}-\bar U_s^{s,\xi'}}_{L^2}
 \le
 C_T\norm{\xi-\xi'}_{L^2}.
\]
Thus \eqref{eq:mf-control-restart-estimate} holds, and hence \({\rm(H3)}\).

Finally, the stationarity equation in
\eqref{eq:mf-control-pontryagin}, together with
\eqref{eq:mf-control-first-variation}, yields
\[
 DJ_x(\alpha^*)=0
 \qquad\text{in }\mathcal A.
\]
Since \eqref{eq:mf-control-full-gradient-monotonicity} implies that
\(J_x\) is strongly convex on \(\mathcal A\), the control
\(\alpha^*\) is the unique global minimizer of \(J_x\).
\end{proof}

\begin{theorem}
\label{thm:mf-control-averaging}
Assume \({\rm(C1)}\)--\({\rm(C2)}\).  Let
\((X^\eps,Y^\eps,U^\eps,V^\eps)\) solve
\eqref{eq:mf-control-relaxation}, and let
\((\bar X,\alpha^*,\bar U,\bar V)\) solve
\eqref{eq:mf-control-pontryagin}.  There exist
\(C_T<\infty\) and \(\eps_T>0\) such that, for
\(0<\eps\le\eps_T\),
\begin{align}
 &\E\sup_{0\le t\le T}\abs{X_t^\eps-\bar X_t}^2
 +\E\sup_{0\le t\le T}\abs{U_t^\eps-\bar U_t}^2
 \notag\\
 &\quad+
 \E\int_0^T
 \bigl(
 \abs{V_t^{1,\eps}-\bar V_t}^2
 +\abs{V_t^{2,\eps}}^2
 \bigr)\dd t
 \le
 C_T(1+\abs x^2+\abs y^2)\eps,
 \label{eq:mf-control-slow-rate}\\
 &\E\int_0^T\abs{Y_t^\eps-\alpha_t^*}^2\dd t
 \le
 C_T(1+\abs x^2+\abs y^2)\eps.
 \label{eq:mf-control-fast-rate}
\end{align}
In particular, the admissible control
\(\alpha^\eps:=Y^\eps\) is near optimal:
\begin{equation}
 0\le
 J_x(\alpha^\eps)-\mathcal V(x)
 \le
 C_T(1+\abs x^2+\abs y^2)\eps.
 \label{eq:mf-control-value-rate}
\end{equation}
\end{theorem}

\begin{proof}
To apply Theorem~\ref{thm:main}, note that the system \eqref{eq:mf-control-relaxation} is a special case of
\eqref{eq:original-mf-system}, with
\[
 h=-\mathcal R,\qquad
 g=0,\qquad
 F=\mathcal G,\qquad
 \beta=\Psi,
\]
and with the slow coefficients given by the first equation of
\eqref{eq:mf-control-relaxation}.
Assumption \({\rm(C1)}\) gives \({\rm(H1)}\), and
\eqref{eq:mf-control-residual-gap} is exactly \({\rm(H2)}\).
Proposition~\ref{prop:mf-control-well-posedness} gives \({\rm(H3)}\).
The invariant kernel is \eqref{eq:mf-control-dirac-kernel}, so the averaged
coefficients in \eqref{eq:averaged-b-definition} and
\eqref{eq:averaged-F-definition} give exactly
\eqref{eq:mf-control-pontryagin}.  Estimate
\eqref{eq:mf-control-slow-rate} therefore follows directly from
Theorem~\ref{thm:main}.

It remains to estimate the fast control.  On each microscopic block used in
Section~\ref{sec:proof}, the conditional invariant kernel is a Dirac mass.
Consequently, the conditional Wasserstein discrepancy in
\eqref{eq:integrated-equilibrium-tracking} satisfies, for
\(t\in J_{i,k}\),
\[
 D_{i,k}(t)^2
 =
 \E\bigl|
 \widehat Y_t^{\eps,i}
 -\Gamma\bigl(
 \widetilde X_{s_{i,k}}^i,
 \widetilde U_{s_{i,k}}^i,
 \eta_{i,k}
 \bigr)
 \bigr|^2.
\]
For such \(t\), the error is decomposed through
\(\widehat Y_t^{\eps,i}\) and
\(\Gamma(\widetilde X_{s_{i,k}}^i,
\widetilde U_{s_{i,k}}^i,\eta_{i,k})\).
By \eqref{eq:conditional-invariant-stability}, the difference
$
 \Gamma(\widetilde X_{s_{i,k}}^i,
 \widetilde U_{s_{i,k}}^i,\eta_{i,k})
 -
 \Gamma(X_t^\eps,U_t^\eps,\Law(X_t^\eps,U_t^\eps))
$
is controlled by the corresponding differences of the slow environment
and its law.  Hence
\eqref{eq:integrated-equilibrium-tracking},
\eqref{eq:fast-comparison-integrated},
\eqref{eq:local-freezing-modulus}, and
\eqref{eq:defect-local-final}, summed over the finitely many macroscopic
intervals, yield
\begin{equation*}
 \E\int_0^T
 \bigl|
 Y_t^\eps-
 \Gamma\bigl(
 X_t^\eps,U_t^\eps,\Law(X_t^\eps,U_t^\eps)
 \bigr)
 \bigr|^2\dd t
 \le
 C_T(1+\abs x^2+\abs y^2)\eps.
\end{equation*}
By \eqref{eq:conditional-invariant-stability} and the Dirac form of the
conditional invariant kernel, together with
\eqref{eq:mf-control-slow-rate}, we obtain
\[
 \E\int_0^T
 \bigl|
 \Gamma\bigl(
 X_t^\eps,U_t^\eps,\Law(X_t^\eps,U_t^\eps)
 \bigr)
 -\Gamma(\bar X_t,\bar U_t,\eta_t^*)
 \bigr|^2\dd t
 \le
 C_T(1+\abs x^2+\abs y^2)\eps.
\]
Since
\(\alpha_t^*=\Gamma(\bar X_t,\bar U_t,\eta_t^*)\),
the preceding two estimates together yield
\eqref{eq:mf-control-fast-rate}.

Finally, \({\rm(C1)}\), the state estimate, and the adjoint estimate imply
that \(DJ_x:\mathcal A\to\mathcal A\) is globally Lipschitz with Lipschitz
constant \(L_J\).  Since \(DJ_x(\alpha^*)=0\), by
\cite[Theorem~18.15(iii)]{BauschkeCombettes2017},
\[
 0\le J_x(Y^\eps)-J_x(\alpha^*)
 \le
 \frac{L_J}{2}
 \norm{Y^\eps-\alpha^*}_{\mathcal A}^2.
\]
Equation \eqref{eq:mf-control-fast-rate} proves
\eqref{eq:mf-control-value-rate}.
\end{proof}

\begin{remark}
\label{rem:mf-control-affine-scope}
{\rm
The affine mean-field state equation is useful here because
it makes the reduced cost strongly convex and turns the original relaxation
into the monotone evolution equation
\eqref{eq:mf-control-gradient-flow}.  Together with the convexity assumptions,
the affine structure also yields the generalized monotonicity of the averaged
Pontryagin system, which permits a direct verification of \({\rm(H3)}\).
For a general nonlinear drift \(b(x,a,\Law(X,a))\), the abstract averaging
theorem remains applicable, but the corresponding verification of
\({\rm(H3)}\) requires additional arguments.
}
\end{remark}

\begin{remark}
\label{rem:mf-control-initial-layer}
{\rm
The convergence in \eqref{eq:mf-control-fast-rate} is necessarily
formulated in an integrated sense.  Indeed, if
\(y\ne\alpha_0^*\), then
$
 \E\sup_{0\le t\le T}\abs{Y_t^\eps-\alpha_t^*}^2
 \ge
 \E\abs{y-\alpha_0^*}^2,
$
and hence uniform-in-time convergence cannot hold as \(\eps\to0\).
The mismatch at the initial time produces a fast transient, namely the
usual initial layer of the relaxation.  This behavior is also visible in
Figure~\ref{fig:numerical-mf-paths} of the numerical experiment below,
where the control error is largest near \(t=0\) and rapidly decreases
after the initial transient.
}
\end{remark}

\subsection{Numerical Experiment}
\label{sec:mf-control-numerics}

We now turn to a numerical example comparing the direct solution of the
implicit mean-field optimality condition with its two-time-scale relaxation.
The example extends the benchmark in
\cite[Section~5.2]{SWY_GlobalFBSDE} by introducing dependence on
the mean state and control in the dynamics and costs.
The control problem considered here is nonquadratic, since a quadratic
control cost would make the optimal control explicit and therefore favor
the direct method.
To obtain a genuinely implicit and nonseparable problem while retaining
global regularity, we take
\(n=m=d_1=4\), \(T=1\), and
\[
 A=-0.35I_4,
 \qquad
 \widehat A=0.08I_4,
 \qquad
 B=I_4,
 \qquad
 \widehat B=0.05I_4,
 \qquad
 c=0,
\]
with \(S_j=0\), \(s_j=0.28e_j\),
\[
 x=(1.2,-0.8,0.9,-1.1)^\top,
 \qquad y=0.
\]
For \(\mu\in\Ptwo(\R^4\times\R^4)\), write
\(m_x=m_x(\mu)\) and \(m_a=m_a(\mu)\), and define
\begin{align*}
 \ell(x,a,\mu)
 &:=\frac q2\abs{x}^2+\frac r2\abs{a}^2
 +\eta\sum_{i=1}^4\log\cosh\bigl((Ca)_i\bigr)
 +\frac{\bar q}{2}\abs{m_x}^2+\frac{\bar r}{2}\abs{m_a}^2,
 \\
 \varphi(x,\nu)
 &:=\frac{q_T}{2}\abs{x}^2
 +\frac{\bar q_T}{2}
 \abs{\int_{\R^4}x'\,\nu(\dd x')}^2,
\end{align*}
where
\[
 q=1,
 \quad \bar q=0.20,
 \quad q_T=1.5,
 \quad \bar q_T=0.30,
 \quad r=0.18,
 \quad \bar r=0.06,
 \quad \eta=0.80,
\]
and
\[
 C=
 \begin{pmatrix}
  1    &0.22&0   &0.12\\
  0.16 &1   &0.20&0   \\
  0    &0.18&1   &0.22\\
  0.20 &0   &0.14&1
 \end{pmatrix}.
\]
The off-diagonal entries of \(C\) couple the control coordinates in the
minimization.

We first verify that this example lies within the preceding theory.  The
complete lifted derivatives are
\begin{align*}
 L_x(x,a,\mu)&=qx+\bar q\,m_x(\mu),\qquad
 L_a(x,a,\mu)=ra+\eta C^\top\tanh(Ca)+\bar r\,m_a(\mu),\\
 \Psi(x,\nu)&=q_Tx+\bar q_T\int_{\R^4}x'\,\nu(\dd x').
\end{align*}
The Lipschitz and growth conditions, as well as the lifted convexity
inequalities in \({\rm(C1)}\), follow directly from these expressions and
the monotonicity of \(z\mapsto\tanh z\).  Thus \({\rm(C1)}\) holds with
\(\lambda=r\).  To check
\({\rm(C2)}\), put
\[
 L_\rho:=\sqrt{\bar r^2+0.05^2}=\sqrt{0.0061}.
\]
Using the monotonicity of \(a\mapsto C^\top\tanh(Ca)\), the bound of the
mean differences by \(W_2(\rho,\rho')\), and Young's inequality, we obtain,
for arbitrary \(\theta_1,\theta_2>0\),
\begin{align*}
 2\bigl\langle a-a',
 \mathcal R(x,a,u,\rho)-\mathcal R(x',a',u',\rho')
 \bigr\rangle
 &\ge
 (2r-\theta_1-\theta_2)\abs{a-a'}^2
 -\frac1{\theta_1}\abs{u-u'}^2
 -\frac{L_\rho^2}{\theta_2}W_2^2(\rho,\rho').
\end{align*}
Taking \(\theta_1=0.02\) and \(\theta_2=0.08\) yields
\[
 \kappa=0.26,
 \qquad K_1=50,
 \qquad K_2=0.07625<\kappa.
\]
Hence both the direct Pontryagin system and the fast relaxation are globally
well posed by Proposition~\ref{prop:mf-control-well-posedness}.

In this example, the direct feedback is the unique solution
of
\begin{equation}
 r\alpha_t^*+\eta C^\top\tanh(C\alpha_t^*)
 +\bar r\,\E\alpha_t^*
 +\bar U_t+0.05\,\E\bar U_t=0.
 \label{eq:numerical-mf-implicit-feedback}
\end{equation}
At each time step, the empirical mean of the control enters the optimality condition for every particle, so the controls of all particles must be determined simultaneously by solving a nonlinear system. The fast equation replaces this solve
by the residual evaluation
\begin{equation*}
 \dd Y_t^\eps
 =-\frac1\eps
 \bigl[
 rY_t^\eps+\eta C^\top\tanh(CY_t^\eps)
 +\bar r\,\E Y_t^\eps+U_t^\eps+0.05\,\E U_t^\eps
 \bigr]\dd t.
\end{equation*}

We use \(N=160\) uniform time steps with
$
 t_k=k\Delta t,\Delta t=T/N, k=0,\ldots,N,
$
together with \(M_{\rm train}=5000\) training paths and
\(M_{\rm test}=20000\) independent test paths.  The forward equations are
discretized by Euler--Maruyama.  We denote the resulting state approximations
by \(X_k^{\rm dir}\) for the direct system and \(X_k^\eps\) for the relaxed
system.  Conditional expectations in the backward step are estimated by
least-squares Monte Carlo, and the coupled forward-backward systems are solved
by a damped Markovian iteration as in
\cite{BenderZhang2008,GobetLemorWarin2005}.  For the least-squares regression,
we use cubic polynomial basis functions applied separately to each coordinate
of \(X_k^{\rm dir}\) in the direct system and to each coordinate of
\((X_k^\eps,Y_k^\eps)\) in the relaxed system.

For either system, write \(X_k\) and \(U_k\) for the corresponding discrete
state and adjoint variables, \(\widehat{\E}\) for the training-particle mean,
\(m_k^X:=\widehat{\E}X_k\), and \(\widehat{\mathbb E}_k\) for the least-squares approximation of the
conditional expectation at time \(t_k\).  The backward update is
\begin{align*}
 m_k^U
 &=
 \frac{
 \widehat{\mathbb E}[\,\widehat{\mathbb E}_k U_{k+1}\,]
 +(q+\bar q)\Delta t\,m_k^X
 }{
 1+(0.35-0.08)\Delta t
 },
 \qquad
 U_k
 =
 \frac{
 \widehat{\mathbb E}_k U_{k+1}
 +\Delta t(qX_k+\bar q\,m_k^X)
 +0.08\Delta t\,m_k^U
 }{
 1+0.35\Delta t
 }.
\end{align*}
The Picard coefficients are damped by \(0.35\), and the iteration is stopped
when their relative change is below \(3\times10^{-4}\).  In all six
computations reported below, the final relative change is below
\(2\times10^{-4}\).

The direct benchmark solves
\eqref{eq:numerical-mf-implicit-feedback} by iterating the empirical control
mean, with five Newton steps for each pathwise nonlinear solve.  For the
relaxed system, we use a semi-implicit discretization in which only the
linear term \(rY^\eps\) is evaluated at the new time level:
\begin{equation}
 Y_{k+1}^\eps
 =
 \frac{Y_k^\eps-\dfrac{\Delta t}{\eps}
 \bigl[
 \eta C^\top\tanh(CY_k^\eps)+U_k^\eps
 +0.05\,\widehat{\E}U_k^\eps
 +\bar r\,\widehat{\E}Y_k^\eps
 \bigr]
 }{
 1+r\Delta t/\eps}.
 \label{eq:numerical-mf-semi-implicit}
\end{equation}
The direct and relaxed systems use common Brownian increments.

For \(j=1,\ldots,M_{\rm test}\), let the superscript \(j\) denote the
\(j\)-th test path.  Define
\begin{align*}
 \mathcal E_\alpha(\eps)
 &:=\frac1{M_{\rm test}}
 \sum_{j=1}^{M_{\rm test}}
 \sum_{k=0}^{N-1}
 \abs{Y_k^{\eps,j}-\alpha_k^{{\rm dir},j}}^2\Delta t,\\
 \mathcal E_X(\eps)
 &:=\frac1{M_{\rm test}}
 \sum_{j=1}^{M_{\rm test}}
 \max_{0\le k\le N}
 \abs{X_k^{\eps,j}-X_k^{{\rm dir},j}}^2.
\end{align*}
Table~\ref{tab:numerical-mf-comparison} gives the discrete errors, the
empirical cost, and the paired cost gap.  The standard error of the direct
cost is \(1.46\times10^{-3}\), whereas the standard errors of the paired
cost gaps range from \(4.40\times10^{-5}\) to \(4.65\times10^{-4}\).
The paired cost gap measures the empirical cost of the two-time-scale
approximation relative to the direct numerical benchmark.

\begin{table}[htbp]
 \centering
 \small
 \begin{tabular}{c|c|c|c|c|c}
 Method
 & \(\eps\)
 & empirical cost
 & paired gap
 & \(\mathcal E_\alpha(\eps)\)
 & \(\mathcal E_X(\eps)\)
 \\ \hline
 Direct
 & --
 & \(1.56079\)
 & --
 & --
 & --
 \\
 Two-scale
 & \(0.2\)
 & \(1.82440\)
 & \(2.6361\times10^{-1}\)
 & \(8.8490\times10^{-1}\)
 & \(1.6057\times10^{-1}\)
 \\
 Two-scale
 & \(0.1\)
 & \(1.72225\)
 & \(1.6146\times10^{-1}\)
 & \(6.0002\times10^{-1}\)
 & \(8.2709\times10^{-2}\)
 \\
 Two-scale
 & \(0.05\)
 & \(1.65369\)
 & \(9.2901\times10^{-2}\)
 & \(3.7362\times10^{-1}\)
 & \(3.6447\times10^{-2}\)
 \\
 Two-scale
 & \(0.025\)
 & \(1.61322\)
 & \(5.2430\times10^{-2}\)
 & \(2.2223\times10^{-1}\)
 & \(1.4266\times10^{-2}\)
 \\
 Two-scale
 & \(0.0125\)
 & \(1.59100\)
 & \(3.0203\times10^{-2}\)
 & \(1.3140\times10^{-1}\)
 & \(5.1440\times10^{-3}\)
 \end{tabular}
 \caption{Comparison with the direct numerical benchmark.}
 \label{tab:numerical-mf-comparison}
\end{table}

\begin{figure}[H]
 \centering
 \includegraphics[width=0.8\textwidth]{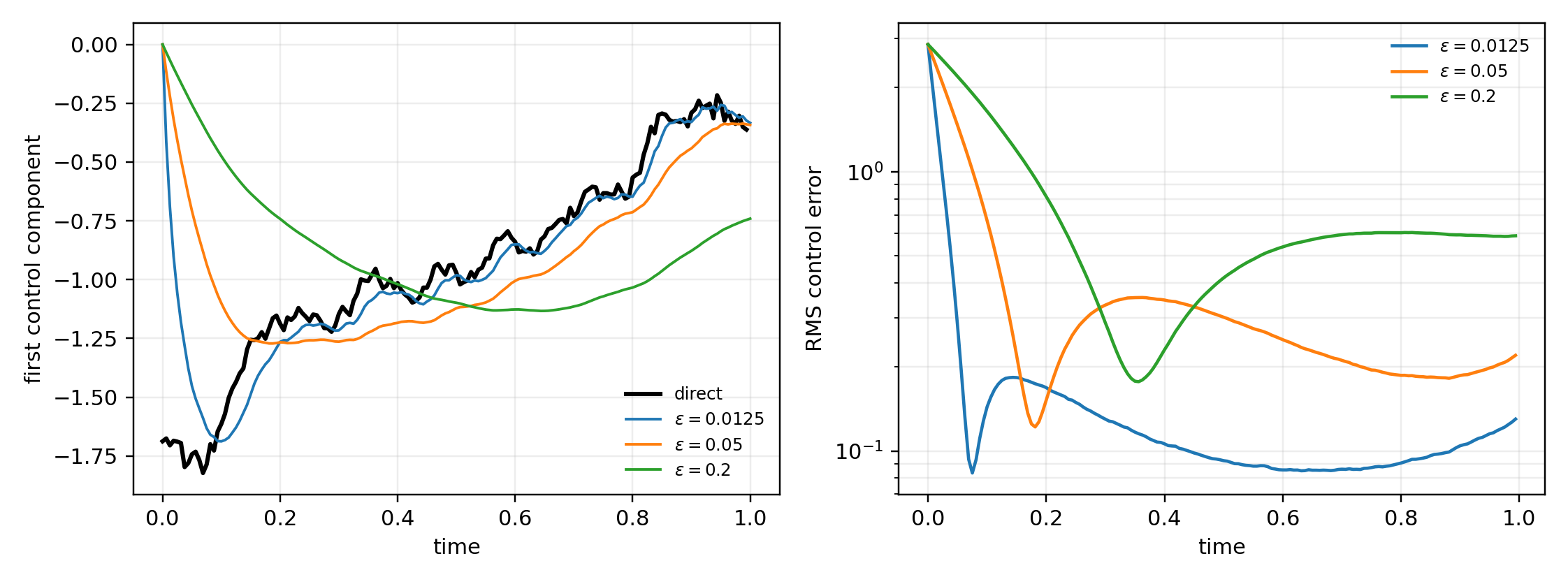}
 \caption{Left: the first components of the direct benchmark and two-time-scale
 controls along a representative test path.  Right: the root-mean-square
 control error over the test paths at each time.  Since
 \(Y_0^\eps=0\neq\alpha_0^{\rm dir}\), the tracking error exhibits an initial
 layer that narrows as \(\eps\) decreases.}
 \label{fig:numerical-mf-paths}
\end{figure}

\begin{figure}[H]
 \centering
 \includegraphics[width=0.7\textwidth]{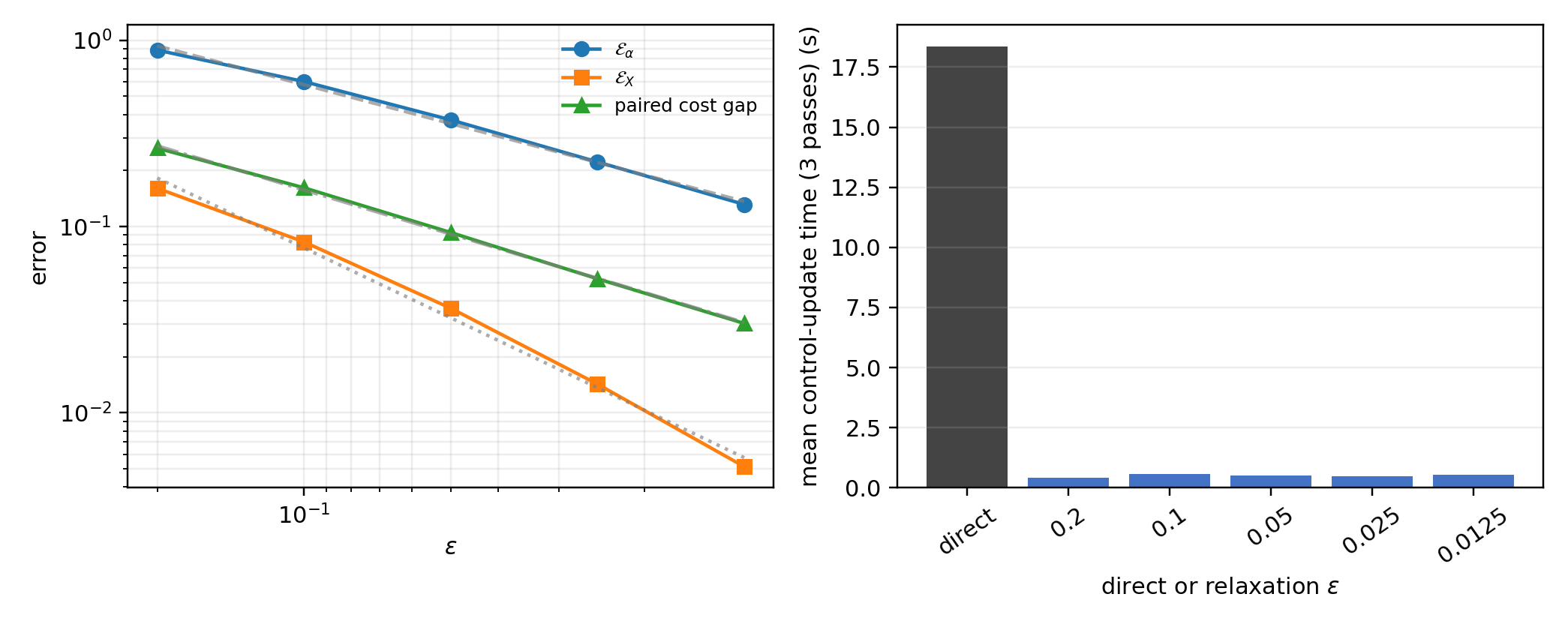}
 \caption{Strong control and state errors, the paired cost gap, and the average time
spent on the control update over three complete simulation runs. The dotted lines
represent least-squares power-law fits, with empirical exponents \(0.69\) for
\(\mathcal E_\alpha(\eps)\), \(1.25\) for \(\mathcal E_X(\eps)\), and \(0.79\)
for the paired cost gap.}
 \label{fig:numerical-mf-errors}
\end{figure}

Both strong errors and the paired cost gap decrease monotonically as
\(\eps\downarrow0\). On the fixed grid used here, averaging three complete simulation runs, the control update requires \(18.35\) seconds for
the direct method, whereas the two-time-scale method requires between
\(0.43\) and \(0.58\) seconds. This corresponds to a speedup factor between
\(32\) and \(43\). The saving results from replacing the repeated Newton
iterations for the implicit optimality condition with the semi-implicit
residual update in \eqref{eq:numerical-mf-semi-implicit}. Overall, the
numerical results show that, for the given grid and implementation, the
two-time-scale method substantially reduces the computational cost while
its approximation errors decrease as \(\eps\) becomes smaller.

\end{document}